\documentclass[reqno]{amsart}

\usepackage[english]{babel}
\usepackage{amssymb,amsmath,hyperref}
\usepackage{bbold}
\usepackage{amsrefs}  
\usepackage{stackrel}
\usepackage[foot]{amsaddr}
\usepackage{mathrsfs}
\usepackage{cleveref, enumerate}
\usepackage{comment}

\makeatletter
\def\@settitle{\begin{center}%
  \baselineskip14\p@\relax
  \bfseries
  \uppercasenonmath\@title
  \@title
  \ifx\@subtitle\@empty\else
     \\[1ex]\uppercasenonmath\@subtitle
     \footnotesize\mdseries\@subtitle
  \fi
  \end{center}%
}
\def\subtitle#1{\gdef\@subtitle{#1}}
\def\@subtitle{}
\makeatother

\usepackage{tikz}
\usetikzlibrary[positioning]
\usetikzlibrary{patterns}
\usetikzlibrary{shapes,backgrounds,arrows, positioning,matrix}

\usetikzlibrary{calc,trees,positioning,arrows,chains,shapes.geometric,%
    decorations.pathreplacing,decorations.pathmorphing,shapes,%
    matrix,shapes.symbols}

\tikzset{
>=stealth',
  punktchain/.style={
    rectangle, 
    rounded corners, 
    draw=black, very thick,
    text width=10em, 
    minimum height=3em, 
    text centered, 
    on chain},
  line/.style={draw, thick, <-},
  element/.style={
    tape,
    top color=white,
    bottom color=blue!50!black!60!,
    minimum width=8em,
    draw=blue!40!black!90, very thick,
    text width=10em, 
    minimum height=3.5em, 
    text centered, 
    on chain},
  every join/.style={->, thick,shorten >=1pt},
  decoration={brace},
  tuborg/.style={decorate},
  tubnode/.style={midway, right=2pt},
}

\tikzset{
    >=stealth',
    punkt/.style={
           rectangle,
           rounded corners,
           draw=black, very thick,
           text width=6.5em,
           minimum height=2em,
           text centered},
    pil/.style={
           ->,
           thick,
           shorten <=2pt,
           shorten >=2pt,}
}

\usetikzlibrary{calc}
\usetikzlibrary{shapes}

\makeatletter
\newdimen\@myBoxHeight%
\newdimen\@myBoxDepth%
\newdimen\@myBoxWidth%
\newdimen\@myBoxSize%
\newcommand{\SquareBox}[2][]{%
    \settoheight{\@myBoxHeight}{#2}
    \settodepth{\@myBoxDepth}{#2}
    \settowidth{\@myBoxWidth}{#2}
    \pgfmathsetlength{\@myBoxSize}{max(\@myBoxWidth,(\@myBoxHeight+\@myBoxDepth))}%
    \tikz \node [shape=rectangle, shape aspect=1,draw=black,inner sep=2\pgflinewidth, minimum size=\@myBoxSize,#1] {#2};%
}%
\makeatother

\newtheorem{thm}{Theorem}

\newtheorem{cor}[thm]{Corollary}
\newtheorem{defi}[thm]{Definition}
\newtheorem{rem}[thm]{Remark}
\newtheorem{nota}[thm]{Notation}
\newtheorem{exa}[thm]{Example}

\newtheorem{princ}[thm]{Principle}

\newtheorem{ack}[thm]{Acknowledgement}

\newtheorem{obs}[thm]{Observation}

\newtheorem{tempo}[thm]{Template}
\newcommand\be{\begin{equation}}
\newcommand\ee{\end{equation}} 
\def\bdefi{\begin{defi}\rm}
\def\edefi{\end{defi}}
\def\bnota{\begin{nota}\rm}
\def\enota{\end{nota}}
\def\brem{\begin{rem}\rm}
\def\erem{\end{rem}}

\newbox\gnBoxA  
\newdimen\gnCornerHgt
\setbox\gnBoxA=\hbox{$\ulcorner$}
\global\gnCornerHgt=\ht\gnBoxA
\newdimen\gnArgHgt
\def\Godelnum #1{%
	\setbox\gnBoxA=\hbox{$#1$}%
	\gnArgHgt=\ht\gnBoxA%
	\ifnum \gnArgHgt<\gnCornerHgt
		\gnArgHgt=0pt%
	\else
		\advance \gnArgHgt by -\gnCornerHgt%
	\fi
	\raise\gnArgHgt\hbox{$\ulcorner$} \box\gnBoxA %
		\raise\gnArgHgt\hbox{$\urcorner$}}
 
\newcommand{\Sh}{\ensuremath{\protect{S_{\st{}}}}}

\newcommand{\forallst}{\forall^{\st{}}}
\newcommand{\existsst}{\exists^{\st{}}}
\newcommand{\tup}{\underline} 
\def\FIVE{\Pi_{1}^{1}\text{-\textsf{CA}}_{0}}

\def\NCR{\textup{\textsf{NCR}}}

\def\pw{\textup{\textsf{pw}}}
\def\ATR{\textup{\textsf{ATR}}}
\def\LEM{\textup{\textsf{LEM}}}
\def\IST{\textup{\textsf{IST}}}

\def\HIP{\textup{\textsf{HIP}}}

\def\STP{\textup{\textsf{STP}}}

\def\intern{\textup{\textsf{int}}}

\def\DNR{\textup{DNR}}

\def\H{\textup{\textsf{H}}}
\def\RCA{\textup{\textsf{RCA}}}

\def\RCAo{\textup{\textsf{RCA}}_{0}^{\omega}}
\def\RCAO{\textup{\textsf{RCA}}_{0}^{\Omega}}

\def\ef{\textup{\textsf{ef}}}
\def\ns{\textup{\textsf{ns}}}

\def\WKL{\textup{\textsf{WKL}}}

\def\IVT{\textup{IVT}}
\def\IVT{\textup{\textsf{IVT}}}

\def\T{\textsf{\textup{T}}}
\def\TT{{\mathcal{T}}}

\def\bye{\end{document}}

\def\P{\textup{\textsf{P}}}

\def\N{{\mathbb  N}}
\def\Q{{\mathbb  Q}}
\def\R{{\mathbb  R}}

\def\I{{\textsf{\textup{I}}}}

\def\MUC{\textup{\textsf{MUC}}}

\def\MCT{\textup{\textsf{MCT}}}

\def\R{{\mathbb{R}}}
\def\({\textup{(}}
\def\){\textup{)}}

\def\st{\textup{st}}
\def\nst{\textup{nst}}
\def\asa{\leftrightarrow}

\def\di{\rightarrow}

\def\eps{\varepsilon}
\def\epsa{\varepsilon_{\alpha}}

\def\M{\mathcal{M}}

\def\ACA{\textup{\textsf{ACA}}}
\def\paai{\Pi_{1}^{0}\textup{-\textsf{TRANS}}}
\def\Paai{\Pi_{1}^{1}\textup{-\textsf{TRANS}}}

\def\QFAC{\textup{\textsf{QF-AC}}}

\def\HBU{\textup{\textsf{HBU}}}
\def\DNR{\textup{\textsf{DNR}}}
\def\NUC{\textup{\textsf{NUC}}}
\def\KOE{\textup{\textsf{KOE}}}

\def\ECF{\textup{\textsf{ECF}}}

\def\SCF{\textup{\textsf{SCF}}}
\def\TJ{\textup{\textsf{TJ}}}
\def\ZFC{\textup{\textsf{ZFC}}}
\def\ZF{\textup{\textsf{ZF}}}
\def\DIF{\textup{\textsf{DIF}}}
\def\NSD{\textup{\textsf{NSD}}}

\def\RIE{\textup{\textsf{RIE}}}
\def\REI{\textup{\textsf{RIE}}}

\def\MU{\textup{\textsf{MU}}}

\def\HAC{\textup{\textsf{HAC}}}

\def\INT{\textup{\textsf{int}}}

\numberwithin{equation}{section}
\numberwithin{thm}{section}

\begin{document}
\title{To be or not to be constructive}
\subtitle{That is not the question}

\author{Sam Sanders}
\address{Munich Center for Mathematical Philosophy, LMU Munich, Germany \& Department of Mathematics, Ghent University}

\email{sasander@me.com} 
\begin{abstract}
In the early twentieth century, L.E.J.\ Brouwer pioneered a new philosophy of mathematics, called \emph{intuitionism}.  
Intuitionism was revolutionary in many respects but stands out -mathematically speaking- for its challenge of 
\emph{Hilbert's formalist philosophy} of mathematics and rejection of the \emph{law of excluded middle} from the `classical' logic used in mainstream mathematics.  
Out of intuitionism grew \emph{intuitionistic logic} and the associated \emph{Brouwer-Heyting-Kolmogorov} 
interpretation by which `there exists $x$'  intuitively means `an algorithm to compute $x$ is given'.    
A number of schools of constructive mathematics were developed, inspired by Brouwer's intuitionism and invariably based on intuitionistic logic, but with varying interpretations
of what constitutes an algorithm.  This paper deals with the \emph{dichotomy} between constructive and 
non-constructive mathematics, or rather the \emph{absence} of such an `excluded middle'. 
In particular, we challenge the `binary' view that mathematics is either constructive or not.   
To this end, we identify a part of \emph{classical} mathematics, namely \emph{classical Nonstandard Analysis}, and show it inhabits the twilight-zone between the constructive and non-constructive.  
Intuitively, the predicate `$x$ is standard' typical of Nonstandard Analysis can be interpreted as `$x$ is computable', giving rise to computable (and sometimes constructive) mathematics obtained directly from \emph{classical} Nonstandard Analysis.  
Our results formalise Osswald's longstanding conjecture that classical Nonstandard Analysis is \emph{locally constructive}.      
Finally, an alternative explanation of our results is provided by Brouwer's thesis that \emph{logic depends upon mathematics}.    
\end{abstract}


\maketitle
\thispagestyle{empty}


\vspace{-1cm}

\section{Introduction}\label{intro}
This volume is dedicated to the founder of \emph{intuitionism}, L.E.J.\ Brouwer, who pursued this revolutionary programme with great passion and against his time's received view of mathematics and its foundations (\cites{brouw,brouw2,eagle, gesticht,brouw22}).  We therefore find it fitting that our paper attempts to subvert (part of) \emph{our time's} received view of mathematics and its foundations.  As suggested by the title, we wish to challenge the binary distinction \emph{constructive\footnote{The noun `constructive' is often used as a synonym for `effective', while it refers to the foundational framework \emph{constructive mathematics} in logic and the foundations of mathematics (\cite{troeleke1}).  Context determines the meaning of `constructive' in this paper (usually the latter).} versus non-constructive} mathematics.  We shall assume basic familiarity with constructive mathematics and \emph{intuitionistic logic} with its \emph{Brouwer-Heyting-Kolmogorov} 
interpretation.    

\medskip

Surprising as this may be to the outsider, the quest for the (ultimate) foundations of mathematics was and is an \emph{ongoing and often highly emotional affair}.  
The \emph{Grundlagenstreit} between Hilbert and Brouwer is perhaps the textbook example (See e.g.\ \cite{gesticht}*{II.13}) of a fierce struggle between competing views on the foundations of mathematics, namely Hilbert's formalism and Brouwer's intuitionism.  Einstein was apparently disturbed by this controversy and exclaimed the following:
\begin{quote}
What is this frog and mouse battle among the mathematicians? (\cite{square}*{p.\ 133})
\end{quote}
More recently, Bishop mercilessly attacked \emph{Nonstandard Analysis} in his review \cite{kuddd} of Keisler's monograph \cite{keisler3}, even going as far as debasing Nonstandard Analysis\footnote{The naked noun \emph{Nonstandard Analysis} will always implicitly include the adjective \emph{classical}, i.e.\ based on classical logic.  We shall not directly deal with \emph{constructive Nonstandard Analysis}, i.e.\ based on intuitionistic logic, but do discuss its relationship with our results in Section \ref{palmke}.} to a \emph{debasement of meaning} in \cite{kluut}.  Bishop, as Brouwer, believed that to state the existence of an object, one has to provide a construction for it, while Nonstandard Analysis cheerfully includes ideal/non-constructive objects \emph{at the fundamental level}, the textbook example being infinitesimals.  Note that Brouwer's student, the intuitionist Arend Heyting, had a higher opinion of Nonstandard Analysis (\cite{heyting}).            

\medskip

A lot of ink has been spilt over the aforementioned struggles, and we do not wish to add to that literature.  By contrast, the previous paragraph is merely meant to establish the well-known juxtaposition of \emph{classical/mainstream/non-constructive} versus \emph{constructive} mathematics\footnote{A number of approaches to constructive mathematics exist (\cite{beeson1}*{III}, \cite{troeleke1}*{I.4}, \cite{brich}), and both Brouwer's \emph{intuitionism} and Bishop's \emph{Constructive Analysis} (\cite{bish1}) represent a school therein.}.  
The following quote by Bishop emphasises this `two poles' view for the specific case of Nonstandard Analysis, which Bishop believed to be the worst exponent of classical mathematics.  
\begin{quote}
[Constructive mathematics and Nonstandard Analysis] are at opposite poles.
Constructivism is an attempt to deepen the meaning of mathematics; non-standard analysis, an attempt to dilute it further. (\cite{bishl}*{p.\ 1-2})
\end{quote}    
To be absolutely clear, lest we be misunderstood, we only wish to point out the current state-of-affairs in contemporary mathematics:  On one hand, there is mainstream mathematics with its classical logic and other fundamentally non-constructive features, of which Nonstandard Analysis is the \emph{nec plus ultra} according to some; on the other hand, there is constructive mathematics with its intuitionistic logic and computational-content-by-design.  In short, there are two opposing camps (classical and constructive) in mathematics separated by a no-man's land, with the occasional volley exchanged as in e.g.\ \cites{boerken, frogman, trots,trots2}.  

\medskip 

Stimulated by Brouwer's revolutionary spirit, our goal is to subvert the above received view.  To this end, we will identify a field of classical mathematics 
which occupies the twilight-zone between constructive and classical mathematics.  
Perhaps ironically, this very field is \emph{Nonstandard Analysis}, so vilified by Bishop for its alleged \emph{fundamentally} non-constructive nature.  
We introduce a well-known axiomatic\footnote{Hrbacek (\cite{hrbacek2}) and Nelson (\cite{wownelly}) independently introduced an axiomatic approach to Nonstandard Analysis in the mid-seventies (See \cite{reeken} for more details).} approach to Nonstandard Analysis, Nelson's \emph{internal set theory}, in Section \ref{NSA}.  We discuss past claims regarding the constructive nature of Nonstandard Analysis in Section \ref{praxis}, in particular Osswald's notion of \emph{local constructivity}.    

\medskip

In Section \ref{main}, we establish that Nonstandard Analysis occupies the twilight-zone between constructive and classical mathematics.  As part of this endeavour, we show that a major part of \emph{classical} Nonstandard Analysis has computational content, much like constructive mathematics itself.  
Intuitively, the predicate `$x$ is standard' unique to Nonstandard Analysis can be interpreted as `$x$ is computable', giving rise to computable (and even constructive) mathematics obtained directly from \emph{classical} Nonstandard Analysis.  
Our results formalise Osswald's longstanding conjecture from Section \ref{loco} that classical Nonstandard Analysis is \emph{locally constructive}.      
By way of reversal, we establish in Section \ref{thereandback} that certain \emph{highly constructive} results extracted from Nonstandard Analysis in turn imply the nonstandard theorems from which they were obtained.  Advanced results regarding the constructive nature of Nonstandard Analysis may be found in Section \ref{grab}.     

\medskip

In Section \ref{titatovenaar}, we shall offer an alternative interpretation of our results based on Brouwer's view that \emph{logic is dependent on mathematics}.  
Indeed, Brouwer already explicitly stated in his dissertation that \emph{logic depends upon mathematics} as follows: 
\begin{quote}
While thus mathematics is independent of logic, logic does depend upon mathematics: in the first place \emph{intuitive logical reasoning} is that special kind of mathematical reasoning which remains if, considering mathematical structures, one restricts oneself to relations of whole and part; 
(\cite{brouwt}*{p.\ 127}; emphasis in Dutch original)
\end{quote}
This leads us to the following alternative interpretation: If one fundamentally changes mathematics, as one arguably does when introducing Nonstandard Analysis, it stands to reason that the associated logic will change along as the latter depends on the former, in Brouwer's view.
By way of an example, we shall observe in Section \ref{fraki} that when introducing the notion of `being standard' fundamental to Nonstandard Analysis, the law of excluded middle of classical mathematics moves from `the original sin of non-constructivity' to a \emph{computationally inert} principle which does not have any real non-constructive consequences anymore.  

\medskip

Thus, our results vindicate Brouwer's thesis that logic depends upon mathematics by showing that classical logic becomes `much more constructive' when shifting from `usual' mathematics to Nonstandard Analysis, as discussed in Section \ref{kiko}.  As also discussed in the latter, Nelson and Robinson have indeed claimed that Nonstandard Analysis constitutes a new kind of mathematics and logic.  We also discuss the connection between Nonstandard Analysis and Troelstra's views of intuitionism, and Brouwer's \emph{first act} of intuitionism.  

\medskip

Finally, we point out a veritable \emph{Catch22} connected to this paper.  On one hand, if we were to emphasise that large parts of Nonstandard Analysis give rise to computable mathematics (and vice versa), then certain readers will be inclined to dismiss Nonstandard Analysis as `nothing new', like Halmos was wont to:
\begin{quote}
it's a special tool, too special, and other tools can do everything
it does. It's all a matter of taste.  (\cite{halal}*{p.\ 204})
\end{quote}
On the other hand, if we were to emphasise the new mathematical objects (no longer involving Nonstandard Analysis) with strange properties one can obtain from Nonstandard Analysis, then certain readers will be inclined to dismiss Nonstandard Analysis as `weird' or `fundamentally different' from (mainstream) mathematics, as suggested by Bishop and Connes (See Section \ref{crackp}).  
To solve this conundrum, we shall walk a tightrope between new and known results, and hope all readers agree that neither of the aforementioned two views is correct. 

\section{An introduction to Nonstandard Analysis}\label{NSA}
We provide an informal introduction to Nonstandard Analysis in Section \ref{intronsa}.  Furthermore, we discuss Nelson's \emph{internal set theory}, 
an axiomatic approach to Nonstandard Analysis, in detail in Section \ref{nellyke}.  Fragments of internal set theory based on \emph{Peano} and \emph{Heyting arithmetic} are studied in Section \ref{frag}.  
\subsection{An informal introduction to Nonstandard Analysis}\label{intronsa}
In a nutshell, \emph{Nonstandard Analysis} is the best-known\footnote{We briefly discuss alternative formalisations of the infinitesimal calculus in Section \ref{palmke}.} formalisation of the intuitive \emph{infinitesimal calculus}.  The latter is used to date in large parts of physics and engineering, 
and was used historically in mathematics by e.g.\ Archimedes, Leibniz, Euler, and Newton until the advent of the Weierstra\ss~`epsilon-delta' framework, as detailed in \cite{nieteerlijk}.  

\medskip

Robinson first formulated the \emph{semantic} approach to Nonstandard Analysis (NSA for short) around 1960 using \emph{nonstandard models}, which are nowadays built using \emph{free ultrafilters} (\cite{robinson1, nsawork2}).  Hrbacek (\cite{hrbacek2}) and Nelson (\cite{wownelly}) independently introduced an axiomatic approach to NSA in the mid-seventies (See \cite{reeken}).      
We will work in axiomatic systems that can be viewed as subsystems of either approach. We will use the terminology and notation of Nelson's approach.
In all cases, there is a universe of \emph{standard} (usual/everyday) mathematical objects and a universe of \emph{nonstandard} (new/other-worldly) mathematical objects, including infinitesimals.  
Tao has formulated a \emph{pragmatic} view of NSA in \cite{taote,tao2} which summarises as: 
\begin{center}
\textbf{NSA is another tool for mathematics}.  
\end{center}
A sketch of Tao's view is as follows: given a \emph{hard} problem in (usual/standard) mathematics, 
pushing this problem into the nonstandard universe of NSA generally converts it to a \emph{different and easier} problem, a basic example being `epsilon-delta' definitions.  A \emph{nonstandard} solution can then be formulated and pushed back into the universe of standard/usual mathematics.  

\medskip

This `pushing back-and-forth' between the standard and nonstandard universe is enabled by \emph{Transfer} and \emph{Standard Part} (aka \emph{Standardisation}).   Robinson established the latter as properties of his nonstandard models, while Nelson adopts them as axioms of $\IST$.  
The following figure illustrates Tao's view and the practice of NSA, as also discussed in Section \ref{keiskers} to \ref{loco}.        
Note that \emph{Transfer} is also used sometimes in tandem with \emph{Standard Part}, and vice versa.  
~\\~\\~\\~\\
\begin{figure}[h!]
\begin{tikzpicture}
\begin{pgflowlevelscope}{\pgftransformscale{0.9}}
  \matrix (m) [matrix of math nodes,row sep=1em,column sep=1.5em,minimum width=0.75em]
  {
~ &  {\color{black}  \text{\emph{Transfer}}} &~  \\
   \text{standard universe/usual mathematics} &~& \color{black}\text{nonstandard universe} \\
      \text{\textbf{hard problem}} &~& \color{black}\text{\textbf{easy problem} with nonstandard solution} \\
            \text{\emph{epsilon-delta} definitions} &~& \color{black}\text{nonstandard definitions involving \emph{infinitesimals}} \\
  & \text{\color{black}\emph{Standard Part}}  &~  \\
    };
 \path[-stealth]
     (m-1-2)     edge [-> ] (m-2-3) 
     (m-1-2)     edge [<-] (m-2-1)
      (m-4-3)     edge [->] (m-5-2)
      (m-4-1)     edge [<-] (m-5-2);
\end{pgflowlevelscope}
\end{tikzpicture}
\vspace{1,4cm}
\caption{The practice of NSA}
\label{figsaregross}
\end{figure}
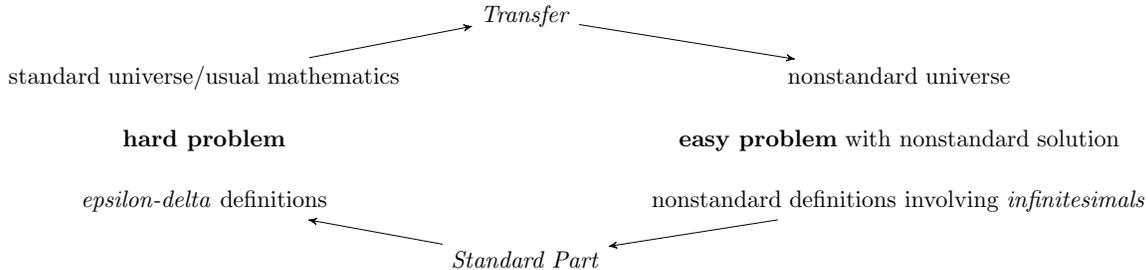~\\
We shall study a number of basic examples of the scheme from Figure \ref{figsaregross} in Section~\ref{main} for notions as continuity and Riemann integration.  
Furthermore, the following table lists some correspondences between the standard and nonstandard universe:  When pushed from the standard to the nonstandard universe, an object on the left corresponds to (or is included in) an object on the right side of the following table.\\
\begin{center}
\begin{tabular}{|c|c|}
\hline
\textbf{standard universe}  & \textbf{nonstandard universe} \\
  \hline \hline
continuous object & discrete object \\
\hline
infinite set & finite set\\
\hline
epsilon-delta definition & universal definition using infinitesimals \\
\hline
quantitative result & qualitative result\\\hline
\end{tabular}
\end{center}~\\
Finally, both \emph{Transfer} and \emph{Standard Part} are \emph{highly non-constructive} in each approach to NSA, i.e.\ these principles imply the existence of \emph{non-computable objects}.  
By contrast, the mathematics in the nonstandard universe is usually very basic, involving \emph{little more than discrete sums and products} of nonstandard length.  Various authors have made similar observations, which we discuss in Sections \ref{keiskers} to \ref{loco}.  
In particular, Osswald has characterised this observation as \textbf{NSA is locally constructive} or the \textbf{local constructivity of NSA} as discussed in Section \ref{loco}, and this view will be vindicated by our results in Section \ref{main}.

\subsection{An axiomatic approach to Nonstandard Analysis}\label{nellyke}
We shall introduce Nelson's \emph{internal set theory} $\IST$ in this section.  
\subsubsection{Internal Set Theory}
In Nelson's \emph{syntactic} (or `axiomatic') approach to Nonstandard Analysis (\cite{wownelly}), a new predicate `st($x$)', read as `$x$ is standard' is added to the language of \textsf{ZFC}, the usual foundation of mathematics\footnote{The acronym $\ZFC$ stands for \emph{Zermelo-Fraenkel set theory with the axiom of choice}.}.    
The notations $(\forall^{\st}x)$ and $(\exists^{\st}y)$ are short for $(\forall x)(\st(x)\di \dots)$ and $(\exists y)(\st(y)\wedge \dots)$.  A formula is called \emph{internal} if it does not involve `st', and \emph{external} otherwise.   

\medskip

The external axioms \emph{Idealisation}, \emph{Standardisation}, and \emph{Transfer} govern the new predicate `st';  we state these axioms\footnote{The superscript `fin' in \textsf{(I)} means that $x$ is finite, i.e.\ its number of elements is bounded by a natural number.} as follows:  
\begin{enumerate}
\item[\textsf{(I)}] $(\forall^{\st~\textup{fin}}x)(\exists y)(\forall z\in x)\varphi(z,y)\di (\exists y)(\forall^{\st}x)\varphi(x,y)$, for internal $\varphi$ with any (possibly nonstandard) parameters.  
\item[\textsf{(S)}] $(\forall^{\st} x)(\exists^{\st}y)(\forall^{\st}z)\big((z\in x\wedge \varphi(z))\asa z\in y\big)$, for any $\varphi$.
\item[\textsf{(T)}] $(\forall^{\st}t)\big[(\forall^{\st}x)\varphi(x, t)\di (\forall x)\varphi(x, t)\big]$, where $\varphi(x,t)$ is internal, and only has free variables $t, x$.  
\end{enumerate}
The system \textsf{IST} is the \emph{internal} system \textsf{ZFC} extended with the aforementioned \emph{external} axioms;  
Internal set theory $\IST$ is a conservative extension of \textsf{ZFC} for the internal language (\cite{wownelly}*{\S8}), i.e.\ these systems 
prove the same \emph{internal} sentences.  Note that Nelson's $\IST$ inherits non-constructive objects from $\ZFC$, like the \textbf{standard} Turing jump\footnote{That the Turing jump is standard follows from applying \emph{Transfer} to the internal statement `there is a set of natural numbers which solves the Halting problem' formalised inside $\ZFC$.}, and axioms like the \textbf{internal} \emph{law of excluded middle} $A\vee \neg A$.    

\medskip

For those familiar with Robinson's approach to Nonstandard Analysis, 
we point out one fundamental difference with Nelson's view of internal set theory:  In the former one studies \emph{extensions} of structures provided by nonstandard models, while in the latter one studies the original structures, in which certain objects happen to be standard, as identified by `st'.  Nelson formulates this view\footnote{As discussed in \cite{samBIG}, one does \emph{not} have to share Nelson's view to study $\IST$: the latter is a logical system while the former is a possible philosophical point of view regarding $\IST$.  However, like Nelson, we will (mostly) view the mathematics and this point of view as one package.} as follows:
\begin{quote}
\emph{All theorems of conventional mathematics remain valid}. No change in terminology is required. What is new in internal set theory is only an addition, not
a change. We choose to call certain sets standard (and we recall that in $\ZFC$ every mathematical object-a real number, a function, etc.-is a set), but the theorems of conventional mathematics apply to all sets, nonstandard as well as standard.  (\cite{wownelly}*{p.\ 1165}; emphasis in original)
\end{quote}
\begin{quote}
\emph{Every specific object of conventional mathematics is a standard set}. It remains unchanged in the new theory. For example, in internal
set theory there is only one real number system, the system $\R$ with which we are already familiar.  (\cite{wownelly}*{p.\ 1166}; emphasis in original)
\end{quote}
By way of an example, $\IST$ just inherits $\R$ from $\ZFC$ and calls some real numbers standard and some nonstandard.  All the `specific' numbers like $0, 1, \pi, e, \dots$ are standard, while there are 
also \emph{nonstandard} ones (in $\IST$).  By contrast, Robinson's approach starts with $\R$ and \emph{extends} this set to the much bigger $^{*}\R$ which is such that $\R\subset {^{*}\R}$ and which satisfies the axioms of the real numbers too; the objects in $^{*}\R\setminus \R$ are called `nonstandard'.    Thus, both approaches have a \emph{different} universe (or `world') of standard objects, which however only matters for advanced applications of NSA (See Section \ref{loebM}).  
In each approach, the intuitive picture as in Figure \ref{figsaregross} is valid, bearing in mind the right interpretation of the standard world.   

\medskip

It goes without saying that 
the step from $\ZFC$ to $\IST$ can be done for a large spectrum of logical systems weaker than $\ZFC$.  In Section \ref{frag}, we study this extension of the usual classical and constructive axiomatisation of 
arithmetic, called \emph{Peano Arithmetic} and \emph{Heyting Arithmetic}.  
In the next sections, we discuss the intuitive meaning of the external axioms.  
\subsubsection{Intuitive meaning of Idealisation}
First, the \emph{contraposition} of \textsf{I} implies that
\be\label{criv}
(\forall y)(\exists^{\st}x)\varphi(x,y) \di (\exists^{\st~\textup{fin}}x)\underline{(\forall y)(\exists z\in x)\varphi(z,y)},
\ee
for all internal $\varphi$, and where the underlined part is thus also \emph{internal}.  Hence, intuitively speaking, \emph{Idealisation} allows us to `pull a standard quantifiers like $(\exists^{\st}x)$ in \eqref{criv} through a normal quantifier $(\forall y)$'.  We will also refer to \eqref{criv} as \emph{Idealisation}, as the latter axiom is usually used in this contraposed form.    

\medskip

Note that the axioms $B\Sigma_{n}$ of Peano arithmetic (\cite{buss}*{II}) play a similar role:  The former allow one to `pull an unbounded number quantifier through a bounded number quantifier'.  In each case, one obtains a formula in a kind of `normal form' with a block of certain quantifiers (resp.\ external/unbounded) up front followed by another block of different quantifiers (resp.\ internal/bounded).  Example \ref{leffe} involving nonstandard continuity in $\IST$ is illustrative (See \cite{wownelly}*{\S5} for more examples).   
\begin{exa}\label{leffe}\rm
We say that $f$ is \emph{nonstandard continuous} on the set $X\subseteq \R$ if
\be\label{soareyouke}
(\forall^{\st}x\in X)(\forall y\in X)(x\approx y \di f(x)\approx f(y)),  
\ee
where $z\approx w$ is $(\forall^{\st} n\in \N)(|z-w|<\frac{1}{n})$.  Resolving `$\approx$' in \eqref{soareyouke}, we obtain  
\[\textstyle
(\forall^{\st}x\in X)(\forall y\in X)\big( (\forall^{\st}N\in\N) (|x- y|<\frac{1}{N}) \di (\forall^{\st}k \in\N)(|f(x)- f(y)|<\frac{1}{k})\big).  
\]
We may bring out the `$(\forall^{\st}k\in \N)$' and `$(\forall^{\st}N\in\N) $' quantifiers as follows:
\[\textstyle
(\forall^{\st}x\in X)(\forall^{\st}k\in \N)\underline{(\forall y\in X)(\exists^{\st}N\in\N)\big( |x- y|<\frac{1}{N} \di|f(x)- f(y)|<\frac{1}{k}\big)}.  
\]
Applying \textsf{I} as in \eqref{criv} to the underlined formula, we obtain a finite and standard set $z\subset \N$ such that $(\forall y\in X)(\exists N\in z)$ in the previous formula.  
Now let $N_{0}$ be the maximum of all numbers in $z$, and note that for $N=N_{0}$
\[\textstyle
(\forall^{\st}x\in X)(\forall^{\st}k\in \N)(\exists^{\st}N\in\N)(\forall y\in X)\big( |x- y|<\frac{1}{N} \di|f(x)- f(y)|<\frac{1}{k}\big).  
\]
The previous formula has all standard quantifiers up front and is very close to the `epsilon-delta' definition of continuity from mainstream mathematics.  
Hence, we observe the role of \textsf{I}: to connect the worlds of nonstandard mathematics (as in \eqref{soareyouke}) and mainstream mathematics.    
\end{exa}
\subsubsection{Intuitive meaning of Transfer}
The axiom \emph{Transfer} expresses that certain statements about \emph{standard} objects are also true for \emph{all} objects.    This property is essential in proving the equivalence between epsilon-delta statements and their nonstandard formulation.  The following example involving continuity is illustrative.
\begin{exa}\label{huntress}\rm
Recall Example \ref{leffe}, the definition of nonstandard continuity \eqref{soareyouke} and the final equation in particular.  Dropping the `st' for $N$ in the latter:
\[\textstyle
(\forall^{\st}x\in X)(\forall^{\st}k\in \N)(\exists N\in\N)(\forall y\in X)\big( |x- y|<\frac{1}{N} \di|f(x)- f(y)|<\frac{1}{k}\big).  
\]
Assuming $X$ and $f$ to be standard, we can apply \textsf{T} to the previous to obtain 
\be\textstyle\label{kikj}
(\forall x\in X)(\forall k\in \N)(\exists N\in\N)(\underline{\forall y\in X) |x- y|<\frac{1}{N} \di|f(x)- f(y)|<\frac{1}{k}}
\ee
Note that \eqref{kikj} is just the \emph{usual epsilon-delta definition of continuity}.  In turn, to prove that \eqref{kikj} implies nonstandard continuity as in \eqref{soareyouke}, fix \emph{standard} $X, f, k$ in \eqref{kikj} and apply the contraposition of \textsf{T} to `$(\exists N\in \N)\varphi(N)$' where $\varphi$ is the underlined formula in \eqref{kikj}.  
The resulting formula $(\exists^{\st} N\in \N)\varphi(N)$ immediately implies nonstandard continuity as in \eqref{soareyouke}.  
\end{exa}
By the previous example, nonstandard continuity \eqref{soareyouke} and epsilon-delta continuity \eqref{kikj} are equivalent for standard functions in $\IST$.  However, the former involves far less quantifier alternations and is close to the intuitive understanding of continuity as `no jumps in the graph of the function'. 
Hence, we observe the role of \textsf{T}: to connect the worlds of nonstandard mathematics (as in \eqref{soareyouke}) and mainstream mathematics (as in \eqref{kikj}), and we also have our first example of 
the practice of NSA from Figure \ref{figsaregross} involving nonstandard and epsilon-delta definitions.  
\subsubsection{Intuitive meaning of Standardisation}
The axiom \emph{Standardisation} (also called \emph{Standard Part}) is useful as follows: It is in general easy to build \emph{nonstandard and approximate} solutions to mathematical problems in $\IST$, but a \emph{standard} solution is needed as the latter also exists in `normal' mathematics (as it is suitable for \emph{Transfer}).   Intuitively, the axiom \textsf{S} tells us that from a \emph{nonstandard approximate solution}, we can \emph{always} find a \emph{standard one}.  Since we may (only) apply \emph{Transfer} to formulas involving the latter, we can then also prove the latter is an object of normal mathematics.  The following example is highly illustrative. 
\begin{exa}\rm
The \emph{intermediate value theorem} states that for every continuous function $f:[0,1]\di [0,1]$ such that $f(0) f(1)< 0$, there is $x\in [0,1]$ such that $f(x)=0$.  Assuming $f$ is standard, it is easy\footnote{Since $[0,1]$ is compact, we may assume that $f$ is uniformly continuous there.  Similar to Example \ref{huntress}, we may assume $f$ is nonstandard uniformly continuous as in $(\forall x, y\in [0,1])[x\approx y \di f(x)\approx f(y)]$.  Let $N$ be a nonstandard natural number and let $j\leq N$ be the least number such that $f(\frac{j}{N})f(\frac{j+1}{N})\leq 0$.  Then $f(j/N)\approx 0$ by nonstandard uniform continuity, and we are done. \label{Exake}} to find a nonstandard real $y$ in the unit interval such that $f(y)\approx 0$, i.e.\ $y$ is an intermediate value `up to infinitesimals'.  
The axiom \textsf{S} then tells\footnote{Let $y\in [0,1]$ be such that $f(y)\approx 0$ and consider the set of rationals $z=\{ q_{1}, q_{1}, q_{2}, \dots, q_{N}\}$ where $q_{i}$ is a rational such that $|y-q_{i}|<\frac{1}{i}$ and $q_{i}=\frac{j}{2^{i}}$ for some $j\leq i$, and $N$ is a nonstandard number.  Applying \textsf{S}, there is a standard set $w$ such that $(\forall^{\st}q)(q\in w\asa q\in z)$.  The \emph{standard} sequence provided by $w$ thus converges to a standard real $x\approx y$.} us that there is a \emph{standard} real $x$ such that $x\approx y$, and by the nonstandard continuity of $f$ (See previous example), we have $f(x)\approx 0$.  Now apply \textsf{T} to the latter\footnote{Recall that `$f(x)\approx 0$' is short for $(\forall^{\st}k\in \N)(|f(x)|<\frac{1}{k})$.} to obtain $f(x)=0$.  Hence, we have obtain the (internal) intermediate value theorem for standard functions, and \textsf{T} yields the full theorem.            
\end{exa}
Hence, we observe the role of \textsf{S}: to connect the worlds of nonstandard mathematics and mainstream (standard) mathematics by providing standard objects `close to' nonstandard ones.    
We shall observe another natural interpretation of \textsf{S} in Section \ref{metastable} in terms of Tao's notion \emph{metastability}.  

\medskip

In conclusion, the external axioms of $\IST$ provide a connection between nonstandard and mainstream mathematics:  They allow one to `jump back and forth' between the standard and nonstandard world as sketched in Figure \ref{figsaregross}.  This technique is useful as some problems (like switching limits and integrals) may be easier to solve in the discrete/finite world of nonstandard mathematics than in the continuous/infinite world of standard mathematics (or vice versa).  This observation lies at the heart of Nonstandard Analysis and is a first step towards understanding its power.       
Nelson formulated this observation nicely as follows: 
\begin{quote}
\emph{Part of the power of nonstandard analysis is due to the fact that a complicated internal notion is
frequently equivalent, on the standard sets, to a simple external notion.} (\cite{wownelly}*{p.\ 1169})
\end{quote}

\subsection{Fragments of Internal Set Theory}\label{frag} 
Fragments of $\IST$ have been studied before and we are interested in the systems $\P$ and $\H$ introduced in \cite{brie}.
 
\medskip 
 
In a nutshell, $\P$ and $\H$ are versions of $\IST$ based on the usual classical and intuitionistic axiomatisations of arithmetic, namely \emph{Peano and Heyting arithmetic}.  
We refer to \cite{kohlenbach3} for the exact definitions of our version of Peano and Heyting arithmetic, commonly abbreviated respectively as $\textsf{E-PA}^{\omega}$ and $\textsf{E-HA}^{\omega}$.  
In particular, the systems $\P$ and $\H$ are conservative\footnote{Like for $\ZFC$ and $\IST$, if the system $\P$ (resp.\ $\H$) proves an internal sentence, then this sentence is provable in $\textsf{E-PA}^{\omega}$ (resp.\ \textsf{E-HA}$^{\omega}$).} extensions of \emph{Peano arithmetic} $\textsf{E-PA}^{\omega}$ and \emph{Heyting arithmetic} $\textsf{E-HA}^{\omega}$, as also follows from Theorem \ref{TERM}.  We discuss the systems $\P$ and $\H$ in detail in Sections~\ref{graf1} and \ref{graf2}, and list the axioms in full detail in Section \ref{FULL}.  We discuss why $\P$ and $\H$ are important to our enterprise in Section \ref{fraki}.  Notations and definitions in $\P$ and $\H$ of common mathematical notions like natural and real numbers may be found in Section \ref{leng}.    

\subsubsection{The classical system $\P$}\label{graf1}
We discuss the fragment $\P$ of $\IST$ from \cite{brie}.   Similar to the way $\IST$ is an extension of $\ZFC$,
$\P$ is just the internal system $\textsf{E-PA}^{\omega}$ with the language extended with a new standardness predicate `\st' and with some special cases of the external axioms of $\IST$.  The technical details of this extension may be found in Section \ref{PAPA} while we now provide an intuitive motivation for the external axioms of $\P$, assuming basic familiarity with the finite type system of G\"odel's system $T$, also discussed in Section \ref{TITI}.  

\medskip

First of all, the system $\P$ does not include any fragment of \emph{Transfer}.  The motivation for this omission is as follows:  
The system $\ACA_{0}$ proves the existence of the non-computable \emph{Turing jump} (\cite{simpson2}*{III}) and very weak fragments of \textsf{T} already imply versions of $\ACA_{0}$.
In particular, the following axiom is the \emph{Transfer} axiom limited to universal number quantifiers:
  \be\tag{$\paai$}
(\forall^{\st}f^{1})\big[(\forall^{\st}n^{0})f(n)\ne0\di (\forall m)f(m)\ne0\big].
\ee
As proved in \cite{sambon}*{\S4.1} and
Section \ref{RMKE}, $\paai$ is essentially the nonstandard version of the Turing jump.  
Hence, no \emph{Transfer} is present in $\P$ as this axiom is fundamentally non-constructive, and would result in a non-conservative extension.  

\medskip

Secondly, the system $\P$ involves the full axiom \emph{Idealisation} in the language of $\P$ as follows:  For any internal formula in the language of $\P$:
\be\label{wellhung}
(\forall^{\st} x^{\sigma^{*}})(\exists y^{\tau} )(\forall z^{\sigma}\in x)\varphi(z,y)\di (\exists y^{\tau})(\forall^{\st} x^{\sigma})\varphi(x,y), 
\ee
As it turns out, the axiom \textsf{I} does not yield any `non-constructive' consequences, which also follows from Theorem \ref{TERM} below.  
As for \eqref{criv}, we will also refer to the contraposition of \eqref{wellhung} as \emph{Idealisation}, as the latter is used mostly in the contra-posed form similar to \eqref{criv} (but with finite sequences rather than finite sets).  

\medskip

Thirdly, the system $\P$ involves a weakening of \emph{Standardisation}, which is motivated as follows:  The axiom \textsf{S} may be equivalently formulated as follows:
\be\tag{$\textsf{S}'$}
(\forall^{\st}x)(\exists^{\st}y)\Phi(x, y)\di \big(\exists^{\st}F\big)(\forall^{\st}x)\Phi(x,F(x)),
\ee
for any formula $\Phi$ (possible involving nonstandard parameters) in the language $\IST$, i.e.\ we observe that \textsf{S} implies\footnote{We point out that the relationship between $\textsf{S}$ and the axiom of choice is non-trivial: in $\ZF + \I + \textsf{S} +\T$ one can only prove that every filter can be extended to an ultrafilter, but not the full axiom of choice (See \cite{hrbacekJLA}).  On the other hand, the system $\IST\setminus \textsf{S}$ does prove $\textsf{S}\asa \textsf{S}'$.} a fragment of the axiom of choice (relative to `st').    
In light of the possible non-constructive content of the latter (See e.g.\ \cite{dias}), it is no surprise that \textsf{S} has to be weakened.  
In particular, $\P$ includes the following version of \textsf{S}, called the \emph{Herbrandised Axiom of Choice}:  
\be\label{HACINT2}\tag{$\HAC_{\INT}$}
(\forall^{\st}x^{\rho})(\exists^{\st}y^{\tau})\varphi(x, y)\di \big(\exists^{\st}G^{\rho\di \tau^{*}}\big)(\forall^{\st}x^{\rho})(\exists y^{\tau}\in G(x))\varphi(x,y),
\ee
where $\varphi$ is any internal formula in the language of $\P$.  Note that $G$ does not output a witness for $y$, but a \emph{finite list} of potential witnesses to $y$.  
This is quite similar to \emph{Herbrand's theorem} (\cite{buss}*{I.2.5}), hence the name of $\HAC_{\INT}$.  

\medskip

Finally, we list two basic but important axioms of $\P$ from Definition \ref{debs} below.  These axioms are inspired by Nelson's claim about $\IST$ as follows:
\begin{quote}
\emph{Every specific object of conventional mathematics is a standard set}. It remains unchanged in the new theory $[\IST]$. (\cite{wownelly}*{p.\ 1166})
\end{quote}
Specific objects of the system $\P$ obviously include the constants  $0, 1, \times, +$, and anything built from those.  Thus, the system $\P$ includes the following three axioms; we refer to Definition \ref{debs} for the exact technical details.     
\begin{enumerate}  
\renewcommand{\theenumi}{\roman{enumi}} 
\item All constants in the language of $\textsf{E-PA}^{\omega}$ are standard.
\item A standard functional applied to a standard input yields a standard output. 
\item If two objects are equal and one is standard, so is the other one.  \label{mikeh}
\end{enumerate}
As a result, the system $\P$ proves that any term of $\textsf{E-PA}^{\omega}$ is standard, which will turn out to be essential in Section \ref{thereandback}.  
Note that `equality' as in item \eqref{mikeh} only applies to `actual' equality \emph{and not equality on the reals} in $\P$, as discussed in Remark \ref{equ}. 

\subsubsection{The constructive system $\H$}\label{graf2}
We discuss the fragment $\H$ of $\IST$ from \cite{brie}.   Similar to the way $\IST$ is an extension of $\ZFC$,
$\H$ is just the internal system $\textsf{E-HA}^{\omega}$ with the language extended with a new standardness predicate `\st' and with some special cases of the external and internal axioms of $\IST$.  

\medskip

The technical details of this extension may be found in Section \ref{TATA}; we now provide an intuitive motivation for the external axioms of $\H$, assuming basic familiarity with the finite type system of G\"odel's system $T$ (See Section \ref{TITI}).  Note that $\textsf{E-HA}^{\omega}$ and $\H$ are based on \emph{intuitionistic logic}.  

\medskip

First of all, the system $\H$ does not involve \emph{Transfer} for the same reasons $\P$ does not.  By contrast, the axioms $\HAC_{\INT}$ and \emph{Idealisation} as in \eqref{wellhung} (and its contraposition) are included in $\H$, with the restrictions on $\varphi$ lifted even.  

\medskip

Secondly, the system $\H$ includes some `non-constructive' axioms relativised to `\st'.  We just mention the names of these axioms and refer to Section \ref{TATA} for a full description.  
The system $\H$ involves nonstandard versions of the following axioms: \emph{Markov's principle} (See e.g.\ \cite{beeson1}*{p.\ 47}) and the \emph{independence of premises} principle (See e.g.\ \cite{kohlenbach3}*{\S5}).  \emph{Nonetheless}, $\H$ proves the same internal sentence as $\textsf{E-HA}^{\omega}$ by Theorem \ref{TERM}, i.e.\ the nonstandard versions are not really non-constructive.  

\medskip 

Finally, $\H$ also includes the basic axioms from Definition \ref{debs} as listed at the end of Section \ref{graf1}.  
In particular, $\H$ proves that any term of $\textsf{E-HA}^{\omega}$ is standard.  

\section{The (non-)constructive nature of Nonstandard Analysis}\label{praxis}
\subsection{Introduction}\label{lintro}
A number of informal claims have over the years been made about the constructive nature of Nonstandard Analysis.  These claims range from short quotes, three of which listed below, to more detailed conjectures.  Regarding the latter, we discuss in Sections \ref{keiskers} to \ref{loco} observations by Keisler, Wattenberg, and Osswald on the constructive nature of Nonstandard Analysis.  
Note that `constructive' is often used as a synonym of `effective' in this context.      
\begin{quote}
{It has often been held that nonstandard analysis is highly non-constructive, thus somewhat suspect, depending as it does upon the ultrapower construction to produce a model \textup{[\dots]} On the other hand, nonstandard \emph{praxis} is remarkably constructive; having the extended number set we can proceed with explicit calculations.} (Emphasis in original: \cite{NORSNSA}*{p.\ 31})
\end{quote}
\medskip
\begin{quote}
On the other hand, nonstandard arguments often have in practice a
very constructive flavor, much more than do the corresponding standard proofs. Indeed, those who use nonstandard arguments often say of
their proofs that they are ``constructive modulo an ultrafilter." This is
especially true in measure theory. (\cite{rosse}*{p.\ 230})
\end{quote}
\medskip
\begin{quote}
The aim of this section is to describe a nonstandard map ``explicitly''. There is only one inconstructive step in the proof, namely the choice of a so-called ultrafilter. [\dots] Moreover, the fact that the approach is ``almost constructive'' has the advantage that in special cases one can better see what happens: Up to the ultrafilter one can ``calculate'' the nonstandard embedding $*$. (\cite{fath}*{p.\ 44})
\end{quote}
By contrast, Connes and Bishop have claimed that Nonstandard Analysis is somehow \emph{fundamentally non-constructive}; 
their views are discussed briefly in Section \ref{crackp} for the sake of completeness, while a detailed study may be found in \cite{samsynt}.

\medskip

Finally, the attentive reader has noted that we choose not to discuss the claims of the `French school' of Nonstandard Analysis, in particular Reeb and Harthong's \emph{Intuitionnisme 84} (\cite{reeb3}), regarding the constructive nature of Nonstandard Analysis.  
We motivate this choice by the observation that the associated literature is difficult to access (even in the digital age) and is written in large part in rather academic French, which (no pun intended) has long ceased to be the \emph{lingua franca}.  

\subsection{Keisler's lifting method}\label{keiskers}
We discuss Keisler \emph{lifting method} which is a template for the practice of Nonstandard
Analysis involving the hyperreal line, and a special case of the scheme in Figure \ref{figsaregross}.  
Keisler formulates the lifting method as:
\begin{quote}
{The following strategy, sometimes called the \emph{lifting method}, has
been used to prove results [\dots] on the ordinary real line. }
\begin{enumerate}
\item[Step 1] \textrm{Lift the given `real' objects up to internal
approximations on the hyperfinite grid.}
\item[Step 2] \textrm{Make a series of hyperfinite computations to construct
some new internal object on the hyperfinite grid.}
\item[Step 3] \textrm{Come back down to the real line by taking standard
parts of the results of the computations.}
\end{enumerate}
The hyperfinite computations in Step 2 will typically replace more
problematic infinite computations on the real line.  (\cite{kieken}*{p.\ 234})
\end{quote}
Keisler's description of the lifting method involves the words \emph{constructive} and \emph{computations}, which are elaborated upon
as follows by Keisler:
\begin{quote}
[\dots] hyperreal proofs seem to
be more `constructive' than classical proofs. For example, the solution
of a stochastic differential equation given by the hyperreal proof is
obtained by solving a hyperfinite difference equation by a simple 
induction and taking the standard part. The hyperreal proof is not constructive
in the usual sense, because in $\ZFC$ the axiom of choice is needed even
to get the existence of a hyperreal line. What often happens is that a proof
within \textsf{RZ} or $\IST$ of a statement of the form $(\exists x)\phi(x)$ will produce an $x$
which is definable from $H$, where $H$ is an arbitrary infinite hypernatural
number. Thus instead of a pure existence proof, one obtains an explicit
solution except for the dependence on $H$. The extra information one
gets from the explicit construction of the solution from $H$ makes the proof
easier to understand and may lead to additional results.  (\cite{kieken}*{p.\ 235})
\end{quote}
We conclude that Keisler's lifting method reflects the ideas behind the scheme in Figure \ref{figsaregross}.  
Furthermore, we stress that one can define the notion of (non-)standard real number in the systems $\P$ and $\H$ from Section \ref{frag} and develop Nonstandard Analysis there  including the Stone-Weierstra\ss~theorem as in \cite{sambon}*{\S3-4}.  In particular, most of \textsf{ZFC} is not needed to do (large parts of) Nonstandard Analysis, though one could (wrongly so) get the opposite   
impression from Keisler's previous quote.  

\subsection{Wattenberg's algorithms}\label{wat}
Almost two decades ago, Wattenberg published the paper \emph{Nonstandard Analysis and Constructivism?}~in which he speculates on a connection between Nonstandard Analysis and constructive mathematics (\cite{watje}).  
\begin{quote}
This is a speculative paper. For some time the author has been struck by an apparent affinity between two rather unlikely areas of mathematics - nonstandard analysis and constructivism. [\dots] 
The purpose of this paper is to investigate these ideas by examining several examples. (\cite{watje}*{p.\ 303})
\end{quote}
On one hand, with only slight modification, some of Wattenberg's theorems in Nonstandard Analysis are seen to yield effective and constructive theorems using the systems $\P$ and $\H$ from Section \ref{frag}.  On the other hand, some of Wattenberg's (explicit and implicit) claims regarding the constructive status of the axioms \emph{Transfer} and \emph{Standard Part} from Nonstandard Analysis can be shown to be incorrect.  These results will be explored in \cite{samwatje}, but we can make preliminary observations.  
    
\medskip

Wattenberg digs deeper than the quotes in Section \ref{lintro} by making the following important observation regarding the praxis of Nonstandard Analysis.
\begin{quote}
Despite an essential nonconstructive kernel, many nonstandard arguments are constructive until the final step, a step that frequently involves the standard part map. (\cite{watje}*{p.\ 303})
\end{quote}    
Wattenberg's observation echoes the scheme in Figure \ref{figsaregross} from Section \ref{intronsa}, namely that the axiom \emph{Standardisation} (aka \emph{Standard Part}) is used to push nonstandard solutions down to the world of standard/usual mathematics at the end of a nonstandard proof.  However, in either approach to Nonstandard Analysis this axiom is highly non-constructive, and even very weak instances are such:  Consider \emph{Standardisation} as in $(\textsf{S}')$ from Section \ref{graf1} weakened  as follows: 
\be\tag{$\textsf{S}''$}
(\forall f^{1})\big[(\forall^{\st}x^{0})(\exists^{\st}y^{0})(f(x)=_{0}y)\di \big(\exists^{\st}F\big)(\forall^{\st}x)(f(x)=_{0}F(x))\big].
\ee
It is proved in Section \ref{techni} that $\P+\textsf{S}''$ proves the non-constructive \emph{weak K\"onig's lemma}.  
Furthermore, $(\textsf{S}'')$ is translated to the \emph{special fan functional} (First introduced in \cite{samGH}; see Section \ref{RMKE2}) during term extraction, an object which is extremely hard\footnote{In particular, the special fan functional $\Theta$ from \cite{samGH} is not computable (in the sense of Kleene's S1-S9) from the Suslin functional, the functional version of $\FIVE$, as proved in \cite{dagsam}.} to compute, as studied in \cite{dagsam}.  

\medskip

In a nutshell, Wattenberg suggests that the mathematics in the nonstandard universe in Figure \ref{figsaregross} from Section \ref{intronsa} is rich in constructive content, but that stepping from the nonstandard to the standard universe using \emph{Standardisation} is fundamentally non-constructive.  Thus, the latter axiom should be avoided if one is interested in obtaining the constructive content of Nonstandard Analysis.  Note however that Wattenberg (explicitly and implicitly) makes use of \emph{Transfer} in \cite{watje}, which is \emph{also} highly non-constructive, as established in Section \ref{RMKE} and \cite{sambon}*{\S4}.  In other words, the picture Wattenberg tries to paint is somewhat imperfect, which shall be remedied in the next section.  
    
\subsection{Osswald's local constructivity}\label{loco}
We discuss the notion \emph{local constructivity}, originally formulated by Osswald  (See e.g.\ \cite{nsawork2}*{\S7}, \cite{Oss3}*{\S1-2}, or \cite{Oss2}*{\S17.5}).  
The following heuristic principle describes the notion of \emph{local constructivity}.  
\begin{princ}
\label{locoloco} A mathematical proof is \textbf{locally constructive} if the
core, the essential part, of the proof is constructive in nature.
\end{princ}
Intuitively, a proof $P$ of a theorem $T$ is \emph{locally constructive} if
we can omit a small number of non-constructive initial and/or final steps in 
$P$ and obtain a proof $Q$ of a constructive theorem $T^{\prime }$
very similar to $T$.  In other words, a locally constructive proof has the form $P=(P',P'', P''')$, where $P', P'''$ are inessential and may be non-constructive, and the `main part' $P''$ is constructive.  
As it happens, isolating the constructive part $P''$ can lead to new results, as in the paper\footnote{As a tentative example, the anonymous referee of \cite{Oss3} wrote in his report that techniques in the paper produce `results, substantially extending those in the literature'.\label{curxxxxxx}} \cite{Oss3} and Section~\ref{main}.

\medskip

Osswald has conjectured the \textbf{local constructivity} of (proofs in) classical \textbf{Nonstandard Analysis}.  This conjecture is based on the following three observations regarding the praxis of Nonstandard Analysis:
\begin{enumerate}
\item The first step in a proof in Nonstandard Analysis often involves \emph{Transfer} to enter the nonstandard universe, e.g.\ to convert epsilon-delta definitions 
to nonstandard ones, like for epsilon-delta continuity as in Example \ref{huntress}.  However, \emph{Transfer} is non-constructive as established in Section \ref{RMKE}.  
\item The {mathematical practice} of Nonstandard Analysis
\textbf{after entering the nonstandard world} in large part consists of the manipulation of (nonstandard) universal formulas not involving existential quantifiers. 
This manipulation amounts to mere computation, in many cases nothing more than hyperdiscrete sums and products.  
\item The final step in a proof of Nonstandard Analysis often involves \emph{Standardisation} to push nonstandard objects back into the standard universe, as also observed by Keisler and Wattenberg in the previous sections.  However, this axiom is also non-constructive.  
\end{enumerate}
These three observations suggest the following:  If we take a proof $P=(P',P'', P''')$ in Nonstandard Analysis where $P'$ (resp.\ $P'''$) collects the initial (resp.\ final) applications of \emph{Transfer} (resp.\ of \emph{Standardisation}), then $P''$ is a proof rich in constructive content.  The proof $P$ is thus \textbf{locally constructive} in the sense that $P''$ contains the core argument, and is constructive in a certain as-yet-undefined sense (which will be formalised in Section \ref{main}).   Furthermore, to avoid the use of \emph{Transfer}, we should avoid epsilon-definitions and work instead directly with the nonstandard definitions involving infinitesimals, as also discussed in Section \ref{palmke}.     

\medskip

In conclusion, Osswald has conjectured that \textbf{Nonstandard Analysis is locally constructive}, suggesting that what remains of Nonstandard Analysis after stripping away \emph{Transfer} and \emph{Standardisation} is constructive.  We will formalise this observation by introducing` \emph{pure} Nonstandard Analysis' in Definition~\ref{pure}.  

\subsection{The Bishop-Connes critique of Nonstandard Analysis}\label{crackp}
For completeness, we briefly discuss the critique of Nonstandard Analysis by Bishop and Connes, relating to the Bishop quote in Section \ref{intro}.  
Their critique may be summarised as:
\begin{center}
\emph{The presence of ideal objects \(in particular infinitesimals\) in Nonstandard Analysis yields the absence of computational content.}
\end{center}
In particular, for rather different reasons and in different contexts, Bishop and Connes equate `meaningful mathematics' and `mathematics with computational content', and therefore claim {Nonstandard Analysis} is devoid of meaning as it lacks -in their view- any and all computational content.
A detailed study of this topic by the author, debunking the Bishop-Connes critique, may be found in \cite{samsynt}.  

\medskip

The often harsh criticism of Nonstandard Analysis by Bishop and Connes is a matter of the historical record and discussed in remarkable detail in \cites{dahaus, art,kaka, kano2, gaanwekatten,tallmann, daupje}.  Furthermore, the \emph{arguments by Connes and Bishop for} this critique have been dissected in surprising detail \emph{and found wanting} in the aforementioned references, but establishing the \emph{opposite} of the Bishop-Connes critique, namely that Nonstandard Analysis is \emph{rich in computational content}, first took place in \cite{sambon,samsynt}.  

\medskip
As noted in Section \ref{intro}, we are not interested in adding to this literature, but we shall examine the motivations of Connes and Bishop for their critique, as this will provide more insight into Nonstandard Analysis.  
To this end, recall that Bishop considers his mathematics to be concerned with \emph{constructive objects}, i.e.\ those described by algorithms on the integers.  The following quote is telltale.  
\begin{quote}  
Everything attaches itself
to number, and every mathematical statement ultimately expresses the
fact that if we perform certain computations within the set of positive
integers, we shall get certain results. (\cite{bish1}*{\S1.1})
\end{quote}
This quote suggests a \emph{fundamental} ontological divide between Bishop's mathematics and Nonstandard Analysis, as the latter \emph{by design} is based on ideal objects (like infinitesimals) with prima facie \emph{no algorithmic description at all}.   
In this light, it is not a stretch of the imagination to classify Nonstandard Analysis as \emph{fundamentally non-constructive}, i.e.\ antipodal to Bishop's mathematics, or to paraphrase Bishop: \emph{an attempt to dilute meaning/computational content further}.  

\medskip

Furthermore, this `non-constructive first impression' is confirmed by the fact that the `usual' development of Nonstandard Analysis involves \emph{quite non-constructive}\footnote{The usual development of Robinson's Nonstandard Analysis proceeds via the construction of a nonstandard model using a free ultrafilter.  The existence of the latter is only slightly weaker than the axiom of choice of $\ZFC$ (\cite{nsawork2}).} axioms.  The latter is formulated by Katz and Katz as:
\begin{quote}
[\dots] the hyperreal approach incorporates an element of non-constructivity at the basic level of the very number system itself. (\cite{kaka}*{\S3.3})
\end{quote}
Finally, the aforementioned `first impression' does not seem to disappear if one considers more basic mathematics, e.g.\ arithmetic rather than set theory.
Indeed, \emph{Tennenbaum's theorem} \cite[\S11.3]{kaye} `literally' states that any nonstandard model of Peano Arithmetic is not computable.  \emph{What is meant} is that for a nonstandard model $\M$ of Peano Arithmetic, the operations $+_{\M}$ and $\times_{\M}$ cannot be computably defined in terms of the operations $+_{\N}$ and $\times_{\N}$ of the standard model $\N$ of Peano Arithmetic.  
In light of the above, Nonstandard Analysis \emph{seems} fundamentally non-constructive even at the level of arithmetic.   

\medskip

Now, the Fields medallist Alain Connes has formulated similar negative criticism of classical Nonstandard Analysis in print on at least seven occasions.  
The first table in \cite{kano2}*{\S 3.1} runs a tally for the period 1995-2007.  Connes judgements range from \emph{inadequate} and \emph{disappointing}, to \emph{a chimera} and \emph{irremediable defect}.  Regarding the effective content of Nonstandard Analysis, the following quote is also telltale.    
\begin{quote}
Thus a non-standard number gives us canonically a non-measurable subset of $[0,1]$. This is the end of the rope for being `explicit' since (from another side of logics) one knows that it is just impossible to construct explicitely a non-measurable subset of $[0,1]$! 
(Verbatim copy of the text in \cite{conman3})
\end{quote}
In contrast to Bishop, Connes does not take a foundational position but has more pragmatic arguments in mind:  According to Connes, Nonstandard Analysis is fundamentally non-constructive and 
thus useless for physics, as follows:
\begin{quote}
 The point is that as soon as you have a non-standard number, you get a non-measurable set. And in Choquet's circle, having well studied the Polish school, we knew that every set you can name is measurable. So it seemed utterly doomed to failure to try to use non-standard analysis to do physics.  (\cite{conman}*{p.\ 26})
\end{quote}  
Indeed, a major aspect of physics is the testing of hypotheses against experimental data, nowadays done on computers.  But how can this testing be done if the mathematical formalism at hand is fundamentally non-constructive, as Connes claims?  

\medskip

In Section \ref{fraki}, we establish the opposite of the Bishop-Connes critique: rather than being `fundamentally non-constructive', infinitesimals and other nonstandard objects provide an \emph{elegant shorthand} for computational/constructive content.  

\section{Jenseits von Konstruktiv und Klassisch}\label{main}
In this section, we shall establish that Nonstandard Analysis occupies the twilight zone between the constructive and non-constructive.   
Perhaps more accurately, as also suggested by the title\footnote{We leave it to the reader to decide how our title corresponds to Nietzsche's famous book.} of this section, Nonstandard Analysis will be shown to \emph{transcend} the latter distinction.  

\medskip

To this end, we explore the constructive status of $\H$ extended with classical logic in Section \ref{fraki} and shall discover that these extensions can neither rightfully be called non-constructive nor constructive.   
What is more, the observations from Section \ref{fraki} pertaining to logic allow us to uncover the vast computational content of classical Nonstandard Analysis.  
We shall consider three examples in Section \ref{hum} to \ref{RMKE2} and formulate a general template in Section~\ref{henk} for obtaining computation from \emph{pure} Nonstandard Analysis, as defined in Definition \ref{pure}.    
These results will establish the \textbf{local constructivity} of Nonstandard Analysis as formulated in Section \ref{loco}.  
As a crowning achievement, we show in Section \ref{thereandback} that theorems of Nonstandard Analysis are `meta-equivalent' to theorems rich in computational content.  

\subsection{The twilight zone of mathematics}\label{fraki}
In this section, we study some computational and constructive aspects of the systems $\H$ and $\P$.  During this study, it will turn out that the latter occupies the twilight zone between the constructive and non-constructive, as suggested in Section \ref{intro}.    

\medskip

We first consider the system $\H$, the conservative extension of the constructive system \textsf{E-HA}$^{\omega}$ from the previous section.  
Now, if $\LEM$ consists of $A\vee \neg A$ for all internal $A$, then $\textsf{E-HA}^{\omega}+\LEM$ is just $\textsf{E-PA}^{\omega}$.  In other words, the addition of $\LEM$ turns $\textsf{E-HA}^{\omega}$ into a \emph{classical} system, losing the constructive (BHK) interpretation of the quantifiers.   
Now consider $\H+\LEM$ and recall the following theorem, which is called the \emph{main theorem on program extraction} in \cite{brie}*{Theorem 5.9}.  
\begin{thm}[Term extraction I]\label{TERM}
Let $\varphi$ be internal, i.e.\ not involving `$\st$', and let $\Delta_{\INT}$ be a collection of internal formulas. 
If $\H+\Delta_{\INT}\vdash(\forall^{\st}x)(\exists^{\st}y)\varphi(x,y)$, then we can extract a term $t$ from this proof with $\textsf{\textup{E-HA}}^{\omega}+\Delta_{\INT}\vdash(\forall x)(\exists y\in t(x))\varphi(x,y)$.  
The term $t$ is such that $t(x)$ is a finite list and does not depend on $\Delta_{\INT}$.    
\end{thm}
Note that the \emph{conclusion} of the theorem, namely `$\textsf{E-HA}^{\omega}+\Delta_{\INT}\vdash(\forall x)(\exists y\in t(x))\varphi(x,y)$', \emph{does not involve Nonstandard Analysis} anymore. 
The term $t$ from the previous theorem is part of G\"odel's system $T$ from Section \ref{TITI}, i.e.\ essentially a computer program formulated in e.g.\ Martin-L\"of type theory or Agda (\cites{loefafsteken, agda}).  

\medskip

What is more important is the following observation (stemming from the proof of \cite{brie}*{Theorem 5.5}):  Taking $\Delta_{\INT}$ to be empty, we observe that constructive information (the term $t$) can be obtained from proofs in $\H$.  This is a nice result, but hardly surprising in light of the existing semantic approach to \emph{constructive} Nonstandard Analysis (See Section \ref{palmke}) and the realisability interpretation of constructive mathematics.   What is surprising is the following:  Taking $\Delta_{\INT}$ to be $\LEM$ in Theorem~\ref{TERM}, we observe that proofs in $\H+\LEM$ or $\H$ provide \emph{the same kind of constructive information}, namely the term $t$.  Of course, the final system (namely $\textsf{E-HA}^{\omega}+\LEM$) is non-constructive in the former case, but the fact remains that the \emph{original sin of non-constructivity} $\LEM$ \emph{does not influence} the constructive information (in the form of the term $t$) extracted from $\H$, as is clear from Theorem \ref{TERM}.  

\medskip

Hence, while $\textsf{E-HA}^{\omega}$ loses its constructive nature when adding $\LEM$, the system $\H$ retains its `term extraction property' by Theorem \ref{TERM} when adding $\LEM$.  
To be absolutely clear, we do \textbf{not} claim that $\H+\LEM$ is a constructive system, but observe that it has the \emph{same term extraction property} as the constructive system $\H$.   
For this reason, $\H+\LEM$ is neither constructive (due to the presence $\LEM$) nor can it be called non-constructive (due to the presence of the \textbf{same} term extraction property as the \emph{constructive} system $\H$).  
In conclusion, $\H+\LEM$ is an example of a system which inhabits the \emph{twilight zone} between the constructive and non-constructive.  

\medskip

An alternative view based on Brouwer's dictum \emph{logic depends upon mathematics} is that extending $\textsf{E-HA}^{\omega}$ to the nonstandard system $\H$ \emph{fundamentally} changes the mathematics.   The logic, dependent as it is on mathematics according to Brouwer, changes along too; an example of this change-in-logic is as follows: $\LEM$ changes from \emph{the original sin of non-constructivity} in the context of $\textsf{E-HA}^{\omega}$ to a \emph{computationally inert} statement in the context of $\H$ by Theorem \ref{TERM}.  We shall discuss this view in more detail in Section \ref{kiko}.    

\medskip

As a logical next step, we discuss $\P$, which is essentially $\H+\LEM_{\ns}$ and where the latter is the law of excluded middle $\Phi \vee \neg\Phi$ for \emph{any} formula (in the language of $\H$).  Note that $\LEM_{\ns}$ is much more general than $\LEM$.  \emph{Nonetheless}, we have the following theorem, not proved explicitly in \cite{brie}, and first formulated in \cite{sambon}.
\begin{thm}[Term extraction II]\label{TERM2}
Let $\varphi$ be internal, i.e.\ not involving `$\st$', and let $\Delta_{\INT}$ be a collection of internal formulas. 
If $\P+\Delta_{\INT}\vdash(\forall^{\st}x)(\exists^{\st}y)\varphi(x,y)$, then we can extract a term $t$ from this proof with $\textsf{\textup{E-PA}}^{\omega}+\Delta_{\INT}\vdash(\forall x)(\exists y\in t(x))\varphi(x,y)$.  
The term $t$ is such that $t(x)$ is a finite list and does not depend on $\Delta_{\INT}$.    
\end{thm}  
As for $\H+\LEM$, the system $\P$ is neither constructive due to the presence $\LEM$ \emph{and} $\LEM_{\ns}$, but it has the \textbf{same} term extraction property as the constructive system $\H$ by Theorem \ref{TERM2}, i.e.\ calling $\P$ non-constructive does not make sense in light of (and does not do justice to) Theorem \ref{TERM} for $\Delta_{\INT}=\emptyset$.  

\medskip

In short, the systems $\H+\LEM$ and $\P$ share essential properties with classical mathematics (in the form of $\LEM$ and $\LEM_{\ns}$) and with constructive mathematics (the term extraction property of $\H$ from Theorem \ref{TERM}).  For this reason, these systems can be said to occupy the \emph{twilight zone} between the constructive and non-constructive: They are neither one nor the other.  More evidence for the latter claim may be found in Section \ref{wag1} where $\P$ is shown to satisfy a natural nonstandard version of the \emph{existence property} and related `hallmark' properties of intuitionistic systems.  
We show in Section~\ref{RMKE} that adding \emph{Transfer} to $\P$ results in an unambiguously \emph{non-constructive} system; the axiom \emph{Transfer} will be shown to be the `real' law of excluded middle of Nonstandard Analysis in a very concrete way.       

\medskip

Furthermore, on a more pragmatic note, the above results betray that we can gain access to the computational/constructive content of \emph{classical} Nonstandard Analysis via the term extraction property formulated in Theorem~\ref{TERM2}.  
It is then a natural question \emph{how wide the scope of the latter is}.  The rest of this section is dedicated to showing that this scope encompasses \emph{pure Nonstandard Analysis}.  
The latter is Nonstandard Analysis restricted to \emph{nonstandard axioms} (\emph{Transfer} and \emph{Standardisation}) and \emph{nonstandard definitions} (for continuity, compactness, integration, convergence, et cetera).  This will establish the \textbf{local constructivity of Nonstandard Analysis} as formulated in Section \ref{loco}.  

\medskip
     
The above discussion deals with \emph{logic} and we now discuss a first \emph{mathematical} example.   
In particular, we now consider an elementary application of Theorem \ref{TERM2} to \emph{nonstandard continuity} as in Example \ref{leffe}.  We shall refer to a formula of the form $(\forall^{\st}x)(\exists^{\st}y)\varphi(x,y)$ with $\varphi$ internal as a \emph{normal form}.  
\begin{exa}[Nonstandard and constructive continuity]\label{krel}\rm
Suppose $f:\R\di \R$ is nonstandard continuous, provable in $\P$.  In other words, similar to Example \ref{leffe}, the following is provable in our system $\P$:
\be\label{frik}
(\forall^{\st}x\in \R)(\forall y\in \R)(x\approx y \di f(x)\approx f(y)).
\ee    
Since $\P$ includes \emph{Idealisation} (essentially) as in $\IST$, $\P$ also proves the following:
\begin{align}\label{exagoe}\textstyle
(\forall^{\st}x\in \R)(\forall^{\st}k\in \N)&(\exists^{\st}N\in\N)\\
&\textstyle\underline{(\forall y\in \R)\big( |x- y|<\frac{1}{N} \di|f(x)- f(y)|<\frac{1}{k}\big)},\notag
\end{align}
in exactly the same way as proved in Example \ref{leffe}.  Since the underlined formula in \eqref{exagoe} is internal, \eqref{exagoe} is a normal form and we note that Theorem \ref{TERM2} applies to `$\P\vdash \eqref{exagoe}$'.  Applying the latter, 
we obtain a term $t$ such that $\textsf{E-PA}^{\omega}$ proves:
\[\textstyle
(\forall  x\in\R, k\in \N)(\exists N\in t(x, k)){(\forall y\in \R)\big( |x- y|<\frac{1}{N} \di|f(x)- f(y)|<\frac{1}{k}\big)}.
\]
Since $t(x, k)$ is a finite list of natural numbers, define $s(x, k)$ as the maximum of $t(x, k)(i)$ for $i<|t(x, k)|$ where $|t(x,k)|$ is the length of the finite list $t(x,k)$.  
The term $s$ is called a \emph{modulus} of continuity of $f$ as it satisfies:
\be\label{yeah}\textstyle
(\forall  x\in\R, k\in \N){(\forall y\in \R)\big( |x- y|<\frac{1}{s(x, k)} \di|f(x)- f(y)|<\frac{1}{k}\big)}.
\ee
Similarly, from the proof in $\P$ that $f$ is nonstandard \emph{uniform} continuous, we may extract a modulus of \emph{uniform} continuity (See Section \ref{hum}).  
This observation is \emph{important}: moduli are an essential part of Bishop's \emph{Constructive Analysis} (See e.g.\ \cite{bish1}*{p.\ 34}), so we just proved that such constructive information \emph{is implicit in the nonstandard notion of continuity}!  

\medskip

For the `reverse direction', the basic axioms (See Definition \ref{debs}) of $\P$ state that all constants in the language of $\textsf{E-PA}^{\omega}$ are standard.  
Hence, if \textsf{E-PA}$^{\omega}$ proves \eqref{yeah} for some term $s$, then $\P$ proves \eqref{yeah} \textbf{and} that $s$ is standard.  Hence, $\P$ proves  
\be\label{yeah2}\textstyle
(\forall^{\st}  x\in\R, k\in \N)(\exists^{\st}N\in \N){(\forall y\in \R)\big( |x- y|<\frac{1}{N} \di|f(x)- f(y)|<\frac{1}{k}\big)}, 
\ee
as we may take $N=s(x, k)$.  However, \eqref{yeah2} clearly implies \eqref{frik}, i.e.\ nonstandard continuity also follows from `constructive' continuity with a modulus!  Furthermore, one establishes in exactly the same way that $\P\vdash[ \eqref{yeah}\di \eqref{frik}]$ for any term $t$.   
\end{exa}
By the previous example, the nonstandard notion of (uniform) continuity amounts to the same as the `constructive' definition involving moduli.  
As discussed in Section~\ref{henk}, other nonstandard definitions (of integration, differentiability, compactness, convergence, et cetera) are also `meta-equivalent' to their constructive counterparts.  
Thus,  rather than being `fundamentally non-constructive' as claimed by Bishop and Connes in Section \ref{crackp}, infinitesimals (and other nonstandard objects) provide an \emph{elegant shorthand} for computational content, like moduli from constructive mathematics.  Furthermore, the types involved in nonstandard statements are generally lower (than the associated statement not involving Nonstandard Analysis), which should appeal to the practitioners of \emph{Reverse Mathematics} where higher-type objects are represented via so-called codes in second-order arithmetic.  

\medskip

Finally, Theorem \ref{TERM2} allows us to extract computational content (in the form of the term $t$) from (classical) Nonstandard Analysis in the spirit of Kohlenbach's \emph{proof mining} program.  
We refer to \cite{kohlenbach3} for the latter, while we shall apply Theorem \ref{TERM2} in the next sections to establish the local constructivity of Nonstandard Analysis.

\subsection{Two examples of local constructivity}\label{hum}
In this section, we study two basic examples of theorems proved using Nonstandard Analysis, namely the \emph{intermediate value theorem} and \emph{continuous functions on the unit interval are Riemann integrable}.  We shall apply the heuristic provided by \emph{local constructivity}, i.e.\ only consider the parts of the proof in the nonstandard universe not involving \emph{Transfer} and \emph{Standard Part}, 
and apply Theorem~\ref{TERM2} to 
obtain the effective content.  
\subsubsection{Intermediate value theorem}\label{IVTNRM}
In this section, we discuss the computational content of a nonstandard proof of the \emph{intermediate value theorem}; we shall make use of the \emph{uniform} notion of nonstandard continuity, defined as follows.  
\bdefi\label{Kont}
The function $f$ is \emph{nonstandard uniformly continuous} on $[0,1]$ if
\be\label{soareyou4}
(\forall x, y\in [0,1])[x\approx y \di f(x)\approx f(y)].
\ee
\edefi
We first prove the \emph{intermediate value theorem} in $\IST$; all notions have their usual (epsilon-delta/internal) meaning, unless explicitly stated otherwise.  
\begin{thm}[$\IVT$]\label{soemple}
For every continuous function $f:[0,1]\di \R$ such that $f(0) f(1)< 0$, there is $x\in [0,1]$ such that $f(x)=0$.
\end{thm}
\begin{proof}
Clearly $\IVT$ is internal, and we may thus assume that $f$ is standard, as applying \emph{Transfer} to this restriction yields full $\IVT$.   
By Example~\ref{huntress}, $f$ is nonstandard continuous as in \eqref{soareyouke} for $X=[0,1]$ by \emph{Transfer}.  
The latter can also be used to show that $f$ is nonstandard \emph{uniform} continuous as in \eqref{frik}, but this follows easiest from $(\forall x\in [0,1])(\exists^{\st}y\in [0,1])(x\approx y)$, i.e.\ the nonstandard compactness of the unit interval, which follows in turn from applying \emph{Standard Part} to the trivial formula $(\forall^{\st}k\in \N)(\exists^{\st} q\in \Q)(|x-q|<\frac{1}{2^{k}})$ and noting that the resulting sequence converges (to a standard real).  
Now let $N$ be a nonstandard natural number and let $j\leq N$ be the least number such that $f(\frac{j}{N})f(\frac{j+1}{N})\leq 0$ (whose existence is guaranteed by $f(0)f(1)<0$ and induction).  Then $f(j/N)\approx 0$ by nonstandard uniform continuity.  By nonstandard compactness there is standard $y\in [0,1]$ such that $y\approx j/N$.  Then $f(y)\approx 0$ by nonstandard uniform continuity, and the former is short for $(\forall^{\st}k^{0})(|f(x)|<\frac{1}{k})$.  Finally, applying \emph{Transfer} yields $f(x)=0$.
\end{proof}
Local constructivity as in Section \ref{loco} dictates that we omit all instances of \emph{Transfer} and \emph{Standard Part} to obtain the constructive core of the nonstandard proof of $\IVT$.  
Note that the former axioms were used to obtain nonstandard \emph{uniform} continuity.  
Thus, we obtain the following version, which we prove in $\P$.  
\begin{thm}[$\IVT_{\ns}$]
For every standard and nonstandard uniformly continuous $f:[0,1]\di \R$ such that $f(0) f(1)< 0$, there is $x\in [0,1]$ such that $f(x)\approx 0$.
\end{thm}
\begin{proof}
Let $f$ be as in the theorem, fix $N^{0}>0$ and suppose for all $i\leq N$ that $f(\frac{i}{N})f(\frac{i+1}{N})> 0$.  If $f(0)>0$, let $\varphi(i)$ be the formula $i\leq N\di f(i/N)>0$.  Then $\varphi(0)$ and $\varphi(i)\di \varphi(i+1)$ for all $i\in \N$, both by assumption.  Induction now yields that $f(1)>0$, a contradiction.     
Thus, let $N$ be a nonstandard natural number and let $j\leq N$ be the least number such that $f(\frac{j}{N})f(\frac{j+1}{N})\leq 0$.  
Then $f(j/N)\approx 0$ by nonstandard uniform continuity.  The same approach works for $f(0)<0$.
\end{proof}
Applying \emph{term extraction} to $\IVT_{\ns}$ as in Theorem \ref{TERM2}, we obtain the following theorem, which includes a \emph{constructive} `approximate' version of $\IVT$ (See \cite{beeson1}*{I.7}).  
\begin{thm}\label{TE1}
From the proof `$\P\vdash \IVT_{\ns}$', we obtain a term $t$ such that $\textsf{\textup{E-PA}}^{\omega}$ proves 
that for all $f,g$ where $g$ is the modulus of uniform continuity of $f$, we have $(\forall k\in \N)(|f(t(f,g,k))|<_{\R}\frac{1}{k})$.
\end{thm}
\begin{proof}
We may apply Theorem \ref{TERM2} to normal forms only.  Now, the nonstandard uniform continuity of $f$ on $[0,1]$ has the following normal form:
\be\label{the}\textstyle
(\forall^{\st}k\in \N)(\exists^{\st}N\in\N)\underline{(\forall x,y\in [0,1])\big( |x- y|<\frac{1}{N} \di|f(x)- f(y)|<\frac{1}{k}\big)},
\ee
which is proved in $\P$ exactly as for nonstandard continuity in Example \ref{leffe}; the underlined formula is abbreviated $A(f, N, k )$.  The conclusion of $\IVT_{\ns}$ is not a normal form, but 
is \emph{equivalent}\footnote{Apply \emph{Idealisation} to $(\forall^{\st} z^{0^{*}})(\exists x\in [0,1])(\forall k \in z)(|f(x)|<\frac{1}{k})$ to see this.} to $(\forall^{\st} k\in \N)(\exists x\in [0,1])(|f(x)|<\frac{1}{k})$ by \emph{Idealisation}.  By nonstandard uniform continuity, we also have $(\forall^{\st} k\in \N)(\exists^{\st} q \in [0,1])(|f(q)|<\frac{1}{k})$, which \emph{is} a normal form.  
Hence, taking `$f\in D$' to mean that $f:\R\di \R$ satisfies $f(0)f(1)<0$, $\IVT_{\ns}$ becomes 
\be\label{squint}\textstyle
(\forall^{\st}f\in D)\big[(\forall^{\st}l^{0})(\exists^{\st}N^{0})A(f, N, l)  \di (\forall^{\st} k^{0})(\exists^{\st} q \in [0,1])(|f(q)|<\frac{1}{k})\big].
\ee
Since standard functions output standard values for standard inputs, we obtain 
\be\label{strong}\textstyle
(\forall^{\st}f\in D, g)\big[(\forall^{\st}l^{0})A(f, g(l), l)  \di (\forall^{\st} k^{0})(\exists^{\st} q \in [0,1])(|f(q)|<\frac{1}{k})\big],
\ee
and dropping the remaining `st' in the antecedent, we get
\be\label{weak}\textstyle
(\forall^{\st}f\in D, g)\big[(\forall l^{0})A(f, g(l), l)  \di (\forall^{\st} k^{0})(\exists^{\st} q \in [0,1])(|f(q)|<\frac{1}{k})\big],
\ee
which becomes the following normal form:
\be\label{norman}\textstyle
(\forall^{\st}f\in D, g,k )(\exists^{\st} q \in [0,1])\big[(\forall l^{0})A(f, g(l), l)  \di (|f(q)|<\frac{1}{k})\big].
\ee
Applying Theorem \ref{TERM2} to `$\P\vdash \eqref{norman}$', we obtain a term $s$ such that
\be\label{nosq2}\textstyle
(\forall f\in D, g,k )(\exists  q \in s(f, g, k))\big[(\forall l^{0})A(f, g(l), l)  \di (|f(q)|<\frac{1}{k})\big].
\ee
To obtain the term $t$ from the theorem, note that we can decide if $|f(q)|>\frac{1}{2^{k}}$ or $|f(q)|<\frac{1}{k}$ for any $q$, using the `law of comparison' (\cite{beeson1}*{\S1.3}) of constructive mathematics.  Indeed, this law states that (under the BHK-interpretation):
\be\label{kilkil}
(\forall x, y, z\in \R)(x<y \di x<z \vee z<y).
\ee
Thus, there is an effective procedure to decide which disjunct in \eqref{kilkil} holds.  
Hence, define $t(f, g, k)$ as the first element $q\in s(f,g,2^{k})$ such that $|f(q)|<\frac{1}{k}$.  We obtain  
\be\label{nosq}\textstyle
(\forall f\in D, g )\big[(\forall l^{0})A(f, g(l), l)  \di (\forall k^{0})(|f(t(f,g, k))|<\frac{1}{k})\big], 
\ee
from \eqref{nosq2} by pushing the quantifier pertaining to $k$ inside again.  
\end{proof}
Approximate versions of $\IVT$ are of course well-known, but the previous is not intended to be new, but merely to provide an example of local constructivity as in Section \ref{loco}:  After stripping away \emph{Transfer} and \emph{Standard Part} in the proof of Theorem \ref{soemple}, we are left with a proof of $\IVT_{\ns}$ in $\P$, which converts into a constructive version of $\IVT$ after term extraction by Theorem \ref{TE1}.  

\medskip

Furthermore, we point out the (apparently innocent) use of proof-by-contradiction in the proof of $\IVT_{\ns}$.  Also, Theorem \ref{TE1} is merely \emph{one of many} possible results obtainable by term extraction.  Indeed, while the conclusion of the former theorem has effective content, 
we can as well obtain \emph{relative computability results} from the proof of $\IVT$ in Theorem \ref{soemple}, e.g.\ by not removing the use of \emph{Transfer}.  We shall explore this avenue further in Section \ref{kujio}.    
  
\medskip

Finally, it is also important to note that we wrote out the proof of Theorem \ref{TE1} in full detail, while one readily skips from \eqref{squint} to \eqref{nosq} with some practice.  
To the latter end, we conclude this section with the following remark on normal forms; we show that moduli (like in \eqref{norman}) come about when converting an implication between two normal forms into an normal form in $\P$ (See Theorem \ref{nogwelconsenal} for $\H$).  
\begin{rem}[Normal forms and implication]\label{doeisnormaal}\rm
As discussed in Example \ref{krel}, the nonstandard definition of continuity can be brought into a normal form \eqref{exagoe}, while the nonstandard definition of \emph{uniform} continuity has the normal form \eqref{the}.  These observations are important as normal forms have computational content thanks to Theorem \ref{TERM2}.  The proof of Theorem \ref{TE1} suggests that \emph{normal forms are closed under implication}; indeed, the normal form \eqref{norman} is derived from an \emph{implication} between two normal forms.  We now establish the general case as follows.    
Let $\varphi, \psi$ be internal and consider the following implication between normal forms:
\be\label{nora}
(\forall^{\st}x)(\exists^{\st}y)\varphi(x, y)\di (\forall^{\st}z)(\exists^{\st}w)\psi(z, w).  
\ee
Since standard functionals have standard output for standard input, \eqref{nora} implies
\be\label{nora2}
(\forall^{\st}\zeta)\big[(\forall^{\st}x)\varphi(x, \zeta(x))\di (\forall^{\st}z)(\exists^{\st}w)\psi(z, w)\big].  
\ee
Bringing all standard quantifiers outside, we obtain the following normal form:
\be\label{nora3}
(\forall^{\st}\zeta, z)(\exists^{\st} w, x)\big[\varphi(x, \zeta(x))\di \psi(z, w)\big],
\ee
as the formula in square brackets is internal.  Now, \eqref{nora3} is equivalent to \eqref{nora2}, but one usually (like in the proof of Theorem \ref{TE1}) weakens the latter as follows:  
\be\label{nora4}
(\forall^{\st}\zeta, z)(\exists^{\st} w)\big[(\forall x)\varphi(x, \zeta(x))\di \psi(z, w)\big],
\ee
as \eqref{nora4} is closer to the usual mathematical definitions.  For instance, if the antecedent of \eqref{nora} is (the normal form of) uniform continuity, then the antecedent of \eqref{nora4} is the constructive definition of uniform-continuity-with-a-modulus, while this is not the case for \eqref{nora3}.  
We shall shorten the remaining proofs by just providing normal forms and jumping straight from \eqref{nora} to \eqref{nora4} whenever possible.  
\end{rem}  

\subsubsection{Riemann integration}\label{riekenaan}
In this section, we discuss the computational content of a nonstandard proof that continuous functions are Riemann integrable; we shall make use of the usual definitions of Riemann integration as follows.  
\bdefi[Riemann Integration]\label{kunko}~
\begin{enumerate}
\item A \emph{partition} of $[0,1]$ is an increasing sequence $\pi=(0, t_{0}, x_{1},t_{1},  \dots,x_{M-1}, t_{M-1}, 1)$.  We write `$\pi \in P([0,1]) $' to denote that $\pi$ is such a partition.
\item For $\pi\in P([0,1])$, $\|\pi\|$ is the \emph{mesh}, i.e.\ the largest distance between two adjacent partition points $x_{i}$ and $x_{i+1}$. 
\item For $\pi\in P([0,1])$ and $f:\R\di \R$, the real $S_{\pi}(f):=\sum_{i=0}^{M-1}f(t_{i}) (x_{i+1}-x_{i}) $ is the \emph{Riemann sum} of $f$ and $\pi$.  
\item A function $f$ is \emph{nonstandard integrable} on $[0,1]$ if
\be\label{soareyou5}
(\forall \pi, \pi' \in P([0,1]))\big[\|\pi\|,\| \pi'\|\approx 0  \di S_{\pi}(f)\approx S_{\pi'}(f)  \big].
\ee
\item A function $f$ is \emph{integrable} on $[0,1]$ if
\be\label{soareyou6}\textstyle
(\forall k^{0})(\exists N^{0})(\forall \pi, \rho \in P([0,1]))\big[\|\pi\|,\| \rho\|<\frac{1}{N}  \di |S_{\pi}(f)- S_{\rho}(f)|<\frac{1}{k}  \big].
\ee
A modulus of (Riemann) integration $\omega^{1}$ provides $N=\omega(k)$ as in \eqref{soareyou6}.  
\end{enumerate}
\edefi
Let $\RIE$ be the (internal) statement that a function is Riemann integrable on the unit interval if it is (pointwise) continuous there.  
\begin{thm}\label{henry}
The system $\IST$ proves $\RIE$.
\end{thm}
\begin{proof}
In a nutshell, the proof in \cite{loeb1}*{p.\ 57} goes through with minimal modification (from the Robinsonian framework). 
In the latter proof, \emph{Transfer} and \emph{Standardisation} (although the latter can be avoided) are used to derive the nonstandard definitions of \emph{uniform} continuity and integration from  the epsilon-delta definitions of pointwise continuity and Riemann integration.  
We have proved the equivalence between epsilon-delta and nonstandard continuity in Example \ref{huntress}, and the case for Riemann integration is essentially identical.  
It thus remains to prove that nonstandard uniform continuity implies nonstandard integrability.  

\medskip
\noindent
Thus, let $\pi'=(0, t_{0}', x'_{1},t_{1}',  \dots,x'_{M-1}, t_{M-1}', 1)$ be  $\pi=(0, t_{0}, x_{1},t_{1},  \dots,x_{M-1}, t_{M-1}, 1)$ with 
all reals replaced\footnote{Note that this can be done effectively by using the `law of comparison' (\cite{beeson1}*{\S1.3}) as in \eqref{kilkil}.} with their rational approximations of the form $\frac{i}{2^{M+1}}$.  By nonstandard uniform continuity, we have $\max_{i\leq M-1}|f(t_{i}')-f(t_{i})|=\zeta_{0}\approx0$, and define $x_{i+1}'-x_{i}'=x_{i+1}-x_{i}+\eps_{i}$ where $0\approx| \eps_{i}|\leq \frac{1}{2^{M+1}} $ by definition.  Applying \eqref{the} for $k=1$ and suitable standard $x\in [0,1]$, we see that $f(y)\leq M$ for any $y\in [0,1]$.  Thus $|\sum_{i=0}^{M-1}f(t_{i}') \eps_{i}|\leq  \frac{M^{2}}{2^{M+1}}\approx 0$ and consider the following:
\begin{align}\textstyle
|S_{\pi}(f)-S_{\pi'}(f)| &\textstyle=|\sum_{i=0}^{M-1}f(t_{i}) (x_{i+1}-x_{i}) -\sum_{i=0}^{M-1}f(t_{i}') (x_{i+1}'-x_{i}') |\notag\\
&\textstyle=|\sum_{i=0}^{M-1}(f(t_{i}')-f(t_{i}))\cdot (x_{i+1}-x_{i})+ \sum_{i=0}^{M-1}f(t_{i}') \eps_{i} |\notag \\
&\textstyle\leq \sum_{i=0}^{M-1}|f(t_{i}')-f(t_{i})| \cdot|x_{i+1}-x_{i}|+ |\sum_{i=0}^{M-1}f(t_{i}') \eps_{i} |\notag \\
&\textstyle \lessapprox   \sum_{i=0}^{M-1}\zeta_{0} \cdot|x_{i+1}-x_{i}|\notag\\
&\textstyle=\zeta_{0} \sum_{i=0}^{M-1}\cdot|x_{i+1}-x_{i}|=\zeta_{0}\approx 0. \label{hork}
\end{align}
For two partitions $\pi, \rho\in P([0,1])$, consider the associated `approximate' partitions $\pi'$ and $\rho'$.  
Since the latter only consist of rationals, it is easy to define `refinements' $\pi'', \rho''$ of equal length and which contain all points with even index from both $\pi'$ and $\rho'$; the points of odd index have to be repeated as necessary.  We then have 
\[
S_{\pi}(f)\approx S_{\pi'}(f)\approx S_{\pi''}(f)\approx S_{\rho''}(f)\approx S_{\rho'}(f)\approx S_{\rho}(f),   
\]
where only the third `$\approx$' requires a proof, which is analogous to \eqref{hork}.  
\end{proof}
Local constructivity as in Section \ref{loco} dictates that we omit all instances of \emph{Transfer} and \emph{Standard Part} to obtain the constructive core of the nonstandard proof of $\RIE$.  
Note that the former axioms were used to obtain nonstandard \emph{uniform} continuity and let $\textsf{RIE}_{\ns}$ be the statement that \emph{every function $f$ which is nonstandard uniformly continuous on the unit interval is nonstandard integrable}.  

\medskip

Let $\RIE_{\ef}(t)$ be the statement that \emph{$t(g)$ is a modulus of integration if $g$ is a modulus of uniform continuity of $f$ on the unit interval}.  The latter theorem may be called `constructive' as we copied it from Bishop's constructive analysis (See \cite{bridges1}*{p.\ 53}).  The nonstandard version apparently yields the constructive version.  
\begin{thm}\label{varou}
From a proof of $\RIE_{\ns}$ in $\P $, a term $t$ can be extracted such that $\textup{\textsf{E-PA}}^{\omega} $ proves $\RIE_{{\ef}}(t)$.  
\end{thm}
\begin{proof}
A normal form for nonstandard uniform continuity is \eqref{the}, while one derives the following normal form for nonstandard integrability as in Example \ref{leffe}:
\[\textstyle
(\forall^{\st}k')(\exists^{\st}N')\big[(\forall \pi, \pi' \in P([0,1]))(\|\pi\|,\| \pi'\|\leq \frac{1}{N'}  \di |S_{\pi}(f)- S_{\pi}(f)|\leq \frac{1}{k'} )\big],
\]
where $B(k', N', f)$ is the (internal) formula in square brackets and the underlined formula in \eqref{the} is abbreviated $A(f, N, k )$. Then $\RIE_{\ns}$ is the following implication
\be\label{bikko}
(\forall f:\R\di \R)\underline{\big[ (\forall^{\st}k)(\exists^{\st} N)A(f,N,k)\di (\forall^{\st}k')(\exists^{\st}N')B(k', N', f)\big]}.
\ee
and applying Remark \ref{doeisnormaal}, i.e.\ the step from \eqref{nora} to \eqref{nora4}, to the underlined formula in \eqref{bikko}, we obtain the following:
\be\label{griy}
(\forall f:\R\di \R)(\forall^{\st}g,k')(\exists^{\st}N'){\big[ (\forall k)A(f, g(k),k)\di B(k', N', f)\big]}.
\ee
Now, \eqref{griy} is not a normal form, but note that it trivially implies 
\be\label{hideaal}
(\forall^{\st}g,k')\underline{(\forall f:\R\di \R)(\exists^{\st}N'){\big[ (\forall k)A(f, g(k),k)\di B(k', N', f)\big]}}, 
\ee
where the underlined formula has the right from to apply \emph{Idealisation} \textsf{I}.  We obtain:
\[
(\forall^{\st}g,k')(\exists^{\st}w^{0^{*}}){(\forall f:\R\di \R)(\exists N'\in w){\big[ (\forall k)A(f, g(k), k)\di B(k', N', f)\big]}}, 
\]
which is a normal form.  Define $N_{0}:=\max_{i<|w|}w(i)$ and note that 
\be\label{horsee}
(\forall^{\st}g,k')(\exists^{\st}N_{0}){(\forall f:\R\di \R){\big[ (\forall k)A(f, g(k),k)\di B(k', N_{0}, f)\big]}}, 
\ee
since $B(k',M, f )\di B(k', K, f)$ for $K>M$ by the definition of Riemann integration.  Applying Theorem \ref{TERM2} to `$\P\vdash \eqref{horsee}$', we obtain a term $s$ such that
\[
(\forall g,k')(\exists N_{0}\in s(g,k')){(\forall f:\R\di \R){\big[ (\forall k)A(f, g(k), k)\di B(k', N_{0}, f)\big]}}, 
\]
and defining $t(g,k')$ as $\max_{i<|s(g, k')|}s(g, k')(i)$, we obtain $\RIE_{\ef}(t)$ in light of the definitions of $A, B$ from the beginning of this proof.
\end{proof}
Note that the actual computation in $\RIE_{\ef}(t)$ only takes place on the modulus $g$; this is a consequence of $\RIE_{\ns}$ applying to \emph{all} functions, not just the standard ones. 
Indeed, in the previous proof we could obtain \eqref{hideaal} and then `pull the standard existential quantifier $(\exists^{\st}N)$ through the internal one $(\forall f)$' thanks to \emph{Idealisation}, and then obtain \eqref{horsee} in which the quantifier $(\forall f)$ has no influence anymore on term extraction as in Theorem \ref{TERM2}.  

\medskip

Clearly, the same strategy as in the previous proof works for \emph{any} normal form 
in the scope of an internal quantifier. Indeed $(\forall z)(\forall^{\st}x)(\exists^{\st}y^{\tau})\varphi(x,y, z)$ implies $(\forall^{\st}x)(\exists^{\st}w^{\tau^{*}})(\forall z)(\exists y\in w)\varphi(x,y, z)$ by \emph{Idealisation}.  Furthermore, as is done in the previous proof for $\tau=0$, if $\varphi$ is (somehow) monotone in $y$, we may define $y_{0}$ as (some kind of) maximum of $w(i)$ for $i<|w|$, and obtain 
$(\forall^{\st}x)(\exists^{\st}y_{0}^{\tau})(\forall z)\varphi(x,y, z)$.  We finish this section with an even stronger result: a normal form in the scope of an \emph{external} quantifier over all infinitesimals is again a normal form.  This result also sheds light on the correctness of the intuitive infinitesimal calculus.  
\begin{thm}[$\P$]\label{hujiku}
For internal $\varphi$, $(\forall \eps\approx 0)(\forall^{\st}x)(\exists^{\st}y^{\tau})\varphi(x, y, \eps)$ is equivalent to a normal form.   
If $\tau=0$ and $(\forall z^0, w^{0})\big(z>w\di (\varphi(x, w, \eps)\di \varphi(x, z, \eps))\big)$, then 
\begin{align}
(\forall \eps\approx 0)(\forall^{\st}x)(\exists^{\st}y^{0})\varphi(x, y, \eps)
&\asa (\forall^{\st}x)(\exists^{\st}y^{0})(\forall \eps\approx 0)\varphi(x, y, \eps)\label{frohiki}\\
&\textstyle\asa (\forall^{\st}x)(\exists^{\st}y^{0}, N^{0}){(\forall \eps)\big[|\eps|<\frac{1}{N}\di \varphi(x, y, \eps)\big]}.\notag
\end{align}
\end{thm}  
\begin{proof}
Written out in full, the initial formula from the theorem is:
\[\textstyle
(\forall \eps)\big[(\forall^{\st} k^{0})(|\eps|<\frac{1}{k})\di (\forall^{\st}x)(\exists^{\st}y^{\tau})\varphi(x, y, \eps)\big],
\]
and bringing outside all standard quantifiers as far as possible:
\[\textstyle
(\forall^{\st}x)\underline{(\forall \eps)(\exists^{\st}y^{\tau}, k^{0})\big[|\eps|<\frac{1}{k}\di \varphi(x, y, \eps)\big]},
\]
the underlined formula is suitable for \emph{Idealisation}.  Applying the latter yields
\[\textstyle
(\forall^{\st}x)(\exists^{\st}w^{0^{*}}, z^{\tau^{*}}){(\forall \eps)(\exists y^{\tau} \in z, k^{0}\in w)\big[|\eps|<\frac{1}{k}\di \varphi(x, y, \eps)\big]},
\]
and let $N^{0}$ be the maximum of all $w(i)$ for $i<|w|$.  We obtain:
\[\textstyle
(\forall^{\st}x)(\exists^{\st}z^{\tau^{*}}, N){(\forall \eps)(\exists y^{\tau}\in z)\big[|\eps|<\frac{1}{N}\di \varphi(x, y, \eps)\big]}.
\]
If $\tau=0$ and $\varphi$ is monotone as above, let $y_{0}$ be $\max_{i<|z|}z(i)$ and obtain:
\be\label{duh}\textstyle
(\forall^{\st}x)(\exists^{\st}y_{0}, N){(\forall \eps)\big[|\eps|<\frac{1}{N}\di \varphi(x, y_{0}, \eps)\big]}, 
\ee
which is as required by \eqref{frohiki} when pushing $(\exists^{\st}N)$ inside the square brackets.  
\end{proof}
The monotonicity condition on $\varphi$ in the theorem occurs a lot in analysis:   
In any `epsilon-delta' definition, the formula resulting from removing the `epsilon' and `delta' quantifiers is monotone (in the sense of the theorem) in the `delta' variable.    

\medskip

Furthermore, \eqref{frohiki} explains why correct results can be produced by the intuitive infinitesimal calculus (from physics and engineering):  Many arguments in the latter calculus 
produce formulas in the same form as the left-hand side of \eqref{frohiki}.  Applying term extraction on the bottom normal form of \eqref{frohiki}, we obtain a term $t$ such that:  
\be\label{rung}\textstyle
(\forall x)(\exists y^{0}, N^{0}\leq t(x)){(\forall \eps)\big[|\eps|<\frac{1}{N}\di \varphi(x, y, \eps)\big]},   
\ee
which does not involve Nonstandard Analysis.  However, assuming $x$ as in \eqref{rung} is only used in a certain (discrete) range, as seems typical of physics and engineering, 
we can fix \emph{one} `very large compared to the range of $x$' value $N_{0}$ for $N$ and still obtain correct results in that $(\exists y)\varphi(x, y, 1/N_{0})$ for $x$ in the prescribed range.  

\medskip

In short: \eqref{frohiki} and \eqref{rung} explain why the practice `replacing infinitesimals by very small numbers' typical of physics and engineering can produce correct results.  


\subsection{An example involving \emph{Transfer}}\label{RMKE}
\subsubsection{Introduction}
In the previous section, we provided examples of local constructivity, i.e.\ we stripped the nonstandard proofs of \emph{Transfer} and \emph{Standardisation} and obtained effective results using Theorem \ref{TERM2}.  
A natural question, discussed in this section for \emph{Transfer}, is what happens if we \emph{do not omit} these axioms.  As we will see, nonstandard proofs involving \emph{Transfer} give rise to \emph{relative computability results}.

\medskip

By way of an elementary example, we study a nonstandard proof \emph{involving Transfer} of the \emph{intermediate value theorem} in Section \ref{kujio}.  
By way of an advanced example, we obtain  in Section \ref{MOCO} effective results \`a la \emph{Reverse Mathematics}\footnote{We refer to \cite{simpson2, simpson1} for an overview of Friedman's foundational program \emph{Reverse Mathematics}, first introduced in \cite{fried, fried2}.}.
In particular, we establish that $\P$ proves the equivalence between a fragment of \emph{Transfer} and a nonstandard version of the \emph{monotone convergence theorem} (involving nonstandard convergence), in line with the equivalences from {Reverse Mathematics}.  From this nonstandard equivalence, we 
extract terms of G\"odel's $T$ which -intuitively speaking- convert a solution to the Halting problem into a solution to the monotone convergence theorem, and vice versa.  

\medskip

The above results, published first in part in \cite{sambon}, serve a dual purpose as follows:
\begin{enumerate}
\item \emph{Transfer} is shown to be \emph{fundamentally non-constructive}, as it converts to Feferman's non-constructive mu-operator, i.e.\ essentially the Turing jump.  
\item Nonstandard equivalences as in Theorem \ref{sef} give rise to \emph{effective} equivalences like \eqref{froodcor} which are \emph{rich in computational content}.  
\end{enumerate}
The first item directly vindicates local constructivity from Section \ref{loco}.  
The second item is relevant as there is a field called \emph{constructive Reverse Mathematics} (See e.g.\ \cite{ishi1} for an overview) where equivalences are proved in a system 
\emph{based on intuitionistic logic}.  In particular, the second item establishes that Reverse Mathematics done in classical Nonstandard Analysis gives rise to `rather\footnote{Again, we do not claim that \eqref{froodcor} counts as constructive mathematics (due to the classical base theory $\textsf{E-PA}^{\omega}$), but the former formula contains too much (automatically extracted from the nonstandard equivalence) computational information to be dismissed as `non-constructive'.} constructive' Reverse Mathematics, in line with the observations from Section \ref{fraki} regarding the semi-constructive status of classical Nonstandard Analysis.   

\medskip

We again stress that our results are only meant to illustrate the limits and potential of local constructivity.  
While e.g.\ the effective equivalences in Section~\ref{MOCO} are not necessarily new or surprising, \emph{it is surprising} that we 
can obtain them from nonstandard proofs \emph{without any attempt whatsoever at working constructively}.  

\subsubsection{The intermediate value theorem}\label{kujio}
We study a nonstandard proof involving \emph{Transfer} of the \emph{intermediate value theorem} from Section \ref{IVTNRM}.
To this end, let $\IVT_{\ns}'$ be $\IVT_{\ns}$ with `$\approx$' replaced by `$=_{\R}$' in the conclusion.  For the latter replacement, the following restriction of Nelson's axiom \emph{Transfer} is needed:
\be\tag{$\paai$}
(\forall^{\st}f^{1})\big[(\forall^{\st}n^{0})f(n)\ne0\di (\forall m)f(m)\ne0\big].
\ee 
We also need Feferman's \emph{non-constructive mu-operator} (\cite{avi2}) as follows:  
\be\label{mu}\tag{$\mu^{2}$}
(\exists \mu^{2})\big[(\forall f^{1})( (\exists n)f(n)=0 \di f(\mu(f))=0)    \big],
\ee
where $\textsf{MU}(\mu)$ is the formula in square brackets in \eqref{mu}.  Finally, let $\IVT_{\ef}(t)$ state that $t(f,g)$ is the intermediate value for $f$ with $g$ its modulus of continuity.  
\begin{thm}\label{TE1338}
From the proof $\P\vdash \paai\di  \IVT_{\ns}$, we obtain a term $t$ such that $\textsf{\textup{E-PA}}^{\omega}$ proves $(\forall \mu^{2})(\MU(\mu)\di \IVT_{\ef}(t))$.
\end{thm}
\begin{proof}
Clearly, $\paai$ is equivalent to the following normal form:
\be\label{huwji}
(\forall^{\st}g^{1})(\exists^{\st}m)\big[(\exists n^{0})g(n)=0\di (\exists i\leq m)g(i)=0\big],
\ee
while $\IVT_{\ns}'$ has a normal form $(\forall^{\st}f, g)(\exists^{\st}x)C(f, g, x)$, where $C(f, g, x)$ expresses that if $g$ is a modulus of uniform continuity of $f\in D$, then $f(x)=_{\R}0$. 
The latter normal form is obtained in the same way as for $\IVT_{\ns}$ in the proof of Theorem~\ref{TE1}.   Now let $D(g, m)$ be the formula in square brackets in \eqref{huwji} and note that $ \paai\di  \IVT_{\ns}$ implies the following implication between normal forms:
\[
(\forall^{\st}g)(\exists^{\st} m)D(g, m)\di (\forall^{\st}f, g)(\exists^{\st}x)C(f, g, x).
\]
Following Remark \ref{doeisnormaal}, we may conclude that $\P$ proves:
\[
(\forall^{\st}f, g, \mu^{2})(\exists^{\st}x)\big[(\forall g)D(g, \mu(g))\di C(f, g, x)\big],
\]
and applying Theorem \ref{TERM2}, we obtain a term $s$ such that \text{\textsf{E-PA}}$^{\omega}$ proves:
\be\label{lio}
(\forall f, g, \mu^{2})(\exists x\in s(f, g, \mu))\big[(\forall g)D(g, \mu(g))\di C(f, g, x)\big].  
\ee
We recognise the antecedent of \eqref{lio} as (a slight variation of) $\MU(\mu)$ and note that the implication in $(\mu^{2})$ is also an equivalence.  
In other words, Feferman's mu-operator allows us to decide $\Pi_{1}^{0}$-formulas, which includes formulas like $f(x)=_{\R}0$.  
Thus, given $(\exists x\in s(f, g, \mu))(f(x)=_{\R}0)$ as in \eqref{lio}, we just use the mu-operator to test which $s(f, g, \mu)(i)$ for $i<|s(f, g, \mu)|$ is an intermediate value of $f$.
Let the term $t$ be such that $t(f, g, \mu)$ is such an intermediate value and note that \eqref{lio} implies $(\forall \mu^{2})(\MU(\mu)\di \IVT_{\ef}(t))$, i.e.\ we are done.
\end{proof}
In conclusion, nonstandard proofs involving \emph{Transfer} give rise to \emph{relative computability results} such as in the previous theorem.  
The result in the latter is not meant to be new or optimal (See \cite{kohlenbach2} for such results), but merely illustrates \emph{why} we remove \emph{Transfer} as dictated by local constructivity: Even a small fragment like $\paai$ gives rise to (what is essentially) a solution to the Halting problem.  

\medskip

The conclusion of Theorem \ref{TE1338} need also not be surprising:  $\IVT$ is rejected in constructive mathematics and it hence stands to reason that we need some non-computable object, like e.g.\ Feferman's mu-operator, to compute intermediate values.  
In fact, Kohlenbach proves versions of $(\mu^{2})\asa (\exists \Phi)\IVT_{\ef}(\Phi)$ in \cite{kohlenbach2}*{\S3}.  
\subsubsection{The monotone convergence theorem}\label{MOCO}
In this section, we extract (effective) equivalences \`a la Reverse Mathematics from (non-effective) equivalences in Nonstandard Analysis involving \emph{Transfer}.    

\medskip

First, we introduce a nonstandard version of the monotone convergence theorem $\MCT$ involving \emph{nonstandard convergence}, as follows:  
\bdefi[{\MCT$_{\textsf{ns}}$}] For every \emph{standard} sequence $x_{n}$ of reals, we have
\[
(\forall n \in \N)(0\leq_{\R}x_{n}\leq_{\R} x_{n+1}\leq_{\R} 1)\di (\forall N,M\in \N)[  \neg\st(N)\wedge \neg\st(M)\di   x_{M}\approx x_{N}].
\]
\edefi
Note that the conclusion expresses the nonstandard convergence of $x_{n}$.
The corresponding \emph{effective/constructive}\footnote{Similar to the case of continuity, a modulus of convergence is part and parcel of the constructive definition of convergence, as discussed in e.g.\ \cite{bish1}*{p.\ 27}.} version is:
\bdefi[$\textsf{MCT}_{\textsf{ef}}(t)$] For any sequence $x_{n}$ of reals, we have
\[\textstyle
(\forall n \in \N)(0\leq_{\R}x_{n}\leq_{\R} x_{n+1}\leq_{\R} 1)\di (\forall k, N,M\in \N )[N, M \geq t(x_{(\cdot)})(k))\di |c_{M}- c_{N}|\leq \frac{1}{k} ] .  
\]
\edefi
\noindent
We also require the following functional, which is equivalent to $(\mu^{2})$ by (\cite{kohlenbach2}*{\S3}), 
\be\label{mukio}\tag{$\exists^{2}$}
(\exists \varphi^{2})\big[(\forall f^{1})( (\exists n)f(n)=0 \asa \varphi(f)=0 )   \big],
\ee
and which is called the \emph{Turing jump functional}.
Denote by $\TJ(\varphi)$ the formula in square brackets in \eqref{mukio}.  We have the following theorem and corollary.  
\begin{thm}\label{sef}
The system $\P$ proves that $\MCT_{\ns}\asa \paai$.  
\end{thm}
\begin{proof}
First of all, to establish $\MCT_{\ns}\di \paai$, fix standard $f^{1}$ such that $(\forall^{\st}n)f(n)\ne0$ and define the \emph{standard} sequence  $x_{(\cdot)}$ of reals as follows
\be\label{leuk}
x_{k}:=
\begin{cases}
0 &(\forall i\leq k)f(i)\ne0 \\
\sum_{i=1}^{k}\frac{1}{2^{i}} & \textup{otherwise}
\end{cases}.
\ee  
Clearly, $x_{(\cdot)}$ is weakly increasing and hence nonstandard convergent by $\MCT_{\ns}$.  However, if $(\exists m_{0})f(m_{0}+1)=0$ and $m_{0}$ is the least such number, we would have $0=x_{m_{0}}\not\approx x_{m_{0}+1}\approx 1$.  The latter contradicts $\MCT_{\ns}$ and we must therefore have $(\forall n^{0})(f(n)\ne 0)$.  
Thus, $\paai$ follows and we now prove the other direction in two different ways.  To this end, let $x_{n}$ be as in $\MCT_{\ns}$ and consider the formula
\be\label{funko}\textstyle
(\forall^{\st}  k\in \N)(\exists^{\st}m\in \N) \underline{(\forall^{\st} N,M\geq m)[|x_{M}- x_{N}|\leq \frac1k}], 
\ee
which expresses that $x_{n}$ `epsilon-delta' converges relative to `st'.  If \eqref{funko} is false, there is some $k_{0}$ such that $(\forall^{\st}m\in \N) (\exists^{\st} M>N\geq m)[x_{M}>x_{N}+\frac{1}{k_{0}}]$.  However, applying\footnote{To `apply this formula $k_{0}+1$ times', replace all $x_{n}$ in $(\forall^{\st}m\in \N) (\exists^{\st} M>N\geq m)[x_{M}>x_{N}+\frac{1}{k_{0}}]$ by approximations up to some fixed nonstandard number and apply $\HAC_{\INT}$.  Since all numbers are rationals (due to the approximation), we can test which number is the correct one.} the latter formula $k_{0}+1$ times, we obtain some $M_{0}$ such that $x_{M_{0}}>_{\R}1$, contradiction our assumptions.  Hence, \eqref{funko} holds and applying $\paai$ to the innermost underlined formula yields that $x_{n}$ is nonstandard convergent.    

\medskip

For a second proof of $\paai\di \MCT_{\ns}$ based on the existing literature, note that $\paai$ implies the following normal form:
\be\label{huji}
(\forall^{\st}f^{1})(\exists^{\st}m)\big[(\exists n^{0})f(n)=0\di (\exists i\leq m)f(i)=0\big],
\ee
to which $\HAC_{\INT}$ may be applied to obtain standard $\nu^{1\di 0^{*}}$ such that
\be\label{veil2}
(\forall^{\st}f^{1})(\exists m\in \nu(f))\big[(\exists n^{0})f(n)=0\di (\exists i\leq m )f(i)=0\big].
\ee
Now define $\mu(f)$ as the maximum of $\nu(f)(i)$ for $i<|\nu(f)|$ and conclude that
\be\label{veil}
(\forall^{\st}f^{1})\big[(\exists n^{0})f(n)=0\di (\exists i\leq \mu(f) )f(i)=0\big].
\ee
Note that we actually have equivalence in \eqref{veil}, i.e.\ $\mu$ allows us to decide existential formulas as long as a standard function describes the quantifier-free part as in the antecedent of \eqref{veil}.  Hence, the system $\P+\paai$ proves arithmetical comprehension $\ACA_{0}$ (See \cite{simpson2}*{III}) relative to the `st'.  A well-known result from Reverse Mathematics (See \cite{simpson2}*{III.2.2}) is that $\RCA_{0}$ proves $\ACA_{0}\asa \MCT$.  Since $\P$ also proves the axioms of $\RCA_{0}$ relative to `$\st$', the system $\P+\paai$ proves $\MCT^{\st}$.  
Hence, for $x_{n}$ as in the latter, there is standard $x\in \R$ such that
  \be\label{funko2}\textstyle
(\forall^{\st}  k\in \N)(\exists^{\st}m\in \N) \underline{(\forall^{\st} N,M\geq m)[|x_{M}- x|\leq \frac1k}].
\ee
Applying $\paai$ to the underlined formula in \eqref{funko2} now finishes the proof.  
\end{proof}
\begin{cor}\label{sefcor}
From any proof of $\MCT_{\ns}\asa \paai$ in $\P $, two terms $s, u$ can be extracted such that $\textup{\textsf{E-PA}}^{\omega}$ proves:
\be\label{froodcor}
(\forall \mu^{2})\big[\textsf{\MU}(\mu)\di \MCT_{\ef}(s(\mu)) \big] \wedge (\forall t^{1\di 1})\big[ \MCT_{\ef}(t)\di  \MU(u(t))  \big].
\ee
\end{cor}
\begin{proof}
We prove the corollary for the implication $\paai \di \MCT_{\ns}$ and leave the other one to the reader.    
The sentence $\MCT_{\ns}$ is readily converted to
\begin{align}\label{noniesimpel}
 (\forall^{\st} x_{(\cdot)}^{0\di 1}, k)(\exists^{\st}m)\big[(\forall n^{0})(0\leq &~x_{n}\leq x_{n+1}\leq 1)\\
 &\di (\forall N,M\geq m)[|x_{M}- x_{N}|\textstyle\leq \frac1k],\notag
\end{align}
in the same way as in Example \ref{leffe}.  
Now let $A$ (resp.\ $B$) be the formula in square brackets in \eqref{veil} (resp.\ \eqref{noniesimpel}) and note that the existence of standard $\mu^{2}$ as in \eqref{veil} implies $\paai$.  
Hence, $\paai\di \MCT_{\ns}$ gives rise to the following:
\be\label{loveisall}
\big[(\exists^{\st}\mu^{2})(\forall^{\st}f^{1})A(f, \mu(f))\big]\di (\forall^{\st} x_{(\cdot)}^{0\di 1}, k^{0})(\exists^{\st}m^{0})B(x_{(\cdot)}, k, m), 
\ee
and bringing outside the standard quantifiers, we obtain
\be\label{ecatie}
(\forall^{\st} x_{(\cdot)}^{0\di 1}, k^{0}, \mu^{2})(\exists^{\st}m, f)[A(f, \mu(f))\di B(x_{(\cdot)}, k, m)], 
\ee
which is a normal form as the formula in square brackets is internal.  
Now applying Theorem \ref{TERM2} to `$\P\vdash \eqref{ecatie}$' yields a term $t$ such that $\textsf{E-PA}^{\omega}$ proves
\be\label{hermct}
(\forall  x_{(\cdot)}^{0\di 1}, k^{0}, \mu^{2})(\exists m, f\in t(\mu, x_{(\cdot)}, k))[A(f, \mu(f))\di B(x_{(\cdot)}, k, m)], 
\ee
and define $s(\mu, x_{(\cdot)}, k)$ to be the maximum of all entries for $m$ in $t(\mu, x_{(\cdot)}, k)$.
We immediately obtain, using classical logic, that
\[
(\forall   \mu^{2})[(\forall f^{1})A(f, \mu(f))\di (\forall x_{(\cdot)}^{0\di 1}, k^{0})B(x_{(\cdot)}, k, s(\mu, x_{(\cdot)}, k))], 
\]
which is exactly as required by the theorem, i.e. the first conjunct of \eqref{froodcor}.  
\end{proof}
With some effort, the proof of $\MCT_{\ns}\asa\paai$ goes through in $\H$, i.e.\ the theorem is constructive, and the corollary then applies to $\H$ and $\textsf{E-HA}^{\omega}$.  
The axioms from Definition \ref{flah} are crucial in this context.  In any case, the previous results indicate that \emph{Transfer} is \emph{highly non-constructive}, as it is converted into
the Turing jump (or stronger) after term extraction, thus lending credence to the \textbf{local constructivity} of Nonstandard Analysis.  

\medskip

Furthermore, we stress that the results in \eqref{froodcor} are not (necessarily) surprising in and of themselves.  
\emph{What is surprising} is that we can `algorithmically' derive these effective results from the \emph{quite simple} classical proof of $\paai\asa \MCT_{\ns}$, in which \emph{no efforts towards effective results are made}.  Furthermore, while \eqref{froodcor} deals with higher types, it is possible to derive \emph{classical} computability theory, i.e.\ dealing only with subsets of $\N$, as discussed in Section \ref{forgo}.  

\medskip

Finally, we point out one entertaining property of \emph{Transfer} with regard to the law of excluded middle:  
When working in $\P$, there are in general \emph{three} possibilities for any internal formula $A(x)$:  (i) $A(x)$ holds for all $x$, (ii) there is \emph{standard} $x$ such that $\neg A(x)$, and (iii) there is $y$ such that $\neg A(y)$, but no such standard number exists.  Hence, \emph{in the extended language of} $\P$ 
there is a \textbf{third}\footnote{Note that one can prove using intuitionistic logic that $\neg(\neg Q \wedge \neg\neg Q)$ for any $Q$, i.e.\ there is no explicit `third option' in which both $Q$ and $\neg Q$ are false (\cites{gleuf, brouw}).\label{reminder}} possibility alongside of the two usual ones, namely that there is a counterexample, but it is not standard. 

\medskip

Now, the \emph{Transfer} axiom \textsf{T} excludes the aforementioned third possibility, i.e.\ in $\IST$ we have $(\exists^{\st}x)\neg A(x) \vee (\forall x)A(x)$.    
Similar to Section \ref{fraki}, 
we observe that the fundamental change from mainstream Peano arithmetic $\textsf{E-PA}^{\omega}$ to the system of Nonstandard Analysis $\P$ has an influence on the associated logic:   
The `usual' law of excluded middle included in $\P$ allows for three possibilities as above, while \emph{Transfer} reduces this spectrum to two, i.e.\ the latter plays the role of the law of excluded middle \emph{in the extended language of Nonstandard Analysis}.   
\subsection{An example involving \emph{Standard Part}}\label{RMKE2}
In Section \ref{hum}, we provided examples of local constructivity, i.e.\ we stripped the nonstandard proofs of \emph{Transfer} and \emph{Standardisation} and obtained effective results using Theorem \ref{TERM2}.   A natural question is what happens if we \emph{do not omit} the latter axiom.  

\medskip

To this end, we study in Section \ref{WASP} a nonstandard proof \emph{involving Standard Part} of $\RIE$ formulated with \emph{pointwise} continuity.  
As we will see, nonstandard proofs involving \emph{Standard Part} give rise to \emph{rather exotic} {relative computability results}.
In particular, similar to Section \ref{MOCO}, we establish that $\P$ proves the equivalence between $\STP$, a fragment of \emph{Standard Part}, and a nonstandard version of $\RIE$.   
Applying Theorem \ref{TERM2} to this nonstandard equivalence, we observe that $\STP$ is converted to the \emph{special fan functional}, a cousin of Tait's fan functional with rather exotic computational properties, as discussed in Sections \ref{WASP} and \ref{STPORE}.  However, the special fan functional is also \emph{quite natural} from a different point of view: 
it is implicit in \emph{Cousin's lemma}, as discussed in Section \ref{WARP}. 

\medskip

The above results serve a triple purpose as follows:
\begin{enumerate}
\item \emph{Standard Part} is \emph{non-constructive}, as it gives rise to the \emph{special fan functional}, which is extremely hard to compute in classical mathematics.  
\item The \emph{combination} of the fragments of \emph{Standard Part} and \emph{Transfer} corresponding to $\WKL_{0}$ and $\ACA_{0}$ gives rise to $\ATR_{0}$.  Thus, \emph{Standard Part} is `even more' non-constructive in the presence of \emph{Transfer}.
\item Nonstandard equivalences as in Theorem \ref{expected} give rise to \emph{effective} equivalences like in Theorem \ref{kohort} which are \emph{rich in computational content}.  
\end{enumerate}
The first and second items directly vindicate local constructivity from Section \ref{loco}.  
The final item is relevant for \emph{constructive Reverse Mathematics} as discussed in the previous section.  
By contrast, we discuss in Section \ref{nexot} how results involving the special fan functional give rise to `normal' results involving known objects.  

\subsubsection{Weak K\"onig's lemma and continuity}\label{WASP}
In this section, we study a nonstandard proof involving a fragment of \emph{Standard Part}.  
In particular, we study the theorem $\RIE^{\pw}$ that \emph{a pointwise continuous function on the unit interval is integrable}. 

\medskip

First of all, \emph{weak K\"onig's lemma} ($\WKL$) is the statement \emph{every infinite binary tree has a path} and gives rise to the second `Big Five' system of Reverse Mathematics (See e.g.\ \cite{simpson2}*{IV}).
The following fragments of \emph{Standard Part} are equivalent (See Theorem \ref{lapdog}) and constitute the nonstandard counterpart of {weak K\"onig's lemma}. 
\bdefi[Equivalent fragments of \emph{Standard Part}]\label{kefi}
\be\tag{\textup{\textsf{STP}}}\label{STP}
(\forall \alpha^{1}\leq_{1}1)(\exists^{\st}\beta^{1}\leq_{1}1)(\alpha\approx_{1}\beta).
\ee  
\be\tag{$\STP_{\R}$}\label{STPR}
(\forall x\in [0,1])(\exists^{\st}y\in [0,1])(x\approx y).
\ee
\be\label{frukkklk12}
(\forall f^{1})(\exists^{\st} g^{1})\big( (\forall^{\st}n^{0})(\exists^{\st}m^{0})(f(n)=m)\di   f\approx_{1}g\big).
\ee
\begin{align}\label{fanns2}
(\forall T^{1}\leq_{1}1)\big[(\forall^{\st}n)(\exists \beta^{0})&(|\beta|=n \wedge \beta\in T ) \di (\exists^{\st}\alpha^{1}\leq_{1}1)(\forall^{\st}n^{0})(\overline{\alpha}n\in T)   \big],
\end{align}
\begin{align}\label{frukkklk112}
(\forall^{\st}g^{2})(\exists^{\st}w)\big[(\forall T^{1}\leq_{1}1)\big((\forall \alpha^{1} \in &w(1))(\overline{\alpha}g(\alpha)\not\in T)\\
&\di(\forall \beta\leq_{1}1)(\exists i\leq w(2))(\overline{\beta}i\not\in T) \big)\big]. \notag
\end{align}  
\edefi
There is no deep philosophical meaning in the words `nonstandard counterpart':  This is just what the principle $\STP$ is called in the literature (See e.g.\ \cite{pimpson}). 
Note that \eqref{fanns2} is just $\WKL^{\st}$ with the `st' dropped in the leading quantifier, while \eqref{frukkklk12} is a version of the axiom of choice limited to $\Pi_{1}^{0}$-formulas (See \cite{simpson2}*{Table~4, p.\ 54}), and \eqref{frukkklk112} is a normal form for $\STP$.  Also, \ref{STPR} expresses the nonstandard compactness of the unit interval by \cite{loeb1}*{p.\ 42}.     
Thus, these nonstandard equivalences reflect the `usual' equivalences involving $\WKL$ from \cite{simpson2}*{IV}.  

\medskip

Now, by \cite{simpson2}*{IV.2.7}, $\WKL$ is equivalent to $\RIE^{\pw}$ and 
the nonstandard counterparts behave in exactly the same way: letting $\RIE_{\ns}^{\pw}$ be the statement \emph{every \(pointwise\) nonstandard continuous function is nonstandard integrable on the unit interval}, we have the following (expected) theorem.   
\begin{thm}\label{expected}
The system $\P$ proves that $\STP_{\R}\asa \RIE_{\ns}^{\pw}$.  
\end{thm}
\begin{proof}
For the forward implication, in light of the definitions of (uniform) nonstandard continuity \eqref{frik} and \eqref{soareyou4}, the latter follows from the former by $\STP_{\R}$, and $\RIE_{\ns}^{\pw}$ follows from Theorem \ref{henry} and noting that the latter establishes that $\RIE_{\ns}$ is provable in $\P$.  For the reverse implication, assume $\RIE_{\ns}^{\pw}$ and suppose $\STP_{\R}$ is false, i.e.\ there is $x_{0}\in [0,1]$ such that $(\forall^{\st}y\in [0,1])(x\not\approx y)$.  Note that $0\ll x_{0}\ll 1$ as the endpoints of the unit interval are standard.    
Fix nonstandard $N^{0}$ and define the function $\sin\big( \frac{1}{|x-x_{0}|+\frac{1}{N}}  \big)$.  The latter  is well-defined for any $x\in [0,1]$ and nonstandard pointwise continuous, but not nonstandard continuous at $x_{0}$.  
\end{proof}
While Theorem \ref{expected} may be expected in light of known Reverse Mathematics results, applying term extraction as in Theorem \ref{TERM2} to the former theorem, we obtain an object with \emph{unexpected} properties, namely the \emph{special fan functional}.  
\bdefi[Special fan functional]\label{dodier}
We define $\SCF(\Theta)$ as follows for $\Theta^{(2\di (0\times 1^{*}))}$:
\[
(\forall g^{2}, T^{1}\leq_{1}1)\big[(\forall \alpha \in \Theta(g)(2))  (\overline{\alpha}g(\alpha)\not\in T)
\di(\forall \beta\leq_{1}1)(\exists i\leq \Theta(g)(1))(\overline{\beta}i\not\in T) \big]. 
\]
Any functional $\Theta$ satisfying $\SCF(\Theta)$ is referred to as a \emph{special fan functional}.
\edefi
Note that there is \emph{no unique} special fan functional, i.e.\ it is in principle incorrect to make statements about `the' special fan functional. 
As pointed out in \cite{dagsam}, from a computability theoretic perspective, the main property of the special fan functional is the selection of $\Theta(g)(2)$ as a finite sequence of binary sequences $\langle f_0 , \dots, f_n\rangle $ such that the neighbourhoods defined from $\overline{f_i}g(f_i)$ for $i\leq n$ form a cover of Cantor space;  almost as a by-product, $\Theta(g)(1)$ can then be chosen to be the maximal value of $g(f_i) + 1$ for $i\leq n$.  As discussed in Section \ref{metastable}, the special fan functional is also quite natural from the point of view of Tao's notion of \emph{metastability}.   

\medskip

Let $\RIE_{\ef}^{\pw}(t)$ be the statement that $t(g)$ is a modulus of Riemann integration for $f$ if $g$ is a modulus of pointwise continuity for $f$.  
We have the following theorem.
\begin{thm}\label{kohort}
From any proof $\P\vdash [\STP\di \RIE_{\ns}^{\pw}]$, a term $t$ can be extracted such that $\textsf{\textup{E-PA}}^{\omega}$ proves $(\forall \Theta)[\SCF(\Theta)\di \RIE_{\ef}^{\pw}(t(\Theta))]$.
\end{thm}
\begin{proof}
A normal form for $\STP$ is given by \eqref{frukkklk112} while a normal form for $\RIE_{\ns}^{\pw}$ is similar to the normal for \eqref{horsee} for $\RIE_{\ns}$.  
The rest of the proof is now straightforward in light of Remark \ref{doeisnormaal}.  
\end{proof}
The result of the previous theorem is similar to what one would expect from the previous sections.  
However, the special fan functional has rather \emph{exotic computational properties}, as we discuss in the next section.  

\medskip

Finally, we point out one entertaining property of \emph{Standard Part} with regard to the law of excluded middle:  
When working in $\P$, there are in general \emph{three} possibilities for any internal formula $A(x)$:  (i) $A(x)$ holds for all $x$, (ii) there is \emph{standard} $x$ such that $\neg A(x)$, and (iii) there is $y$ such that $\neg A(y)$, but no such standard number exists.  Hence, \emph{in the extended language of} $\P$ 
there is a \textbf{third}\footnote{Recall Footnote \ref{reminder} concerning the truth values of intuitionistic logic.} possibility alongside of the two usual ones, namely that there is a counterexample, but it is not standard. 
Now, the \emph{Standard Part} axiom \textsf{S} excludes the latter third possibility in certain cases, i.e.\ in $\P+\STP$ 
we have $(\exists^{\st}x^{1}\leq_{1}1)\neg A(x) \vee (\forall x^{1}\leq_{1}1)A(x)$ assuming  $A$ `is nonstandard continuous' as follows:
\be\label{when}
(\forall^{\st} f^{1})(\forall g^{1})(f\approx_{1}g\di (A(g)\di A(f))), 
\ee
which expresses\footnote{To see this, resolve `$\approx_{1}$' in \eqref{when} and push the resulting extra standard quantifier to the front using \emph{idealisation}.} that $A$ only depends on some (standard) initial segment of a (standard) witness, similar to Brouwer's axiom \textsf{WC-N} (\cite{AD}).  
Formulas behaving as in \eqref{when} are readily found in computability theory; take for example $A(f)\equiv (\varphi_{e,s}^{f}(n)=m)$ with standard numerical parameters.  
Thus, the following typical example of \emph{Transfer} also follows from $\STP$ thanks to the above considerations:
\[
(\forall^{\st}e, s, m,n)\big[  (\exists B^{1})(\varphi_{e,s}^{B}(n)=m) \di (\exists^{\st} B^{1})(\varphi_{e,s}^{B}(n)=m)  \big].
\]
In conclusion, as in Section \ref{fraki} and the previous one, we observe that the fundamental change from mainstream Peano arithmetic $\textsf{E-PA}^{\omega}$ to the system of Nonstandard Analysis $\P$ has an influence on the associated logic:   
The `usual' law of excluded middle included in $\P$ allows for three possibilities as above, while \emph{Standard Part} reduces this spectrum to two in certain cases, i.e.\ the latter can play the role of the law of excluded middle \emph{in the extended language of Nonstandard Analysis}.   
\subsubsection{Computational properties of the special fan functional}\label{STPORE}
We discuss the exotic computational properties of the special fan functional.  As it turns out, the latter is \emph{extremely hard to compute} in 
classical mathematics, while easy to compute in intuitionistic mathematics.  All results on the special fan functional are from \cite{dagsam,samGH}.

\medskip

First of all, the `usual' fan functional $\Phi^{3}$ was introduced by Tait (See e.g.\ \cite{noortje,longmann}) as an example of a functional not computable (in the sense of Kleene's S1-S9; \cite{noortje,longmann}) over 
the total continuous functionals.  The usual fan functional $\Phi$ returns a modulus of uniform continuity $\Phi(Y)$ on Cantor space on input a continuous functional $Y^{2}$.    
Since Cantor space and the unit interval are isomorphic, it is clear that the usual fan functional $\Phi$ also computes a modulus of Riemann integration for (pointwise) continuous functionals on the unit interval.  With some effort, these kinds of results follow from the implication $\WKL^{\st}\di (\RIE^{\pw})^{\st}$.  In particular, $\WKL^{\st}$ is translated to (something similar to) the usual fan functional after term extraction as in Theorem \ref{TERM2}.      
Furthermore, the usual fan functional may be computed (via a term in G\"odel's $T$) from $(\mu^{2})$ by combining the results in \cite{kohlenbach4,kooltje}.  

\medskip

Secondly, the special fan functional fulfils a role similar to the usual fan functional in Theorem \ref{kohort}: it computes a modulus of Riemann integration for $f$ from a modulus of pointwise continuity from $f$.  However, $(\mu^{2})$ cannot compute (in the sense of Kleene's S1-S9) any $\Theta$ as in $\SCF(\Theta)$; 
the same holds for \emph{any} type two functional, including the \emph{Suslin functional} which corresponds to $\FIVE$, the strongest Big Five system (\cite{kohlenbach2}).  Thus, it seems one needs the functional $(\mathcal{E}_{2})$:
\be\tag{$\mathcal{E}_{2}$}\label{hah}
(\exists \xi^{3})(\forall Y^{2})\big[  (\exists f^{1})(Y(f)=0)\asa \xi(Y)=0  \big], 
\ee
to compute $\Theta$ such that $\SCF(\Theta)$.  But \eqref{hah} implies \emph{second-order arithmetic}, 
which is \emph{off-the-chart} qua computational strength compared to $(\mu^{2})$, i.e.\ the special fan functional is extremely \emph{hard to compute} in classical mathematics, while the \emph{intuitionistic fan functional} $\MUC$ from \cite{kohlenbach2} computes the special fan functional.
The fact that $(\mu^{2})$ does not compute the special fan functional translates to the nonstandard fact that $\P+\paai$ does not prove $\STP$, in contrast to the situation for the standard counterparts $\ACA_{0}$ and $\WKL_{0}$ from Reverse Mathematics where the former implies the latter.    
Furthermore, $\STP+\paai$ implies $\ATR_{0}$ relative to `st' which translates to the fact that the special fan functional $\Theta$ and Feferman's mu compute a realiser for $\ATR_{0}$ (via a term of G\"odel's $T$).  

\medskip

Thirdly, by contrast, $\textsf{E-PA}^{\omega}+(\exists \Theta)\SCF(\Theta)$ is a conservative extension of $\textsf{E-PA}^{\omega}+\WKL$, i.e.\ the special fan functional has quite \emph{weak first-order strength} in isolation.  This conservation result requires the \emph{intuitionistic fan functional $\Omega^{3}$} as follows:
\be\tag{$\MUC(\Omega)$}
(\forall Y^{2})(\forall f^{1}, g^{1}\leq_{1}1)(\overline{f}\Omega(Y)=_{0}\overline{g}\Omega(Y)\di Y(f)=_{0}Y(g))
\ee
which results in a conservative extension of $\textsf{E-PA}^{\omega}+\WKL$ by \cite{kohlenbach2}*{Theorem 3.15}.  
Now, the intuitionistic fan functional computes  the special one via a term of G\"odel's $T$ by Theorem \ref{foor} and its corollary. 
Thus, the special fan functional is rather \emph{easy} to compute in intuitionistic mathematics.  Proofs may be found in Section \ref{techni}.      
\begin{thm}\label{foor}
The axiom \ref{STP} can be proved in $\P$ plus the axiom
\be\label{kunt}\tag{\textsf{\textup{NUC}}}
(\forall^{\st}Y^{2})(\forall f^{1}, g^{1}\leq_{1}1)(f\approx_{1} g\di Y(f)=_{0}Y(g)).  
\ee
In particular $\P+(\exists^{\st}\Omega^{3})\MUC(\Omega)$ proves $\STP$.  
\end{thm}
\begin{cor}\label{scruf}
From the proof in $\P$ that $\NUC\di \STP$, a term $t^{3\di3}$ can be extracted such that 
$\textup{\textsf{E-PA}}^{\omega}$ proves $(\forall \Omega^{3})\big[\MUC(\Omega)\di \SCF(t(\Omega))]$.  
\end{cor}
We point out that Kohlenbach (\cite{kohlenbach4}*{\S5}) has introduced axiom schemas with properties similar to the special fan functional $\Theta$.  These schemas are logical in nature in that they are formulated using comprehension over formula classes.  By contrast, $\Theta$ is `purely mathematical' as it emerges from the normal form of $\STP$.    
%
%

\medskip

Finally, we point out an anecdote by Friedman regarding Robinson from 1966.  
\begin{quote}
I remember sitting in Gerald Sacks' office at MIT and telling him
about this [version of Nonstandard Analysis based on PA] and the conservative extension proof. He was
interested, and spoke to A. Robinson about it, Sacks told me that A.
Robinson was disappointed that it was a conservative extension. \cite{HFFOM}
\end{quote}
In light of this quote, we believe Robinson would have enjoyed learning about the `new' mathematical object that is 
the special fan functional originating from Nonstandard Analysis.  As it happens, \emph{many} theorems of second-order arithmetic 
can be modified to yield similar `special' functionals with extreme computable hardness.    

\medskip

In conclusion, we provided an example of a nonstandard proof involving \emph{Standard Part} ($\STP$ in particular) which gives rise to a \emph{rather exotic} {relative computability result}, namely involving the \emph{special fan functional} which is extremely hard to compute classically, and rather easy to compute intuitionistically.      
This observation is particularly surprising as by \eqref{fanns2} $\STP$ is so similar to $\WKL$ relative to `st', and the latter behaves completely `standard'.

\subsubsection{Non-exotic computability results and \emph{Standard Part}}\label{nexot}
As discussed at the end of Section \ref{intro}, part of this paper's \emph{Catch22} is that we cannot show too much `strange stuff' emerging from Nonstandard Analysis lest the reader gets the idea that Nonstandard Analysis is somehow fundamentally strange.  For this reason, we discuss in this section how the above results involving the (admittedly strange) special fan functional give rise to `normal' results with more computational content.   For reasons of space, we only provide a sketch while technical details are in \cite{dagsam}.  

\medskip

First of all, consider the following nonstandard version of the axiom of choice:
\be\tag{$\textsf{\textup{AC}}_{\ns}^{0}$}
(\forall^{\st}n^{0})(\exists^{\st}x^{\rho})\Phi(n, x)\di (\exists^{\st}F^{0\di \rho})(\forall^{\st}n^{0})\Phi(n, F(n)), 
\ee
where $\Phi$ can be any formula in the language of $\P$.  As shown in \cites{espa, brie3}, the system $\P+\textsf{AC}_{\ns}^{0}$ satisfies a version of Theorem \ref{TERM2} but with $t$ defined using \emph{bar recursion}, an advanced computation scheme which (somehow) embodies the computational content of the axiom of choice (See \cite{kohlenbach3} for details on bar recursion).  
Hence, since $\textsf{\textup{AC}}_{\ns}^{0}$ readily implies $\STP_{\R}$ by Theorem \ref{lapdog}, we may consider the proof $\P+\textsf{\textup{AC}}_{\ns}^{0}\vdash \RIE_{\ns}^{\pw}$.  Instead of Theorem~\ref{kohort}, we obtain a \emph{bar recursive} term $t$ such that $\RIE_{\ef}^{\pw}(t)$.  While the special fan functional is gone, we now have a term depending on bar recursion, which seems somewhat excessive for the modest theorem $\RIE^{\pw}$.        

\medskip

Secondly, Kohlenbach discusses the $\ECF$-translation for the base theory $\RCAo$ of higher-order Reverse Mathematics from \cite{kohlenbach2}.  
The original $\ECF$-translation may be found in \cite{troelstra1}*{\S2.6.5, p.\ 141} and is somewhat technical in nature.  
Intuitively, the $\ECF$-translation replaces all objects of type two or higher by \emph{type one associates}, also known as \emph{Reverse Mathematics codes}.  
In particular, as shown by Kohlenbach in \cite{kohlenbach2}*{\S3}, if $\RCAo\vdash A$, then $\RCA_{0}^{2}\vdash [A]_{\ECF}$, where $A$ is any formula 
in the language of finite types, $[A]_{\ECF}$ is the $\ECF$-translation of $A$, and $\RCA_{0}^{2}$ is essentially the base theory $\RCA_{0}$ of Reverse Mathematics.
Serendipitously, the $\ECF$-translation applied to the conclusion of Theorem \ref{kohort} gives rise to (a statement equivalent to) $\WKL\di \RIE^{\pw}$ (See \cite{dagsam}).  In other words, while $\STP$ translates to the (arguably strange) special fan functional, the latter becomes `plain old' $\WKL$ when applying the $\ECF$-translation! 
To use an imaginary proverb: All's standard that ends standard!    

\medskip

In conclusion, we have shown how results involving the (admittedly strange) special fan functional from the previous section give rise to `normal' results.  
The $\ECF$-approach is preferable over the one involving bar recursion, in our opinion.   

\subsubsection{Non-exotic computability results and \emph{Standard Part} II}\label{WARP}
As discussed after its introduction in Definition \ref{dodier}, the functional $\Theta$ can be viewed as a kind of realiser for the compactness of Cantor space.  
In Theorem \ref{nolapdog}, we show the equivalence between the existence of $\Theta$ and a version of the Heine-Borel compactness theorem \emph{in the general\footnote{The Heine-Borel theorem in second-order arithmetic (and hence Friedman-Simpson Reverse Mathematics) is `by definition' restricted to \emph{countable} covers (\cite{simpson2}*{IV.1}).} case}, i.e.\ the statement that any (possibly uncountable) 
open cover of the unit interval has a finite sub-cover.  Now, \emph{any} functional $\Psi^{2}$ gives rise to the `canonical' open cover $\cup_{x\in [0,1]}(x-\frac{1}{\Psi(x)+1}, x+\frac{1}{\Psi(x)+1})$ of the unit interval, and hence the (general) Heine-Borel theorem implies: 
\be\label{zosimpelistnie}\tag{$\HBU$}\textstyle
(\forall \Psi^{2})(\exists w^{1^{*}}){(\forall x^{1}\in [0,1])(\exists y\in w)(|x-y|<\frac{1}{\Psi(y)+1})}.
\ee
The naturalness of $\HBU$, as well its essential role in the development of the \emph{gauge integral}, is discussed in Remark \ref{KBU} below.  
We first have the following theorem.  
\begin{thm}\label{nolapdog}
The system $\RCAo+(\exists^{2})+\QFAC$ proves $(\exists \Theta)\SCF(\Theta)\asa \HBU$.  
\end{thm}
\begin{proof}
The theorem is implicit in the results from \cite{bennosam}, and is proved in \cite{dagsamII}.
\end{proof}
The base theory in Theorem \ref{nolapdog} is a $\Pi_{2}^{1}$-conservative extension of $\ACA_{0}$ by \cite{yamayamaharehare}*{Theorem 2.2}.
The equivalences $\STP\asa \HBU$ and $\STP\asa (\exists \Theta)\SCF(\Theta)$ are proved in a different conservative extension of $\ACA_{0}$ in \cite{bennosam}.  Hence, $\Theta$, $\STP$, and $\HBU$ are intrinsically linked, and we now 
discuss how the latter (and hence all of them) are essential for the development of the \emph{gauge integral}.
%
\begin{rem}[Mainstream mathematics and $\Theta$]\rm\label{KBU}
Perhaps surprisingly, $\HBU$ dates back around 125 years, and is \emph{essential} for the development of mathematical physics.
Furthermore, \emph{discontinuous} functions date back almost \emph{two hundred years}, and the combination of these two elements gives rise to $\Theta$, as discussed now.   

\medskip

First of all, $\HBU$ immediately follows from \emph{Cousin's lemma} (\cite{cousin1}*{p.\ 22}), which was proved before 1893 by Pierre Cousin, as discussed in \cite{dugac1}.
Hence, $\HBU$ essentially predates set theory, and as such counts as `non-set theoretical' or `ordinary mathematics', in the sense of RM as stipulated by Simpson in \cite{simpson2}*{I.1}.  

\medskip

Secondly, Cousin's lemma (and hence $\HBU$) is essential for the development of the \emph{gauge integral}, aka \emph{Henstock-Kurzweil integral}, as discussed in \cite{zwette}*{p.\ 5}.  
For instance, the \emph{uniqueness} of the gauge integral already implies $\HBU$; the existence of a function which is \emph{not} gauge integrable also implies $\HBU$.   
As to its history, the gauge integral was introduced around 1912 by Denjoy (in a different form) as a generalisation of the Lebesgue integral, which was in turn introduced around 1900.
Since Lebesgue integration is studied \emph{extensively} in RM (See \cite{simpson2}*{X.1} for an overview), gauge integrals would seem to be the \emph{obvious} kind of topic studied in RM as well.
Moreover, the gauge integral provides a formal framework for the Feynman path integral (See \cite{mullingitover, burkdegardener}), i.e.\ gauge integrals are also highly relevant in (the foundations of) physics.   

\medskip

Finally, as to discontinuous functions, Dirichlet introduces the characteristic function of the rationals around 1829 in \cite{didi1}, while thirty years later Riemann defines a function with countably many discontinuities via a series in his \emph{Habilitationsschrift} (See \cite{kleine}*{p.\ 115}).  As proved in \cite{kohlenbach2}*{\S3}, the existence of discontinuous functions on the reals gives rise to $(\exists^{2})$ (via so-called Grilliot's trick), and Theorem~\ref{nolapdog} thus suggests that $\Theta$ \emph{could} have been derived a century ago.   
\end{rem}
In conclusion, the naturalness of Cousin's lemma bestows the same status upon $\HBU$, $\Theta$, and $\STP$, which are even equivalent in the system from \cite{bennosam}.
%

\subsection{There and back again}\label{thereandback}
The examples of local constructivity discussed in Sections \ref{hum} to \ref{RMKE2} are \emph{by design} such that theorems in Nonstandard Analysis are converted to theorems of computable or constructive mathematics.  As suggested by the title, it is then a natural question whether the latter also again imply the former, i.e.\ do theorems of computable mathematics also (somehow) imply theorems of Nonstandard Analysis?   Are they (somehow) equivalent?

\medskip

At the theoretical level, the answer is a resounding `yes', as is clear from the following theorem which is the `inverse' of term extraction as in Theorem \ref{TERM2}.
\begin{thm}\label{karda}
Let $\varphi$ be internal and $t$ a term.  
The system $\P$ proves 
\be\label{eexa}
(\forall x)(\exists y\in t(x))\varphi(x, y)\di (\forall^{\st}x)(\exists^{\st}y)\varphi(x,y).
\ee 
Furthermore, if $\textsf{\textup{E-PA}}^{\omega}\vdash(\forall x)(\exists y\in t(x))\varphi(x, y)$, then $\P\vdash(\forall^{\st}x)(\exists^{\st}y)\varphi(x,y)$.  
\end{thm}
\begin{proof}
Both claims follow immediately by noting that every term $t$ is standard in $\P$ by the basic axioms in Definition \ref{debs}, and that the latter also implies that $t(x)$ is standard for standard $x$.
\end{proof}
At the practical level things are slightly more complicated:  While normal forms follow from computational statements like in \eqref{eexa}, these normals forms are often \emph{weaker} than the 
original nonstandard theorem.  Indeed, consider the proof of Theorem \ref{TE1} and Remark \ref{doeisnormaal} in which we \emph{weakened} the antecedents of resp.\ 
\eqref{strong} and \eqref{nora3} to resp.\ \eqref{weak} and \eqref{nora4}.  
We now obtain a version of $\RIE_{\ef}(t)$ from $\RIE_{\ns}$ from Section \ref{riekenaan} \emph{without} this weakening.     
\begin{thm}\label{varjou}
From $\P\vdash \RIE_{\ns} $, terms $i, o$ can be extracted s.t.\ $\textup{\textsf{E-PA}}^{\omega} $ proves:
\begin{align}\textstyle
\textstyle~(\forall f, ~&g ,k')\Big[(\forall k\leq i(g,k'))(\forall \textstyle x, y \in [0,1])(|x-y|<\frac{1}{g(k)} \di |f(x)-f(y)|\leq\frac{1}{k})\notag\\
&\label{FEST}\textstyle\di  (\forall \pi, \pi' \in P([0,1]))\big(\|\pi\|,\| \pi'\|< \frac{1}{o(g,k')}  \di |S_{\pi}(f)- S_{\pi}(f)|\leq \frac{1}{k'} \big)  \Big].
\end{align}
 Furthermore, $\P\vdash [\eqref{FEST}\di \RIE_{\ns}]$ and from $\textsf{\textup{E-PA}}^{\omega}\vdash \eqref{FEST}$, one obtains $\P\vdash \RIE_{\ns}$.
\end{thm} 
\begin{proof}
The second part follows from the proof of Theorem \ref{karda} as follows: For standard $g$ and $k'$, we note that $i(g,k')$ and $o(g,k')$ are also standard.   
Hence, we may strengthen the antecedent of \eqref{FEST} as follows:
\begin{align}\textstyle
\textstyle~(\forall f)&(\forall^{\st} g ,k')\Big[(\forall^{\st} k)(\forall \textstyle x, y \in [0,1])(|x-y|<\frac{1}{g(k)} \di |f(x)-f(y)|\leq\frac{1}{k})\notag\\
&\label{FEST2}\textstyle\di  (\forall \pi, \pi' \in P([0,1]))\big(\|\pi\|,\| \pi'\|< \frac{1}{o(g,k')}  \di |S_{\pi}(f)- S_{\pi}(f)|\leq \frac{1}{k'} \big)  \Big].
\end{align}
Note that the antecedent of \eqref{FEST2} does not depend on $k'$ and push inside this associated quantifier as follows:
\begin{align}\textstyle
\textstyle~(\forall f)&(\forall^{\st} g)\Big[(\forall^{\st} k)(\forall \textstyle x, y \in [0,1])(|x-y|<\frac{1}{g(k)} \di |f(x)-f(y)|\leq\frac{1}{k})\notag\\
&\notag\textstyle\di  (\forall^{\st}k')(\forall \pi, \pi' \in P([0,1]))\big(\|\pi\|,\| \pi'\|< \frac{1}{o(g,k')}  \di |S_{\pi}(f)- S_{\pi}(f)|\leq \frac{1}{k'} \big)  \Big].
\end{align}
In turn, we may strengthen `$\|\pi\|,\| \pi'\|< \frac{1}{o(g,k')}$' to `$\|\pi\|,\| \pi'\|\approx 0$' to obtain
\begin{align}\textstyle
\textstyle~(\forall f)&(\forall^{\st} g)\Big[(\forall^{\st} k)(\forall \textstyle x, y \in [0,1])(|x-y|<\frac{1}{g(k)} \di |f(x)-f(y)|\leq\frac{1}{k})\label{FESTJE}\\
&\notag\textstyle\di  (\forall^{\st}k')(\forall \pi, \pi' \in P([0,1]))\big(\|\pi\|,\| \pi'\|\approx 0 \di |S_{\pi}(f)- S_{\pi}(f)|\leq \frac{1}{k'} \big)  \Big].
\end{align}
As $k'$ (resp.\ $g$) only appears in the final formula of the consequent (resp.\ the antecedent), we may push the associated quantifier inside to obtain
\begin{align}\textstyle
\textstyle~(\forall f)&\Big[(\exists^{\st}g)(\forall^{\st} k)(\forall \textstyle x, y \in [0,1])(|x-y|<\frac{1}{g(k)} \di |f(x)-f(y)|\leq\frac{1}{k})\label{FESTJE2}\\
&\notag\textstyle\di (\forall \pi, \pi' \in P([0,1]))\big(\|\pi\|,\| \pi'\|\approx 0 \di  (\forall^{\st}k')(|S_{\pi}(f)- S_{\pi}(f)|\leq \frac{1}{k'}) \big)  \Big].
\end{align}
The consequent of \eqref{FESTJE2} is the definition of nonstandard integrability as in Definition \ref{kunko} while the antecedent implies 
nonstandard uniform continuity \eqref{soareyou4}.  Now, the latter also implies the antecedent of \eqref{FESTJE2}, as can be seen by applying $\HAC_{\INT}$ to the normal form \eqref{the} and taking 
the maximum of the resulting functional.   

\medskip

For the first part, recall that $\RIE_{\ns}$ implies \eqref{bikko}, which yields
\be\label{bikko2333}
(\forall f:\R\di \R)(\forall^{\st}g){\big[ (\forall^{\st}k)A(f, g(k), k)\di (\forall^{\st}k')(\exists^{\st}N')B(k', N', f)\big]}.
\ee
Bringing all standard quantifiers outside in \eqref{bikko2333}, we obtain 
\be\label{bikko3333}
(\forall^{\st}k',g)(\forall f:\R\di \R)(\exists^{\st}N', k){\big[ A(f, g(k), k)\di(\exists^{\st}N')B(k', N', f)\big]}.
\ee
Applying \emph{Idealisation} to \eqref{bikko3333} like in the proof of Theorem \ref{varou}, we obtain 
 \be\label{bikko3444444}
(\forall^{\st}k',g)(\exists^{\st}N', k)(\forall f:\R\di \R){\big[ A(f, g(k), k)\di(\exists^{\st}N')B(k', N', f)\big]}.
\ee
Applying Theorem \ref{TERM2} to `$\P\vdash \eqref{bikko3444444}$' (and taking the maximum of the resulting terms as usual), one immediately obtains \eqref{FEST}, as required.  
\end{proof}
\begin{cor}
The system $\P+(\exists^{\st}t)\REI_{\ef}(t)$ proves that a function with a \emph{standard} modulus of uniform continuity on $[0,1]$ is nonstandard integrable there.  
\end{cor}
\begin{proof}
Repeat the proof of the second part of the theorem with $(\exists^{\st}t)\REI_{\ef}(t)$ replaced by \eqref{FEST}.  
\end{proof}
By the previous, $\REI_{\ef}(t)$ gives rise to a nonstandard theorem strictly \emph{weaker} than $\REI_{\ns}$.  
In particular, it is easy to prove that the statement \emph{a standard function which is nonstandard-uniformly-continuous on the unit interval has a standard modulus of uniform continuity} implies $\paai$.
Similar results may now be obtained for \emph{any nonstandard theorem} (like e.g. $\MCT_{\ns}$).  
The only difference is the omission of the \emph{weakening} of the antecedent of \eqref{nora3} to \eqref{nora4} in Remark \ref{doeisnormaal}.  

\medskip

In conclusion, we have shown that the nonstandard $\RIE_{\ns}$ gives rise to the `highly constructive' theorem \eqref{FEST}.  
In turn, the latter gives rise to $\RIE_{\ns}$ as in Theorem~\ref{varjou}.  In general, \emph{every theorem} of pure Nonstandard Analysis is `meta-equivalent' as in Theorem \ref{varjou} to a `highly constructive' theorem like \eqref{FEST}.  Experience however bears out that this `meta-equivalence' is usually less elegant than the one from Theorem \ref{varjou}.  
 
\subsection{A template for the computational content of Nonstandard Analysis}\label{henk}
In this section, we introduce \emph{pure} Nonstandard Analysis, i.e.\ that part of the latter falling within the scope of Theorem \ref{TERM2}.   
We also formulate a template for obtaining computational content from theorems of pure Nonstandard Analysis.  

\bdefi[Pure Nonstandard Analysis]\label{pure}  
A theorem of \emph{pure} Nonstandard Analysis is built up as follows.   
\begin{enumerate}
\renewcommand{\theenumi}{\roman{enumi}} 
\item Only \emph{nonstandard definitions} (of continuity, compactness, \dots) are used; \textbf{no} epsilon-delta definitions are used.\label{controlisbetter}   
The former have (nice) normal forms and give rise to the associated constructive definitions from Figure \ref{ananlo3}.
\item Normal forms are closed under \emph{implication} by Remark \ref{doeisnormaal} and \emph{internal universal quantifiers} due to \emph{Idealisation} as in \eqref{criv}.  
\item Normal forms are closed under \emph{prefixing a quantifier over all infinitesimals} by Theorem \ref{hujiku}, i.e.\ if $\Phi(\eps)$ has a normal form, so does $(\forall \eps\approx 0)\Phi(\eps)$. 
\item Fragments of the axioms \emph{Transfer} and \emph{Standard Part} have normal forms.
\item Formulas involving the \emph{Loeb measure} (See Section \ref{loebM}) have normal forms.\label{horgi}
\end{enumerate}
\edefi
We stress that item \eqref{horgi} should be interpreted in a specific narrow technical sense, namely as discussed in Section \ref{loebM}.  
We ask the reader to defer judgement until the latter section has been consulted.   

\medskip
 
Finally, the following template provides a way of obtaining the computational content of theorems of pure Nonstandard Analysis.  
\begin{tempo}\rm
To obtain the computational content of a theorem of pure Nonstandard Analysis, perform the following steps.
\begin{enumerate}
\item[(i)] Convert all nonstandard definitions, formulas involving the Loeb measure, and axioms to normal forms using Figure \ref{ananlo3} on the next page.  
\item[(ii)] Starting with the most deeply nested implication, convert implications between normal forms into normal forms using Remark \ref{doeisnormaal}.   
\item[(ii$'$)] If `meta-equivalence' as in Section \ref{thereandback} is desired, omit the weakening from \eqref{nora2} to \eqref{nora4} in Remark \ref{doeisnormaal}.    
\item[(iii)] When encountering quantifiers over all infinitesimals, use Theorem \ref{hujiku} to obtain a normal form.  
When encountering internal universal quantifiers, apply \emph{Idealisation} as in \eqref{criv}.  
\item[(iv)] When a normal form has been obtained, apply Theorem \ref{TERM2} and modify the resulting term as necessary, usually by taking the maximum of the finite sequence at hand.  
\end{enumerate}
\end{tempo}
Now, Definition \ref{pure} is only an approximation of all theorems within the scope of Theorem \ref{TERM2}, but nonetheless covers large parts of Nonstandard Analysis.  
Regarding item \eqref{controlisbetter} and as noted in Example \ref{krel}, nonstandard definitions give rise to the associated constructive definition-with-a-modulus.  
The following list provides an overview for common notions\footnote{A space is \emph{F-compact} in NSA if there is a discrete grid which approximates every point of the space up to infinitesimal error, i.e.\ the intuitive notion of compactness from physics and engineering (See \cite{sambon}*{\S4}).\label{fookie}}, where $\Paai$ is \emph{Transfer} limited to $\Pi_{1}^{1}$-formulas (See \cite{sambon}*{\S4}) and $(\mu_{1})$ is the functional version of $\FIVE$ (See \cite{avi2}).     
\begin{figure}[h!]
\begin{center}
\begin{tabular}{|c|c|}
\hline
Nonstandard Analysis definition  & Constructive/functional definition \\
  \hline \hline
 nonstandard convergence & convergence with a modulus\\
   \hline  
 nonstandard continuity &continuity with a modulus\\
\hline
 nonstandard uniform continuity  &uniform continuity with a modulus\\
\hline
  F-compactness$^{\ref{fookie}}$ & totally boundedness \\
\hline

  nonstandard compactness of $[0,1]$ & The special fan functional $\Theta$ \\
\hline
  nonstandard differentiability as in \eqref{soareyou22} & differentiability with a modulus  \\
   & and the derivative given \\
\hline
   nonstandard Riemann integration & Riemann integration with a modulus \\\hline\hline
      $\paai$ & Feferman's mu-operator $(\mu^{2})$ \\\hline
      $\Paai$ & Feferman's second mu-operator $(\mu_{1})$ \\\hline
            $\STP$ & The special fan functional $\Theta$ \\
\hline
\end{tabular}
\end{center}
\caption{Nonstandard and constructive defitions}
\label{ananlo3}
\end{figure}~\\
We have now completed the main part of this paper, namely to establish the basic claims made in the introduction.  For the rest of this paper, we will consider related \emph{more advanced} results.  
We finish this section with the observation by Hao Wang that constructive mathematics may be viewed as a ``mathematics of doing" while
classical mathematics is a ``mathematics of being'' (See e.g.\ \cite{linkeboel}*{Preface}).  
Nonstandard Analysis then fits into this view as follows: In the second chapter of the \emph{Tao Te Ching}, the philosopher Lao Tse writes \emph{the sage acts by doing nothing}, 
and one could view the associated notion of \emph{wu wei} as the founding principle of Nonstandard Analysis.  

\section{A grab-bag of Nonstandard Analysis}\label{grab}
We discuss various advanced results regarding Nonstandard Analysis and its (local) constructivity.  

\subsection{The computational content of measure theory}\label{loebM}
We discuss the role NSA plays in the (computational) development of measure theory (inside $\P$).

\medskip

First of all, of central importance is the \emph{Loeb measure} (\cite{loeb1,nsawork2}) which is one of the cornerstones of {NSA} and provides a highly general approach to measure theory. The traditional development of the Loeb measure makes use of the non-constructive axiom \emph{Saturation} and \emph{external sets} (not present in $\IST$ pur sang).  
Furthermore, as for most fields of mathematics, NSA provides a highly elegant and conceptually simple development of known measures, such as the Lebesgue measure.  

\medskip

Secondly, as to prior art, a special case of the Loeb measure is introduced in a weak fragment of Robinsonian NSA in \cite{pimpson} \emph{using external sets}, but without the use of \emph{Saturation}; the definition of \emph{measure} from Reverse Mathematics (See \cite{simpson2}*{X.1} and below) is used.  On the internal side, the system $\P$ extended with the (countable) axiom \emph{Saturation} $\textsf{CSAT}$ can be given computational meaning; in particular Theorem~\ref{TERM2} holds for the extension $\P+\textsf{CSAT}$ (See \cites{brie3, espa}), but $t$ then involves \emph{bar recursion}, an advanced computation scheme which (somehow) embodies the computational content of the axiom of choice (\cite{kohlenbach3}).

\medskip

Now, we have established in the previous sections that \emph{pure Nonstandard Analysis} contains plenty of computational content \emph{not involving bar recursion}; 
We believe that the study of measure theory should be \emph{similarly elementary}, \emph{lest history repeat itself} in the form of incorrect claims that NSA is somehow non-constructive (\cite{samsynt}), this time around based on the use of bar recursion in the study of the Loeb measure.   

\medskip

Thirdly, classical measure theory seems hopelessly non-constructive at first glance: we know from Reverse Mathematics that the existence of the Lebesgue measure $\lambda(A)$ for open sets $A$ already gives rise to the Turing jump (\cite{simpson2}*{X.1}).  Indeed, the definition of $\lambda$ on $[0,1]$ in Reverse Mathematics is as follows:
\bdefi[Lebesgue measure $\lambda$]\label{kijk} For $\|g\|:=\int_{0}^{1}g(x)dx$, $U\subseteq [0,1] $, 
\[
\lambda(U):=\sup\{\|g\|: g\in C([0,1]) \wedge 0\leq g\leq1 \wedge (\forall x\in [0,1]\setminus U)(g(x)=0) \}.
\]
\edefi
The existence of this supremum cannot be proved in computable mathematics, like the base theory $\RCA_{0}$ of Reverse Mathematics from \cite{simpson2}*{I}, or constructive mathematics.  
However, there is a way around this `non-existence problem' pioneered in \cite{yussie, yuppie}: the formulas $\lambda(U)>_{\R}0$ and $\lambda(U)=_{\R}0$ \emph{make perfect\footnote{In particular, $\lambda(U)>_{\R}0$ is $(\exists g\in C([0,1]))( \|g\|>_{\R}0 \wedge 0\leq g\leq1 \wedge (\forall x\in [0,1]\setminus U)(g(x)=0))$, and similar for $\lambda(U)=_{\R}0$.} sense}, even if the supremum in Definition~\ref{kijk} does not exist.  Furthermore, most theorems of measure theory only involve the Lebesgue measure (and other measures) via such formulas.  In other words, \emph{even if the Lebesgue measure cannot be defined in general}, we can still write down most theorems of measure theory without any problems, even in the language of second-order arithmetic typical of Reverse Mathematics.     

\medskip

Fourth, the \emph{Loeb measure} $L$ is also defined as a supremum or infimum (See e.g.\ \cite{loeb1, nsawork2}), while the Lebesgue measure may be represented via a nonstandard measure $\lambda^{*}$, which is a supremum involving only the counting measure for (nonstandard) finite sets.  Similar to the Lebesgue measure, such suprema do not necessarily exist in weak systems of Robinsonian Nonstandard Analysis.   
However, formulas like `$L(A)\gg 0$' or `$L(A)\approx 0$' make perfect sense in such weak systems, in the same way as for the (usual) Lebesgue measure, as explored in \cite{pimpson,samnewarix}.  
Hence, most theorems involving the Loeb measure may be studied in this `Reverse Mathematics' way, even though the Loeb measure does not necessarily exist in general.  

\medskip

In light of the previous, we believe that measure theory may be developed as follows in $\P$ (See \cite{samnewarix} for a first case study):  For a measure $\kappa$ from NSA:
\begin{enumerate}
\item[(i)] The formula `$\kappa(A)\gg 0$' or `$\kappa(A)\approx 0$' makes perfect sense in $\P$, even when $A$ is a set definable by an external formula.  
\item[(ii)] If moreover the set $A$ is definable by a normal form, then `$\kappa(A)\gg 0$' or `$\kappa(A)\approx 0$' also has a normal form.   
\end{enumerate}
The first item implies that the \emph{possibly external} formulation of the measure $\kappa$ (e.g.\ using the standard part map) can be expressed as a formula in the language of $\P$.  
The second item means that if $(\forall z\in \R)(z\in A\asa (\forall^{\st}x)(\exists^{\st}y)\varphi(z,y,z))$ where $\varphi$ is internal, then the formula `$\kappa(A)\gg 0$' or `$\kappa(A)\approx 0$' has a normal form in $\P$.  

\medskip

In conclusion, we can study most theorems regarding the Loeb measure in the system $\P$, in much the same way as is done in Reverse Mathematics, by item~(i).  Furthermore, these theorems have lots of computational content by item (ii), in light of the generality of the class of normal forms as discussed in Section \ref{henk}.  

\subsection{Constructive features of Nonstandard Analysis}\label{wag}
In this section, we discuss some constructive features of Nonstandard Analysis.  
\subsubsection{Metamathematical properties}\label{wag1}
While there is no official formal definition of what makes a logical system `constructive', there are a number of metamathematical properties which are often considered to be hallmarks of intuitionistic/constructive theories.  Rathjen explores these properties in \cite{bestaat} for \emph{Constructive Zermelo-Fraenkel Set Theory} \textsf{CZF}, one of the main foundational systems for constructive mathematics.  We discuss in this section to what extent the system $\P$ and its extensions entertain nonstandard versions of the nine properties listed in \cite{bestaat}*{\S1}.     

\medskip

First of all, a system $S$ is said to have the \emph{existence property} \textbf{EP} if it satisfies the following:  Let $\varphi(x)$ be a formula with at most one free variable $x$; 
if $S\vdash (\exists x)\varphi(x)$ then there is some formula $\theta(x)$ with only $x$ free such that $S\vdash (\exists! x)(\theta(x)\wedge \varphi(x))$, where the `$!$' denotes unique existence.  
Now, $\P$ has the following \emph{nonstandard existence property}: If $\P\vdash (\exists^{\st}x^{\rho})\varphi(x)$ with $\varphi (x)$ internal and at most one free variable $x$, then there 
is an internal formula $\theta(x)$ such that $\P\vdash (\exists^{\st}! x^{\rho})(\theta(x)\wedge \varphi(x))$.  
Indeed, applying Theorem \ref{TERM2} to `$\P\vdash (\exists^{\st}x^{\rho})\varphi(x)$', we obtain a term $t^{\rho^{*}}$ such that $\textsf{E-PA}^{\omega}\vdash (\exists x^{\rho}\in t)\varphi(x)$ and
can define $\theta(x):=(\exists j<|t|) (x=_{\rho} t(j) \wedge (\forall i<j)\neg\varphi(t(i))$.  
Since $t$ (and $|t|$ and $t(j)$ for $j<|t|$) is standard in $\P$ by the basic axioms from Definition \ref{debs}, $\theta(x)$ implies that $\st(x)$.  
The previous also holds for any internal extension of $\P$, as for the below properties we consider now.  

\medskip

Secondly, a system $S$ is said to have the \emph{numerical existence property} \textbf{NEP} if it satisfies the following:  Let $\varphi(x)$ be a formula with at most one free variable $x$; 
if $S\vdash (\exists n^{0})\varphi(n)$ then there is some numeral $\overline{n}$ such that $S\vdash \varphi(\overline{n})$.  
Now, $\P$ has the following \emph{nonstandard numerical existence property}: If $\P\vdash (\exists^{\st}x^{0})\varphi(x)$ with $\varphi (x)$ internal and at most one free variable $x$, then there 
is a numeral $\overline{n}$ such that $\P\vdash (\exists x^{0}\leq \overline{n})\varphi(x)$.  The latter is proved in the same way as the previous property.  

\medskip

Thirdly, a system $S$ is said \emph{closed under Church's rule} \textbf{CR} if it satisfies the following:  Let $\varphi(x, y)$ be a formula with at most the free variables shown; 
if $S\vdash(\forall m^{0}) (\exists n^{0})\varphi(m,n)$ then there is some numeral $\overline{e}$ such that $S\vdash (\forall m^{0})\varphi(m, \{\overline{e}\}(m))$ where $\{e\}(n)$ is the value of the $e$-th Turing machine with input $n$.    
Now, $\P$ is closed under \emph{nonstandard Church's rule}: If $\P\vdash (\forall^{\st}m^{0})(\exists^{\st}n^{0})\varphi(m,n)$ with $\varphi (x,y)$ internal and at most the free variables shown, then there is a numeral $\overline{e}$ such that $\P\vdash (\forall m^{0})(\exists n^{0}\leq \{\overline{e}\}(n))\varphi(m,n)$.  
The latter is proved in the same way as the previous property; the term produced in this way is $\eps_{0}$-computable as it is a term definable in Peano arithmetic.  
A similar nonstandard version of \emph{the first variant of Church's rule} \textbf{CR$_{1}$} and \emph{Extended Church's rule} \textbf{ECR} can be obtained, assuming the antecedent in the latter has the form $(\forall^{\st}y)\varphi(y,x^{0})$ for internal $\varphi$.      

\medskip

On a related note, $\P$ and $\H$ prove the following nonstandard version of Church's thesis:  $(\forall^{\st}f^{1})(\exists e^{0})(\forall^{\st}x^{0})(f(x)=\varphi_{e}(x))$.  
This can be proved by applying \emph{Idealisation} to the trivial formula $(\forall^{\st}f^{1},x^{0})(\exists e^{0})(f(x)=\varphi_{e}(x))$, or simply noting that there are 
$N^{N}$ different (primitive recursive) functions mapping $\{0, 1, \dots, N-1\}$ to $\{0, 1, \dots, N-1\}$ (and zero otherwise).  Hence, nonstandard Church's thesis merely expresses that \emph{from the point of view of the nonstandard universe} all standard functions are computable when limited to the standard universe.     
  
\medskip

Fourth, since disjunction `$\vee$' gives rise to internal formulas, this symbol has no computational content (in contrast to `$(\exists^{\st})$').  As a result 
the \emph{Disjunction property} \textbf{DP} and \emph{Unzerlegbarkeits rule} \textbf{UZR} do not apply to $\P$.  The \emph{Uniformity rule} \textbf{UR} is similar to idealisation, but only with essential modification.  

\medskip

In conclusion, we again emphasise that we do \textbf{not} claim that $\P$ is somehow part of constructive mathematics; we merely point out that the standard universe of $\P$ satisfies versions of the typical properties of logical systems of constructive mathematics.  In other words, we again observe that $\P$ is `too constructive' to be called non-constructive (but does include $\LEM$ and is hence not constructive either).

\subsubsection{The axiom of extensionality}
We discuss the constructive status of the axiom of extensionality, and how this axiom relates to our system $\P$.  

\medskip

First of all, Martin-L\"of has proposed his \emph{intuitionistic type theory} as a foundation for (constructive) mathematics (\cite{loefafsteken}).  
These systems are close to computer programming and form the basis for various proof assistants, including Coq, Agda, and Nuprl (\cite{wefteling}).   
Martin L\"of's early type theories were \emph{extensional} (as is Nurpl), but the later systems were \emph{intensional} (as is Agda).  

\medskip

An advantage of intensional 
mathematics is the decidability of type checking, while a disadvantage is that every (even very obvious) equality has to be proved separately.    
The practice of Nuprl suggests however that checking whether an expression is correctly typed, is usually straightforward to perform by hand \emph{in practice}, i.e.\ decidable type checking is not a \emph{conditio sine qua non} when formalising mathematics.     
 
\medskip

Secondly, the system $\P$ of Nonstandard Analysis obviously includes the \emph{internal} axiom of extensionality.  
However, internal axioms are ignored by the term extraction algorithm in Theorem \ref{TERM2}.  
In particular, computational content is extracted from certain statements about the \emph{standard universe}, and it is a natural question whether the latter satisfies the `standard' axiom of extensionality \eqref{EXT}$^{\st}$ as follows:    
\[
(\forall^{\st}  x^{\rho},y^{\rho}, \varphi^{\rho\di\tau}) \big[x\approx_{\rho} y \di \varphi(x)\approx_{\tau}\varphi(y)   \big].
\]
where the notations are taken from Remark \ref{equ}.  Clearly, the previous sentence follows immediately from the (internal) axiom of extensionality by \emph{Transfer}, which is however absent in $\P$.  
In particular, the standard universe in $\P$ is actually highly \emph{intensional} in the sense  of the following theorem.
\begin{thm} Let $\Delta_{\INT}$ be any collection of internal formulas such that $\P+\Delta_{\INT}$ is consistent.  Then the latter cannot prove the following:
\be \label{EXT1337}
(\forall^{\st} F^{2}, f^{1}, g^{1})(f\approx_{1} g \di F(f)=_{0}F(g))
\ee
\end{thm}
\begin{proof}
Since $f\approx_{1}g$ is just $(\forall^{\st}N)(\overline{f}N=_{0}\overline{g}N)$, \eqref{EXT1337} implies the normal form
\be\label{krak}
(\forall^{\st} F^{2}, f^{1}, g^{1})(\exists^{\st}N)(\overline{f}N=_{0}\overline{g}N \di F(f)=_{0}F(g)).
\ee
A proof of \eqref{krak} in $\P+\Delta_{\INT}$ provides a term $t$ from G\"odel's \emph{T} which realises the axiom of extensionality for type two functionals (over $\textsf{E-PA}^{\omega}+\Delta_{\INT}$) as follows:
\be\label{kraksle}
(\forall  F^{2}, f^{1}, g^{1})(\exists N\leq t(F, f, g))(\overline{f}N=_{0}\overline{g}N \di F(f)=_{0}F(g)).
\ee
However, Howard has shown in \cite{howie} that a term as in \eqref{kraksle} does not exist.  
\end{proof}
Howard's results in \cite{howie} also imply that standard extensionality as in \eqref{EXT1337} 
for type three functionals cannot be proved in any extension of $\P$ which is conservative over \emph{Zermelo-Fraenkel set theory} $\ZF$.  Furthermore, we can classify \eqref{EXT1337} as follows, where $\TJ(\varphi, f)$ is $(\exists^{2})$ from Section \ref{RMKE} without the two outermost quantifiers.  
\begin{thm}\label{jayjo}
The system $\P$ proves $\paai\asa [\eqref{EXT1337} ~\wedge(\exists^{\st}\varphi^{2})(\forall^{\st}f^{1})\TJ(\varphi,f)]$.
\end{thm}
\begin{proof}
For the forward implication, $\paai$ applied to the internal axiom of extensionality for standard $F^{2},f^{1}, g^{1}$ immediately yields \eqref{krak}.  
Also, $\paai$ implies \eqref{veil}, and latter yields the second conjunct in the right-hand side of the equivalence.  
For the reverse implication, suppose $\paai$ is false, i.e.\ there is standard $h^{1}$ such that $(\forall^{\st}n^{0})(h(n)=1)\wedge (\exists m^{0})(h(m)= 0)$.  
Clearly, $h\approx_{1} 11\dots$ while for $\varphi$ as in $(\forall^{\st}f^{1})\TJ(\varphi,f)$ we have $\varphi(h)= 0 \ne\varphi(11\dots)$, contradicting extensionality as in \eqref{EXT1337}, and $\paai$ follows.    
\end{proof}
In conclusion, the standard universe of $\P$ is highly intensional in the sense of the previous theorem.  Thus, if one believes that intensionality is essential/important/\dots for the constructive 
nature of a logical system, then the above results provide a partial explanation why $\P$ has so much computational content.  
\subsubsection{Decomposing the continuum}
In this section, we discuss the \emph{indecomposable} nature of the \emph{continuum} in intuitionism. 
We will show that the continuum in Nonstandard Analysis is \textbf{in}decomposable \emph{from the point of view of the standard universe}, and decomposable \emph{from the point of view of the nonstandard universe}.       

\medskip

The \emph{continuum}, i.e.\ the set of real numbers $\R$, is \emph{indecomposable} in intuitionistic mathematics, which means that if $\R=A\cup B$ and $A\cap B=\emptyset$, then $A=\R$ or $B=\R$.  
Brouwer first published this result in \cite{pruisje} (See also \cite{vajuju}*{p.\ 490}) and a modern treatment by van Dalen may be found in \cite{dalencont}.  
This in-decomposability result of course hinges on the absence of the law of excluded middle in intuitionistic mathematics, as e.g.\ $(\forall x\in \R)(x>0 \vee x\leq 0)$ allows us to split $\R$ as $\{x\in \R: x>0 \}\cup \{x\in \R: x\leq 0\}$.  
An intuitive description of the above is as follows:
\begin{quote}
In intuitionistic mathematics the situation is different; the continuum has, as it were, a syrupy nature, one cannot simply take away one point.  
In the classical continuum one can, thanks to the principle of the excluded third, do so. To put it picturesquely, the classical continuum is the frozen intuitionistic continuum. If one removes one point from the intuitionistic continuum, there still are all those points for which it is unknown whether or not they belong to the remaining part. (\cite{dalencont}*{p.\ 1147})
\end{quote}
We now show that the continuum in Nonstandard Analysis is indecomposable in one sense, and decomposable in another sense.  
To this end, let $N\in \N$ be nonstandard and let $[x](N)$ be the $N$-th rational approximation of $x$. Since inequality on $\Q$ is decidable (as opposed to the situation for `$<_{\R}$'), we have
\be\label{D}\tag{$\textsf{D}$}
\R=\{x\in \R : [x](N)>_{\Q}0\}\cup \{x\in \R: [x](N)\leq_{\Q} 0\},
\ee
where the first set is denoted $A_{N}$ and the second one $B_{N}$.
While \eqref{D} provides a (decidable) decomposition of $\R$, there are problems: the decomposition is \emph{fundamentally arbitrary} and based on a \emph{nonstandard} algorithm in an essential way. 

\medskip

First, as to the arbitrariness of \eqref{D}, it is intuitively clear that for $x\approx 0$ we can have $x\in A_{N}$ but $x\in B_{M}$ for nonstandard $N\ne M$, i.e.\ the decomposition \eqref{D} seems \emph{non-canonical} or \emph{arbitrary} as it depends on the choice of nonstandard number.  
To see that this arbitrariness is fundamental, consider \eqref{abs} which expresses that \eqref{D} is \emph{not} arbitrary in the aforementioned sense for standard reals:
\be\label{abs}
(\forall^{\st}x\in \R)(\forall N,M\in \N)\big((\neg\st(N)\wedge \neg\st(M))\di x\in A_{N}\asa x\in A_{M}\big). 
\ee
By the following theorem, \eqref{abs} is \emph{Transfer} in disguise, and hence fundamentally non-constructive by Corollary \ref{sefcor}.  As a consequence, the decomposition \eqref{D} is always fundamentally arbitrary in any system where the Turing jump as in $(\mu^{2})$ is absent, like e.g.\ in constructive mathematics or $\P$.     
\begin{thm}\label{kerfoot}
The system $\P$ proves $\paai\asa \eqref{abs}$.  
\end{thm}
\begin{proof}
For the forward direction, note that for a standard real $x$, $[x](N)>0$ implies $(\exists^{\st}n)([x](n)>0)$ by \emph{Transfer}, and \eqref{abs} is immediate. 
For the reverse direction, suppose $\neg\paai$, i.e.\ there is standard $h^{1}$ such that $(\forall^{\st}n^{0})(h(n)=0)\wedge (\exists m^{0})(h(m)\ne 0)$, and define the real $x_{0}$ as: $[x_{0}](k)=0$ if $(\forall i\leq k)(h(i)=0)$, and $\frac{1}{2^{i_{0}}}$ if $i_{0}$ is the least $i\leq k$ such that $h(i)\ne 0$.  Clearly, $x\in B_{N}$ for small enough nonstandard $N$, and $x\in A_{N}$ for large enough nonstandard $N$.  This contradiction with \eqref{abs} yields the reverse direction.    
\end{proof}
The real $x_{0}$ defined in the proof of the theorem is an example of the aforementioned `syrupy' nature of the continuum: this real (and many like it) can be zero or positive in different models of $\P$, depending on whether \emph{Transfer} holds.  

\medskip

Secondly, as to the `nonstandardness' of \eqref{D}, there clearly exists $F^{2}$ with
\be\label{flahik}
(\forall x\in \R)(F(x)=0 \di x\in A_{N}\wedge F(x)=1\di x\in B_{N}), 
\ee
and $F$ is even given by an algorithm (involving nonstandard numbers).  
However, by Theorem \ref{TERM2}, we can only obtain computational information from \emph{standard objects}, and it is thus a natural question if there is \emph{standard} $F$ as in \eqref{flahik}.  
Now, from the existence of such $F$, 
one readily derives $\WKL^{\st}$, i.e.\ \emph{weak K\"onig's lemma} (\cite{simpson2}*{IV}) relative to `st', which is fundamentally non-constructive and cannot be derived in $\P$.  
The same derivation goes through for $A_{N},B_{N}$ replaced by \emph{any} internal $A,B$.  
Thus, while \eqref{D} constitutes a \emph{decidable} decomposition of $\R$, the associated decision procedure is \emph{fundamentally nonstandard}.

\medskip

In conclusion, while $\R$ has a decidable decomposition \eqref{D} in the system $\P$, the decomposition is \emph{fundamentally arbitrary} and based on a \emph{nonstandard} algorithm in an essential way.  If one requires a decomposition be \emph{canonical} or to have \emph{computational content} (and therefore be standard), then \eqref{D} is disqualified.    

\medskip

Let us rephrase this situation as follows: the continuum in Nonstandard Analysis is indecomposable \emph{from the point of view of the standard universe}, and decomposable \emph{from the point of view of the nonstandard universe}.  Since we can only extract computational information from the standard universe, our results are in line with intuitionistic mathematics.

\subsubsection{Tennenbaum's theorem}
We discuss the (non-)constructive nature of Nonstandard Analysis in light of \emph{Tennenbaum's theorem} (See e.g.\ \cite[\S11.3]{kaye}).    

\medskip

As noted in Section \ref{crackp}, even fragments of Robinson's Nonstandard Analysis based on arithmetic seem fundamentally non-constructive.  
Indeed, \emph{Tennenbaum's theorem} as formulated in \cite[\S11.3]{kaye} `literally' states that any nonstandard model of Peano Arithmetic is not computable.  

\medskip

\emph{What is meant} is that for a nonstandard model $\M$ of Peano Arithmetic, the operations $+_{\M}$ and $\times_{\M}$ cannot be computably defined in terms of the operations $+_{\N}$ and $\times_{\N}$ of the standard model $\N$ of Peano Arithmetic.  Similar results exist for fragments (\cite{kaye}*{\S11.8}) and Nonstandard Analysis thus \emph{seems} fundamentally non-constructive even at the level of basic arithmetic.    

\medskip
 
Now, while certain nonstandard models indeed require non-constructive tools to build, models are not part of Nelson's \emph{axiomatic approach} to Nonstandard Analysis via $\IST$.  What is more, the latter explicitly disallows the formation of external sets like `\emph{the operation $+$ restricted to the standard numbers}'.  
Nelson specifically calls attention to this rule on the first page of \cite{wownelly} introducing $\IST$:  
\begin{quote} 
\emph{We may not use external predicates to define subsets}. We call the violation of this rule illegal set formation. (Emphasis in original)
\end{quote}
To be absolutely clear, one of the fundamental components of Tennenbaum's theorem, namely the \emph{external set} `$+$ restricted to the standard naturals' is missing from \emph{internal} set theory $\IST$, as the latter (what's in a name?) exclusively deals with \emph{internal} sets.  Thus, we may claim that Tennenbaum's theorem is merely an artefact of the 
model-theoretic approach to Nonstandard Analysis.  

\medskip

Finally, Connes' critique of Nonstandard Analysis mentioned in Section \ref{crackp} seems to be based on similarly incorrect observations, namely that the models generally used in Robinsonian Nonstandard Analysis are fundamentally non-constructive and \emph{therefore} Nonstandard Analysis is too.  
By way of an analogy:  non-constructive mathematics (including models) is used in physics all the time, but does that imply that physical reality is therefore non-constructive (in the sense that there exist non-computable objects out there)?

\subsubsection{Bishop's numerical implication}
We discuss the (intimate) connection between Bishop's \emph{numerical implication} from \cite{nukino}*{p.\ 60} and some of the axioms of $\H$ from Definition \ref{flah}.  
The starting point is Bishop's claim from \cite{nukino}*{p.\ 56-57} that the numerical meaning of the BHK-implication is unclear in general:
\begin{quote}  
The most urgent foundational problem of constructive mathematics concerns the numerical meaning of implication. [\dots] the numerical meaning of [BHK] implication is a priori
unclear [\dots] 
\end{quote}
Thus, \emph{numerical implication} (aka G\"odel implication) is an alternative, based on G\"odel's Dialectica interpretation (\cite{godel3}), to the usual implication from intuitionistic logic with the associated BHK interpretation.
Bishop points out that numerical implication and BHK-implication amount to the same thing \emph{in practice}, and even derives the former from the latter, however using 
two \emph{non-constructive} steps (See \cite{nukino}*{p.\ 56-60}).  In particular, \cite{nukino}*{$(4)\di (5)$} is an instance of the \emph{independence of premises principle}, 
while \cite{nukino}*{$(7) \di (12)$} is a generalisation of \emph{Markov's principle}.  

\medskip

Now, the term extraction property of $\P$ and $\H$ is based on a nonstandard version of the Dialectica interpretation, called the \emph{nonstandard Dialectica interpretation} $D_{\st}$ (See \cite{brie}*{\S5}).  Furthermore, the system $\H$ includes (See Definition \ref{flah}) the axiom \textsf{HIP}$_{\forall^{\st}}$, a version of the \emph{independence of premises principle}, and the axiom 
\textsf{HGMP}$^{\st}$, a generalisation of  \emph{Markov's principle}.  By the following theorem, the latter two principles are exactly the ones needed to show (inside $\H$) that the class of normal forms is closed under implication, similar to Bishop's derivation of numerical implication from BHK-implication in \cite{nukino}.  
\begin{thm}\label{nogwelconsenal}
The system $\H$ proves that a normal form can be derived from an implication between normal forms. 
\end{thm}
\begin{proof}
Similar to the derivation in Remark \ref{doeisnormaal}, we work in $\H$ and consider
\be\label{norax}
(\forall^{\st}x)(\exists^{\st}y)\varphi(x, y)\di (\forall^{\st}z)(\exists^{\st}w)\psi(z, w), 
\ee
where $\varphi, \psi$ are internal.  
Since standard functionals yield standard output from standard input, \eqref{norax} implies $(\forall^{\st}\zeta)\big[(\forall^{\st}x)\varphi(x, \zeta(x))\di (\forall^{\st}z)(\exists^{\st}w)\psi(z, w)\big]$.  
We can bring outside the quantifier over $z$ (even in $\H$) as follows:
\be\label{nora3x}
(\forall^{\st}\zeta, z)\big[(\forall^{\st}x)\varphi(x, \zeta(x))\di (\exists^{\st}w)\psi(z, w)\big],
\ee
and applying $\HIP_{\forall^{\st}}$ to the formula in square brackets yields:  
\be\label{nora4x}
(\forall^{\st}\zeta, z)(\exists^{\st} W)\big[(\forall^{\st}x)\varphi(x, \zeta(x))\di (\exists w\in W)\psi(z, w)\big],
\ee
which in turn has exactly the right form to apply \textsf{HGMP}$^{\st}$, yielding   
\be\label{nora5x}
(\forall^{\st}\zeta, z)(\exists^{\st} V, W)\big[(\forall x\in V)\varphi(x, \zeta(x))\di (\exists w\in W)\psi(z, w)\big],
\ee
which is a normal form since the formula in square brackets in \eqref{nora5x} is internal.  
\end{proof}
In conclusion, the previous points to similarities between Bishop's numerical implication and Nonstandard Analysis, but we do not have any deeper insights beyond the observed analogies.  

\subsubsection{The meaning of Nonstandard Analysis}
In this section, we discuss a possible interpretation of the predicate `$x$ is standard' from Nonstandard Analysis.  
In particular, we point out similarities between this predicate and predicates pertaining to computational efficiency in constructive mathematics.  

\medskip

First of all, set theory for instance typically does \textbf{not} deal with questions like `What is a set?' or the meaning of the symbol `$\in$' for elementhood.   
These questions do come up, especially in the philosophy of mathematics, but are mostly absent from the mathematics itself.  
The same holds for the `\st' predicate in Nonstandard Analysis (and e.g.\ infinitesimals), as summarised by Nelson as follows:
\begin{quote}
In addition to the usual undefined binary predicate $\in$ of set theory we adjoin a new undefined unary predicate \emph{standard}. [\dots] 
To assert that $x$ is a standard set has no meaning within conventional mathematics-it is a new undefined notion. (\cite{wownelly}*{p.\ 1165})
\end{quote}
Nonetheless, experience bears out that many newcomers to Nonstandard Analysis do ask questions regarding the meaning of `st' and infinitesimals.  
While the author mostly agrees with Nelson's quote, we did formulate an intuitive way of understanding infinitesimals in Section \ref{fraki}, namely as an elegant shorthand for computational content like moduli.  
Paying homage to the imaginary proverb `in for an infinitesimal, in for the entire framework', we now accommodate the aforementioned newcomers and indulge in the quest for the meaning of the new predicate `st'.   

\medskip

Secondly, by the nature of the BHK-interpretation, all connectives have computational content in constructive mathematics.  
However, for a given theorem stating the existence of $x$, it is possible one does not need \emph{all} this computational content to compute $x$, i.e.\ some of this computational content is superfluous.  By way of an example, one \emph{only} needs a modulus of uniform continuity of $f$ to compute the modulus of integration of $f$ in $\RIE_{\ef}(t)$ from Section \ref{riekenaan}, i.e.\ one does not need $f$ itself.  Hence, if one is interested in \emph{efficient computation}, it makes sense to ignore all superfluous computational content, and develop a mechanism which can perform this task (by hand or automatically).  

\medskip

Now, the previous considerations are the motivation for Berger's \emph{Uniform Heyting arithmetic} (\cite{uhberger}) in which quantifiers `$\exists^{\textup{c}}$' and `$\exists^{\textup{nc}}$' are introduced which may be read as `computationally relevant' and `computationally irrelevant' existence.  
The computational content of connectives decorated with `nc' is ignored, leading to more efficient algorithms.    
The proof-assistant \emph{Minlog} (\cite{minlog}) even sports a \emph{decoration algorithm} to automatically 
decorate a proof with the `c' and `nc' predicates.   
Furthermore, the proof-assistant \emph{Agda} (\cite{agda}) has the `dot notation' which has the same functionality as the `nc' predicate, while 
\emph{homotopy type theory} (\cite{hottbook}) includes `truncated existence' $\|\Sigma\|$ also similar to `$\exists^{\textup{nc}}$'.    
Finally, aspects of the previous may already be found in the work of Lifschitz (\cite{lifken}).    

\medskip  

As discussed in \cite{sambon}, \cite{brie}*{p.\ 1963}, and \cite{benno2}*{\S4}, the predicate `st' in $\P$ seems to behave very similarly to the predicate `c' for `computational relevance', while the role of `$\forall^{\textup{nc}}$' from Minlog is played by `$\forall$' in Nonstandard Analysis.  
By way of an example, consider $\RIE_{\ns}$ from Section \ref{riekenaan} and observe that in the latter the quantifier `$(\forall f)$' is present, rather than `$(\forall^{\st}f)$'.  As a result, 
the extracted term $t$ in $\RIE_{\ef}(t)$ does \textbf{not} depend on $f$, i.e.\ `$(\forall f)$' is not computationally relevant, while $(\forall^{\st}f)$ would be.  
Furthermore, the connectives `$\vee$' and `$\di$' retain their usual classical/non-constructive behaviour in $\P$, which corresponds to $\vee^{\textup{nc}}$ and $\di^{\textup{nc}}$ in Minlog.  
Thus, we obtain an \textbf{informal} analogy between constructive mathematics and classical Nonstandard Analysis (without \emph{Transfer}) as summarised in the following table.  We cannot stress the \emph{heuristic} nature of this comparison enough.    
\begin{figure}[h]
\begin{center}
\begin{tabular}{|c|c|c|}
\hline
 Homotopy & Minlog /  & Nonstandard Analysis\\
Type Theory&  Uniform HA &  as in system $\P$   \\
  \hline \hline
$\Sigma$  & $\exists^{\textup{c}}$ & $\exists^{\st}$\\
\hline
$\Pi$  & $\forall^{\textup{c}}$ & $\forall^{\st}$\\
\hline
$\|\Sigma\|$  & $\exists^{\textup{nc}} $ & $\exists$\\
\hline
~ & $\forall^{\textup{nc}} $ & $\forall$\\\hline
~ & $\vee^{\textup{nc}}$ & $\vee$ \\\hline
~ & $\di^{\textup{nc}}$ & $\di$ \\\hline
\end{tabular}
\end{center}
\end{figure}~\\
We again stress that we do not claim $\P$ to be a constructive system, but we do point out that $\P$ sports too many of properties of constructive mathematics to be dismissed as non-constructive.  Therefore, it occupies the twilight zone between the constructive and non-constructive.      

\subsubsection{Efficient algorithms and Nonstandard Analysis}
While most of this paper is foundational in nature, we now argue that Theorem \ref{TERM2} holds the promise of providing \emph{efficient} terms via term extraction applied to formalised proofs.  
The implementation in the proof assistant Agda of the term extraction algorithm of Theorem \ref{TERM2} is underway by the author and Chuangjie Xu in \cite{EXCESS}.  
The $D_{\st}$ interpretation from \cite{brie} has been implemented and tested on basic examples.     

\medskip

First of all, Nonstandard Analysis generally involves \emph{very short} proofs (compared to proofs in mainstream mathematics).  Short proofs are one heuristic (among others) for extracting \emph{efficient} terms in \emph{proof mining} (\cite{kohlenbach4}).  

\medskip

Secondly, as noted in Theorem \ref{TERM2}, \emph{internal axioms do not contribute to the extracted term}.  In particular, the extracted terms are determined \emph{only} by the external axioms of $\P$; \emph{all axioms of `usual' mathematics, including mathematical induction, do not influence the extracted term}.  Note that \emph{external induction} \textsf{IA}$^{\st}$ (See Definition \ref{debs}) is rarely used in the practice of Nonstandard Analysis.    
 
\medskip

Thirdly, as can be gleaned from the proof of \cite{brie}*{Theorem 5.5}, terms extracted from external axioms of $\H$ and $\P$ \emph{other than external induction} are of trivial complexity, except for \textbf{one} relating to the $(\forall^{\st})$-quantifier (See \cite{brie}*{p.\ 1982}) which involves sequence concatenation.  Furthermore, one rarely uses external induction in practice while certain\footnote{As explored in \cite{EXCESS}, applying Theorem \ref{TERM} to \emph{if $m\leq_{0} n$ for standard $n$, then $m$ is standard} yields a function $g^{0\di 0^{*}}$ such that $g(n)=\langle0, 1, \dots, n\rangle$.  The latter function is standard as the recursor constants and $0$ and $+$ are, and thus yields a proof of the original nonstandard statement.} `basic' applications of external induction yield low-complexity terms too.

\medskip

In conclusion, the three previous observations hold the promise of providing \emph{efficient} terms via term extraction applied to formalised proofs as in Theorem \ref{TERM2}.  

\subsection{Constructive Nonstandard Analysis versus local constructivity}\label{palmke}
While most of this paper deals with \emph{classical} Nonstandard Analysis, we now discuss \emph{constructive} Nonstandard Analysis and its connection to local constructivity.  
We already have the system $\H$ as an example of the syntactic approach to constructive Nonstandard Analysis, and we now discuss the \emph{semantic approach}.  

\medskip

Intuitively, the \emph{semantic} approach to Nonstandard Analysis pioneered by Robinson (\cite{robinson1}) consists in somehow building a nonstandard model of a given structure (say the set of real numbers $\R$) and proving that the original structure is a strict subset of the nonstandard model (usually called the set of \emph{hyperreal numbers} $^{*}\R$) while establishing properties similar to \emph{Transfer}, \emph{Idealisation} and \emph{Standardisation} as theorems of this model and the original structure.  Historically, Nelson and Hrbacek of course studied Robinson's work and independently axiomatised the semantic approach in their logical systems (\cites{wownelly,hrbacek2}).  The most common way of building a suitable nonstandard model is using a \emph{free ultrafilter} (See e.g.\ \cite{loeb1,nsawork2}).  The existence of the latter is a rather strong \emph{non-constructive} assumption, slightly weaker than the axiom of choice of $\ZFC$.            

\medskip

As it turns out, building nonstandard models with properties like \emph{Transfer} can also be done constructively:
Palmgren in \cite{opalm}*{Section 2} and \cite{nostpalm} constructs a
nonstandard model $\M$ (also called a `sheaf' model) satisfying the \emph{Extended Transfer Principle} by \cite{opalm}*{Corollary 4 and Theorem 5} (See also \cite{palmdijk}).   
As noted by Palmgren (\cite{opalm}*{p.\ 235}), the construction of $\M$ can be formalised in 
Martin-L\"of's constructive type theory (\cite{loefafsteken}).  The latter was developed independently of Bishop's constructive mathematics, but can be viewed 
as a foundation of the latter.  

\medskip

It should be clear by now that there is a fundamental difference between our approach and that of Palmgren and Moerdijk:  
The latter attempt to mimic Robinson's approach as much as is possible inside a framework for constructive mathematics, while we attempt to bring out 
the (apparently copious) computational content of \emph{classical Nonstandard Analysis itself} using the system $\P$ and its extensions.       
The constructivist Stolzenberg has qualified the former `mimicking' approach as \emph{parasitic} (or \emph{scavenger}) as follows, though we do not share his view.  
\begin{quote}
   Nowadays, what is called ``constructive'' mathematical practice
consists in taking a classical theorem or theory that is a product
of ordinary classical mathematical practice and trying to produce a
``good'' constructive counterpart to it.  Obviously, this enterprise
is parasitic on the theorems and theories of ordinary classical
practice. (\cite{bloedstollend})
\end{quote}
We only mention Stolzenberg's view as a means of launching a discussion concerning which place our results on Nonstandard Analysis have in the (apparently rather emotionally charged) constructive pantheon.  

\medskip

Next, we establish that our approach cannot easily be reconciled with the approach by Palmgren and Moerdijk.  
To this end, we show that principles (built into the Palmgren-Moerdijk framework) connecting epsilon-delta and nonstandard definitions  give 
rise to non-constructive oracles when considered in our framework.  To this end, let $\NSD$ be the statement \emph{a standard $f:\R\di \R$ differentiable at zero is also \emph{nonstandard differentiable} there}, where the latter is as follows.  
\bdefi A function $f$ is \emph{nonstandard differentiable} at $a$ if
\begin{equation}\label{soareyou22}\textstyle
(\forall \eps, \eps' \ne0)\big(\eps, \eps'\approx 0 \di \frac{f(a+\eps)-f(a)}{\eps}\approx \frac{f(a+\eps')-f(a)}{\eps'}\big),
\end{equation}
\edefi
Now, $\NSD$ is a theorem of $\IST$ but we also have the following implication.  
\begin{thm}\label{hark2} 
The system $\P+\NSD$ proves $\paai$.  
\end{thm}
\begin{proof}
Working in $\P+\NSD$, suppose $\neg\paai$, i.e.\ there is standard $h_{0}^{1}$ such that $(\forall^{\st}n)h_{0}(n)=0$ and $(\exists m_{0})h(m_{0})\ne0$.  
Define the standard real $x_{0}$ as $\sum_{n=0}^{\infty}\frac{h(n)}{2^{n}}$.  Since $0\approx x_{0}>_{\R}0$ the standard function $f_{0}(x):=\sin(\frac{x}{x_{0}})$ is clearly well-defined and differentiable in the usual internal `epsilon-delta' sense.  However, 
\[\textstyle
\frac{f_{0}(x_{0})-f_{0}(0)}{x_{0}}=_{\R}\frac{\sin(1)}{x_{0}} \not \approx \frac{2}{\pi x_{0}}=_{\R} \frac{f_{0}(\frac{\pi}{2}x_{0})-f_{0}(0)}{\frac{\pi}{2}x_{0}}
\]
which implies that $f_{0}$ is not nonstandard differentiable at zero.  
This contradiction yields $\paai$, and we are done.        
\end{proof}
Let $\DIF(\Xi)$ be the statement that if a function $f$ is differentiable at zero, then $\Xi(f)$ is a modulus of differentiability for $f$ as zero.  
\begin{cor}\label{idare}
From the proof that $\P\vdash \NSD\di \paai$, a term $t$ can be extracted such that $\textsf{\textup{E-PA}}^{\omega}\vdash (\forall \Xi^{3})\big(\textsf{\textup{DIF}}(\Xi)\di \MU(t(\Xi)))$
\end{cor}
\begin{proof}
A normal form for \eqref{soareyou22} is easy to obtain and as follows:  
\[\textstyle
(\forall^{\st}k^{0})(\exists^{\st} N^{0})(\forall \eps, \eps' \ne0)\big(|\eps|, |\eps'| <\frac{1}{N} \di \left|\frac{f(a+\eps)-f(a)}{\eps}- \frac{f(a+\eps')-f(a)}{\eps'}\right|<\frac{1}{k}\big),
\]
The proof is straightforward and analogous to the proof of equation \eqref{froodcor}.  
\end{proof}
Note that it is easy to define the derivate from the previous normal form of nonstandard differentiability by using $\HAC_{\INT}$ to obtain a `modulus of differentiability'.  

\medskip

Finally, as part of constructive approaches to Nonstandard Analysis, we should point out the existence of \emph{synthetic differential geometry} (SDG) which constitutes an attempt at formalising the intuitive infinitesimal calculus based on intuitionistic logic and \emph{nilpotent} infinitesimals.  We note that SDG is \emph{inconsistent} with classical mathematics and refer to \cite{evenbellen} for a detailed discussion and references.      
We stress that Giordano has formulated an alternative to SDG in classical NSA as follows:
\begin{quote}
In this work we want to modify the [Kock-Lawvere axiom] so as to obtain a
theory with final results similar to SDG's ones, but in a classical context
and without great logical problems or difficult models. (\cite{gioa}*{p.\ 76})
\end{quote}
Hence, classical NSA can deal perfectly well with nilpotent infinitesimals.  The latter are present in the original works of e.g.\ Leibniz, Lie, and Cartan.  

\subsection{Classical computability theory}\label{forgo}
In the previous sections, we established that theorems in Nonstandard Analysis give rise to results in computability theory like \eqref{froodcor} or Theorem \ref{varou}.  
Now, a distinction\footnote{The distinction between `higher-order' and `classical' computability theory is not completely strict:  Continuous functions on the real numbers are represented in second-order arithmetic by \emph{type one} associates (aka Reverse Mathematics codes), as discussed in \cite{kohlenbach4}*{\S4}.} exists between \emph{higher-order} and \emph{classical} computability theory.  The latter (resp.\ the former) deals with computability on objects of type zero and one (resp.\ of any type) and it is clear that our above results are part of \emph{higher-order} computability theory.  In this section, we provide a generic example of how NSA gives rise to results in \emph{classical} computability theory.  More examples (e.g.\ Ramsey's theorem and $\MCT$) are treated in \cite{sambon2}.    

\medskip

In this section, we study \emph{K\"onig's lemma} (\textsf{KOE} for short), which is the statement that every finitely branching infinite tree has a path; 
$\KOE$ is equivalent to $\ACA_{0}$ over $\RCA_{0}$ by \cite{simpson2}*{III.7}.  Let $\KOE_{\ns}$ be the statement that every \emph{standard} tree as in $\KOE$ has a \emph{standard} path.  As expected, we shall prove $\paai\asa \KOE_{\ns}$ over $\P$ (Theorem \ref{serfu}), 
and obtain results in classical and higher-order computability theory from this equivalence (Theorems \ref{serfi} and \ref{serf2}).  
\begin{thm}\label{serfu} 
The system $\P$ proves $\paai\asa \KOE_{\ns}$.
\end{thm}
\begin{proof}
As noted before, we actually have an equivalence in \eqref{veil}, i.e.\ the standard functional $\mu^{2}$ as in the latter allows us to decide existential formulas as long as a standard function describes the quantifier-free part as in the antecedent of \eqref{veil}.  
Now, if a standard tree $T^{1}$ is infinite and finitely branching, there is $n^{0}$ such that $(\forall m^{0})(\exists \beta^{0^{*}})(\langle n\rangle *\beta \in T\wedge |\beta|=m)$.  Since $T$ is standard, the latter universal formula can be decided using $\mu^{2}$ from \eqref{veil}.  Say we have 
\[
(\forall^{\st}n)\big[g(n)=0\asa (\forall m^{0})(\exists \beta^{0^{*}})(\langle n\rangle *\beta \in T\wedge |\beta|=m)\big]
\]
where $g$ is standard (and involves $\mu$).  Then $(\exists n^{0})(g(n)=0)$ implies $(\exists n^{0}\leq \mu(g))(g(n)=0)$, i.e.\ a search bounded by $\mu(g)$ provides a standard number $n^{0}$ such that the subtree of $T$ starting with $\langle n\rangle$ is infinite.  Since this subtree is also infinite, we can repeat this process to find a \emph{standard} sequence $\alpha^{1}$ such that $(\forall^{\st}n^{0})(\overline{\alpha}n\in T)$.  
Applying $\paai$ to the latter now yields $\KOE_{\ns}$.  

\medskip

Now assume $\KOE_{\ns}$ and suppose $\paai$ is false, i.e.\ there is standard $h^{1}$ such that $(\forall^{\st}n^{0})(h(n)=0)\wedge (\exists m)(h(m)\ne0)$.   
Define the tree $T_{0}$ as:  $\sigma\in T_{0}$ if 
\be\label{treeke}
(\forall i<|\sigma|-1)(\sigma(i)=\sigma(i+1))\wedge h(\sigma(0))\ne 0 \wedge (\forall i<\sigma(0))(h(i)=0).
\ee  
Clearly $T_{0}$ is standard, finitely branching, and infinite, but has no standard path, a contradiction.  Hence $\paai$ follows and we are done.  
\end{proof}
We refer to the previous proof as the `textbook proof' of $\KOE_{\ns}\di \paai$.  The proof of this implication is indeed similar to the proof of $\KOE\di \ACA_{0}$ in Simpson's textbook on RM, as found in \cite{simpson2}*{III.7}.  This `textbook proof' is special in a specific technical sense, as will become clear below.     

\medskip

We first prove the associated higher-order result. 
Let $\KOE_{\ef}(t)$ be the statement that $t(T)$ is a path though $T$ if the latter is an infinite and finitely branching tree.         
\begin{thm}\label{serfi}
From \textbf{any} proof of $\KOE_{\ns}\asa \paai$ in $\P $, two terms $s, u$ can be extracted such that $\textsf{\textup{E-PA}}^{\omega}$ proves:
\be\label{froord}
(\forall \mu^{2})\big[\textsf{\MU}(\mu)\di \KOE_{\ef}(s(\mu)) \big] \wedge (\forall t^{1\di 1})\big[ \KOE_{\ef}(t)\di  \MU(u(t))  \big].
\ee
\end{thm}
\begin{proof}
We establish the second implication in \eqref{froord} and leave the remaining one to the reader.  To this end, note that $\KOE_{\ns}\di \paai$ implies that 
\be\label{derfr}
(\forall^{\st}T^{1})(\exists^{\st}\alpha^{1})A(T, \alpha)\di (\forall^{\st}f^{1})(\exists^{\st}n)B(f,n),
\ee
where $B$ is the formula in square brackets in \eqref{huji} and $A(T, \alpha)$ is the \emph{internal} formula expressing that $\alpha$ is a path in the finitely branching and infinite tree $T$.  A standard functional provides standard output on standard input, and \eqref{derfr} yields   
\be\label{derfr2}
(\forall^{\st}t^{1\di 1})\big[(\forall T^{1})A(T, t(T))\di (\forall^{\st}f^{1})(\exists^{\st}n)B(f,n)\big].
\ee
Bringing outside all standard quantifiers, we obtain:
\be\label{derfr3}
(\forall^{\st}t^{1\di 1}, f^{1})(\exists^{\st}n^{0})\big[(\forall T^{1})A(T, t(T))\di B(f,n)\big].
\ee
Applying Theorem \ref{TERM2} to `$\P\vdash \eqref{derfr3}$', we obtain a term $u$ such that $\textsf{E-PA}^{\omega}$ proves 
\be\label{derfr4}
(\forall t^{1\di 1}, f^{1})(\exists n^{0}\in u(t,f))\big[(\forall T^{1})A(T, t(T))\di B(f,n)\big].
\ee
Taking the maximum of $u$, we obtain the second conjunct of \eqref{froord}.  
\end{proof}
To obtain the counterpart of the previous theorem in classical computability theory, consider the following `second-order' version of $(\mu^{2})$:
\be\tag{$\MU^{A}(\nu, e, n )$}
(\exists m, s)(\varphi_{e,s}^{A}(n)=m)\di (\exists m, s\leq \nu(e, n))(\varphi_{e,s}^{A}(n)=m).
\ee
Furthermore, let $\KOE_{\ef}^{A}(t,e)$ be the statement that if an $A$-computable tree with index $e^{0}$ is infinite and finitely branching, then $t^{1}$ is a path through this tree.  
\begin{thm}\label{serf2}
From the `textbook proof' of $\KOE_{\ns}\di \paai$ in $\P$, terms $s,u$ can be extracted such that $\textsf{\textup{E-PA}}^{\omega}$ proves:
\be\label{froord3}
 (\forall e^{0},n^{0} ,C^{1}, \beta^{ 1})\big[ \KOE_{\ef}^{C}(\beta,u(e,n))\di  \MU^{C}(t( \beta, C,e,n),e,n)  \big].
\ee
\end{thm} 
\begin{proof}
We make essential use of the proofs of Theorems \ref{serfu} and \ref{serfi}, and the associated notations.  In particular, define the term $t^{1\di 1}$ by letting $t(h)$ be the tree $T_{0}$ from \eqref{treeke}.  
Clearly, what is proved in the proof of Theorem~\ref{serfi} is in fact
\be\label{milk0}
(\forall^{\st}f^{1})\big[(\exists^{\st}\alpha^{1})A(t(f), \alpha)\di (\exists^{\st}n)B(f,n)\big], 
\ee
where we used the notations from the proof of Theorem \ref{serfi}.  
Indeed, $\KOE_{\ns}$ is only applied for $T=t(f)=T_{0}$ in the proof of Theorem \ref{serfi}.  Now \eqref{milk0} yields:
\be\label{milkm}
(\forall^{\st}f^{1}, \alpha^{1})(\exists^{\st}n^{0})\big[A(t(f), \alpha)\di B(f,n)\big].
\ee
Applying Theorem \ref{TERM2} to `$\P\vdash \eqref{milkm}$', we obtain a term $v$ such that $\textsf{E-PA}^{\omega}$ proves 
\[
(\forall f^{1}, \alpha^{1})(\exists n\in v(f, \alpha))\big[A(t(f), \alpha)\di B(f,n)\big], 
\]
and define $u(f, \alpha)$ as the maximum of all $v(f, \alpha)(i)$ for $i<|v(f, \alpha)|$.  We obtain:
\be\label{milk}
(\forall f^{1}, \alpha^{1})\big[A(t(f), \alpha)\di B(f,u(f, \alpha)))\big], 
\ee
We now modify \eqref{milk} to obtain \eqref{froord3}.  To this end, define $f_{0}^{2}$ as follows:  $f_{0}(e, n, C, k)=0$ if $(\exists m,s\leq k)(\varphi_{e,s}^{C}(n)=m)$, and $1$ otherwise.  For this choice of function, namely taking $f^{1}=_{1}\lambda k.f_{0}$, the sentence \eqref{milk} implies that
\be\label{milk2}
(\forall C^{1}, e^{0}, n^{0}, \alpha^{1})\big[A(t(\lambda k.f_{0}), \alpha)\di B(\lambda k. f_{0},u(\lambda k.f_{0}, \alpha)))\big].
\ee
Now there are obvious (primitive recursive) terms $x^{1}, v^{1}$ such that for any finite sequence $\sigma$, we have $\sigma\in t(\lambda k.f_{0})$ if and only if $\varphi^{C}_{x(e,n), v(e, n)}(\sigma)=1$; the definition of $x^{1}, v^{1}$ is implicit in the definition of $t$ and $f_{0}$.  Hence, with these terms, the antecedent and consequent of \eqref{milk2} are as required to yield \eqref{froord3}.  
\end{proof}
Note that all objects in \eqref{froord3} are type zero or one, except the extracted terms.  
For those familiar with the concept, applying the so-called $\ECF$-translation (See \cite{troelstra1}) to \eqref{froord3} only changes these terms (to computable associates) and results in 
a `pure' statement of second-order arithmetic, i.e.\ classical computability theory.  

\medskip

Let $\KOE_{\ns}(T)$ and $\paai(f)$ be $\KOE_{\ns}$ and $\paai$ without the quantifiers over $f$ and $T$.  The crux of the previous proof is that the `textbook' proof of Theorem \ref{serfu} establishes $(\forall^{\st}f^{1})[\KOE_{\ns}(t(f))\di \paai(f)]$ as in \eqref{milk0}.   
In particular, the corresponding normal form \eqref{milkm} only involves \emph{type zero and one} objects, rather than type two as in e.g.\ \eqref{loveisall} for $\paai\di \MCT_{\ns}$.    
Thus, \emph{thanks to the textbook proof} all objects in \eqref{milk0} are `of low enough type' to yield classical computability theory as in \eqref{froord3}.  

\medskip

We are now ready to reveal the intended `deeper' meaning of the term `textbook proof':
Intuitively, the latter refers to a proof (which may not exist) of an implication $(\forall^{\st}f)A(f)\di (\forall^{\st} g)B(g)$ which also establishes $(\forall^{\st}g)[A(t(g))\di B(g)]$, \emph{and} in which 
the formula in square brackets has a normal form involving only standard quantifiers of type zero and one.  Following the proof of Theorem \ref{serf2}, such a `textbook proof' gives rise to results in \textbf{classical} computability theory.  Such `textbook proofs' also exist for $\MCT$ and Ramsey's theorem, as explored in \cite{sambon2}.    

\medskip

In a nutshell, to obtain the previous theorem, one first establishes the `nonstandard uniform' version \eqref{milk0} of $\KOE_{\ns}\di \paai$, which yields the `super-pointwise' version \eqref{milk}.  The latter is then modified to \eqref{milk2}; this modification should be almost identical for other similar implications.  
In particular, versions of Theorems \ref{serf2} and \eqref{milk0} are obtained in \cite{sambon2} for \emph{K\"onig's lemma} and \emph{Ramsey's theorem} (\cite{simpson2}*{III.7}). 
Similar or related results for the Reverse Mathematics of $\WKL_{0}$ and $\ACA_{0}$ should be straightforward to obtain.  

\medskip
  
In conclusion, higher-order computability results can be obtained from arbitrary proofs of $\KOE_{\ns}\di \paai$, 
while the \emph{textbook proof} as in the proof of Theorem~\ref{serfu} yields classical computability theory as in \eqref{froord3}.  

\section{Nonstandard Analysis and intuitionistic mathematics}\label{titatovenaar}
We have observed that the system $\P$ inhabits the twilight zone between the constructive and non-constructive: $\P$ is not the former as it (explicitly) includes the law of excluded middle, but $\P$ also has `too many' constructive properties to be dismissed as merely the latter.  In Sections \ref{kiko} to \ref{timem}, we provide a possible explanation for the observed behaviour of $\P$ based on the foundational views of Brouwer and Troelstra.  
Finally, we discuss in Section \ref{metastable} a new connection between intuitionistic analysis and Nonstandard Analysis via Tao's notion of \emph{metastability}.     

\subsection{Nonstandard Analysis and Brouwer's view on logic}\label{kiko}
In this section, we offer an alternative interpretation of our results based on Brouwer's claim that \emph{logic is dependent on mathematics}.  
Indeed, Brouwer already explicitly stated in his dissertation that \emph{logic depends upon mathematics} as follows: 
\begin{quote}
While thus mathematics is independent of logic, logic does depend upon mathematics: in the first place \emph{intuitive logical reasoning} is that special kind of mathematical reasoning which remains if, considering mathematical structures,
one restricts oneself to relations of whole and part; 
(\cite{brouwt}*{p.\ 127}; emphasis in the Dutch original)
\end{quote}
According to Brouwer, logic is thus abstracted from mathematics, and he even found it conceivable that under different circumstances, a different abstraction would emerge from the same mathematics.  
\begin{quote}
Therefore it is easily conceivable that, given the same organization of the
human intellect and consequently the same mathematics, a different language 
would have been formed, into which the language of logical reasoning, 
well known to us, would not fit. (\cite{brouwt}*{p.\ 129})
\end{quote}
Building on Brouwer's viewpoint, it seems reasonable that different kinds of logical reasoning can emerge when the `source', i.e.\ the underlying mathematics, is different.    
Indeed, Brouwer found it conceivable that a different logical language could emerge from the \emph{same} mathematics, suggesting that changing mathematics in a fundamental way, a different kind of logical reasoning is (more?) likely to emerge.  

\medskip

This leads us to our alternative interpretation: If one fundamentally changes mathematics, as one arguably does when introducing Nonstandard Analysis, the logic will change along as the latter depends on the former in Brouwer's view.  Since classical logic cannot really become `more non-constructive', it stands to reason that the logic of (classical) Nonstandard Analysis actually is more constructive than plain classical logic.  
We present two strands of evidence for this claim.  

\medskip

First of all, we have observed in Section \ref{fraki} that when introducing the notion of `being standard' fundamental to Nonstandard Analysis, the law of excluded middle of classical mathematics moves from `the original sin of non-constructivity' to a \emph{computationally inert} principle which does not have any real non-constructive consequences anymore.  In particular, adding $\LEM$ to $\textsf{E-HA}^{\omega}$ leads to the classical system $\textsf{E-PA}^{\omega}$, while adding $\LEM$ (and even $\LEM_{\ns}$) to $\H$ results in the system $\P$ \emph{which has exactly the same term extraction theorem} as $\H$ by Theorems~\ref{TERM} and \ref{TERM2}.  The latter even explicitly state that $\LEM$ 
does not influence the extracted term!   In this way, our results on Nonstandard Analysis vindicate Brouwer's thesis that logic depends upon mathematics:  Changing classical mathematics fundamentally by introducing the `standard versus nonstandard' distinction from Nonstandard Analysis, the associated logic moves from classical logic to a new, more constructive, logic in which `there exists a standard' has computational meaning akin to constructive mathematics and its BHK-interpretation.  

\medskip

As it happens, the observation that the introduction of the framework of Nonstandard Analysis changes the associated logic has been made by the pioneers of Nonstandard Analysis Robinson and Nelson.  The latter has indeed qualified Nonstandard Analysis as a full-fledged new logic as follows. 
\begin{quote}
What Abraham Robinson invented is nothing less than a new logic. (\cite{nelsonub}*{\S1.6})
\end{quote}   
The previous claim is based on the following similar one by Robinson from his original monograph \emph{Non-standard Analysis} (\cite{robinson1}).  
\begin{quote}
Returning now to the theory of this book, we observe that it is presented, naturally, within the framework of contemporary Mathematics, and thus appears to affirm 
the existence of all sorts of infinitary entities. However, from a formalist point of view we may look at our theory syntactically and may consider that what we have done is to introduce 
\emph{new deductive procedures} rather than new mathematical entities. (\cite{robinson1}*{\S10.7}; emphasis in original)
\end{quote}
Thus, Robinson and Nelson already had the view that 
introducing the `standard versus nonstandard' distinction from Nonstandard Analysis, the associated logic moves from classical logic to a new `nonstandard' logic.  Our goal in this paper has been to show that this new logic has constructive content.   

\medskip

Finally, a word of warning is in order: it is fair to say that both Robinson and Nelson had `nonstandard' foundational views:  Robinson subscribed to \emph{formalism} (\cite{robinson64}) while Nelson even rejected the totality of the exponential function (\cite{ohnelly}).  

\subsection{Nonstandard Analysis and Troelstra's view of mathematics}
In this section, we offer an alternative interpretation of our results based on Troelstra's views on intuitionism.  

\medskip

First of all, Robinson-style Nonstandard Analysis involves a nonstandard model of a structure in which the latter is embedded, the textbook example being the real numbers $\R$ which form 
a subset of the `hyperreal' numbers $^{*}\R$.    
At the risk of pointing out the obvious, one thus studies the real numbers $\R$ \emph{from the outside}, namely as embedded into the (much) larger structure $^{*}\R$. 

\medskip

Secondly, according to Nelson\footnote{As noted above and discussed at length in \cite{samBIG}, one does \emph{not} have to share Nelson's view to study $\IST$: the latter is a logical system while the former is a possible philosophical point of view regarding $\IST$.  However, like Nelson, we view the mathematics and this point of view as one.} , Nonstandard Analysis as in $\IST$ does not introduce a new structure like $^{*}\R$, but imposes a new predicate `is standard' onto the \emph{existing} structure $\R$.  
As discussed above, the set of all standard real numbers cannot be formed in $\IST$.  Thus, in Nelson's view, $\IST$ is also fundamentally based on the `standard versus nonstandard' distinction, but one does not view mathematics `from the outside', but performs a kind of `introspection' via the `st' predicate.  

\medskip

Thirdly, while the above observations are not exactly new, the reader may be surprised to learn that Troelstra made similar claims 
regarding the nature of intuitionism, as follows.  
\begin{quote}
We may start with very simple concrete constructions, such as the natural
numbers, and then gradually build up more complicated, but nevertheless ``concrete"
or ``visualizable" structures. Finitism is concerned with such constructions only
(\cite{kreikel}*{3.4}). In intuitionism, we also want to exploit the idea that there is
an intuitive concept of ``constructive", by reflection on the properties of 
``constructions which are implicitely involved in the concept". (I.e. we attempt to
discover principles by introspection.)

\medskip
\noindent
Finitist constructions build up ``from below"; reflecting on the general notion
represents, so to speak, an approach ``from the outside", ``from above". 
(\cite{uprin}*{\S1.1})
\end{quote}
In a nutshell, according to Troelstra, intuitionism deals with the study of concrete structures via a kind of `introspection' or `study from the outside'.
Thus, we observe a remarkable similarity between the \emph{methods of study} in intuitionism and Nonstandard Analysis:  Nonstandard Analysis seems to originate from classical mathematics when applying to the latter the methodological techniques of intuitionism, namely `introspection' (Nelson's $\IST$) and `study from the outside' (Robinsonian nonstandard models).
Now, if all this sounds outrageous to the reader, we are not the first to make this claim: Wallet writes the following about Harthong-Reeb in \cite{reeb10}*{\S7}:
\begin{quote}
On the other hand, Harthong and Reeb explains in \cite{reeb3} that, far from being an artefact, nonstandard analysis necessarily results of an intuitionnistic [sic] interpretation of the classical mathematical formalism (see also \cite{reeb8}).
\end{quote}
The references mentioned by Wallet are in French (and hard to come by), and we therefore do not go into more detail.  
Nonetheless, the message is loud and clear:  There is a clear methodological connection between Nonstandard Analysis and intuitionism, as observed above and independently by Harthong-Reeb.  
Thus, as the French are wont to say: \emph{Discutez!}

\subsection{Nonstandard Analysis and the awareness of time}\label{timem}
In this section, we formulate a motivation for Nonstandard Analysis based on a fundamental intuitionist notion, namely the awareness of time.  As to the latter's central status, Brouwer's \emph{first act of intuitionism} reads as follows:
\begin{quote}
 Completely separating mathematics from mathematical language and hence from the phenomena of language described by theoretical logic, recognizing that intuitionistic mathematics is an essentially languageless activity of the mind having its origin in the perception of a move of time. This perception of a move of time may be described as the falling apart of a life moment into two distinct things, one of which gives way to the other, but is retained by memory. If the twoity thus born is divested of all quality, it passes into the empty form of the common substratum of all twoities. And it is this common substratum, this empty form, which is the basic intuition of mathematics. (\cite{brouwcam}, p.\ 4-5) 
\end{quote}
Needless to say, intuitionism is fundamentally rooted in the awareness of time.  As such, the aforementioned first act of intuitionism gives rise to the natural numbers (considered as a kind of potential infinity obtained by adjoining units), while the second act (See \cite{brouwcam}*{p.\ 8}) gives rise to the continuum.  In particular, Brouwer's twoity simultaniously yields discrete units \emph{and} an inexhaustible continuum.
\begin{quote}
In the following chapters we shall go further into the basic intuition of mathematics (and of every intellectual activity) as the substratum, divested of all quality, of any perception of change, a unity of continuity and discreteness, a possibility of thinking together several entities, connected by a `between', which is never exhausted by the insertion of new entities. 
Since continuity and discreteness occur as inseparable complements, both having equal rights and being equally clear, it is impossible to avoid one of them as a primitive entity, [\dots]
(\cite{brouw}*{p.\ 17})
\end{quote}
We shall not comment on Brouwer's claims here, but instead make the following observation regarding Nonstandard Analysis inspired by the first act.  
Like Brouwer, we assume the natural numbers to be given (by the first act or otherwise). 
\begin{obs}[Nonstandard Analysis and the awareness of time]\label{obG}\rm    
In intuitionism, the natural numbers (as a potential infinity) are obtained from Brouwer's fundamental concept \emph{twoity} which is rooted in the awareness of time, in particular in the passage of time from one moment to the next. 
As a thought experiment, the reader should attempt to imagine a \emph{long} period of time, say a day, a week, or even a decade, \emph{as built up from units of time}, say seconds.  

\medskip

While one \emph{knowns} there are only finitely\footnote{As it turns out, there are 315,360,000 seconds in a decade.} many seconds which make up this long period of time, one cannot really \emph{conceive} of the latter as a whole built up from said units.  In other words, despite the variable perception of time (`time flies' versus `it took an eternity'), there are periods of time which cannot be grasped `all at once' in terms of much smaller periods.  
Thus, there is a tension between one's subjective perception of time and objective time as follows: \emph{If} one introduces a unit of time, \emph{then} there are periods of time which cannot be grasped \emph{fully and all at once} only in terms of these units.  Alternatively, making discrete units small enough, the result starts \emph{looking} like the underlying continuum, i.e.\ there is a holistic property to a large collection of units which in particular cannot be reduced to the underlying units.  
A similar observation holds\footnote{While Brouwer makes a fundamental distinction between the continuous and the discrete, packing enough smaller and smaller discrete units together into the same area inevitably results in something that \emph{looks like} a continuum (which is also the basis of computer displays).  There are of course plenty of collections of discrete units which still look discrete.} for space, but we consider time since Brouwer took the awareness of time to be fundamental, as sketched above.  

\medskip
  
Thus, one observes an intuitive distinction between \emph{short} and \emph{long} time periods, i.e.\ ones which can be directly conceived as built up from units of time, and one which cannot.  
If a time period is short (resp.\ long) in this way, adjoining (resp.\ removing) another unit of time yields another short (resp.\ long) time period.  If the distinction `long versus short' is divested of all quality, one arrives at the usual notion of `standard versus nonstandard' natural number from NSA.  
Thus, the standardness predicate becomes grounded in one's awareness of time, namely in the intuitive distinction between long versus short periods of time. 
\end{obs}
By the previous observation, the fundamental notion `standard versus nonstandard' from NSA can be grounded elegantly in one's awareness of time, the tension between subjective and objective time in particular.  
We leave it to the reader's own imagination what other properties (like e.g.\ Definition \ref{flah}) one can bestow upon the standardness predicate in the same way, i.e.\ motivated by the awareness of time.  We point out that we avoided (the trap of) \emph{ultrafinitism} (See \cite{troeleke1}*{p.\ 29}) by considering the natural numbers as given.  

\medskip

In our opinion, Observation \ref{obG} is merely the `next logical step' following Brouwer's first act inspired by NSA.  
As such, this observation could have been made long time ago.  Reasons for this apparent gap in the literature may be found in 
the rejection of NSA by e.g.\ Bishop (See Section~\ref{crackp}).  
Now, regarding a related `future logical step' for NSA, it stands to reason that advances in artificial intelligence will lead to programs 
which, if concrete inputs and contexts are specified for a theorem, provide numerical estimates for vague notions `long' and `small' related to the theorem.  

\medskip

By way of an example, the \emph{fundamental theorem of calculus} states that Riemann integration and differentiation cancel each other out.  The aforementioned programs would produce numerical estimates for the statement \emph{If the mesh $\|\pi\|$ of a partition $\pi$ is \textbf{very small} compared to $\eps$ in $\frac{f(x+\eps)-f(x)}{\eps}$, then the function $f$ is \textbf{very close} to the Riemann sum of $\lambda x.\frac{f(x+\eps)-f(x)}{\eps}$  and $\pi$}.  However, the latter may also be obtained via term extraction from the \emph{nonstandard version} of the {fundamental theorem of calculus} (See \cite{sambon}*{\S4}).  Time will tell if any cross-over will take place.   

\subsection{Intuitionistic analysis, metastability, and Nonstandard Analysis}\label{metastable}
We discuss the connection between intuitionistic and Nonstandard Analysis provided by Tao's \emph{metastability} (See \cite{taote}*{\S2.3.1} or \cite{kohlenbach3}*{\S2.29}).  
We note that Anil Nerode first suggested the possibility of a connection between intuitionistic and Nonstandard Analysis based on his student Scowcroft's work on intuitionistic analysis.    

\medskip

First of all, metastability is a version of convergence which intuitively speaking swaps `convergence in the limit' for the weaker `convergence on finite but arbitrarily large domains'.  
An advantage is that the weakening from `epsilon delta' convergence to `metastability' however yields \emph{highly uniform} computational content.  
We show in this section that intuitionistic analysis and Nonstandard Analysis give rise to the \textbf{very same} {highly uniform} computational content by introducing finite but arbitrarily large domains in the same way as for metastability.  
We shall provide an example of the aforementioned phenomenon and refer to \cite{dagsamII} for general results.  

\medskip

Secondly, we consider the `textbook example' of metastability in Example \ref{hexal}.  
\begin{exa}[Metastability and convergence]\label{hexal}\rm
The \emph{monotone convergence theorem} states that any non-decreasing sequence in the unit interval converges, and hence is Cauchy.   
The usual `$\eps$-$\delta$' definition of a Cauchy sequence $a_{(\cdot)}^{0\di 1}$ is as follows:
\be\label{cruxie}
(\forall \eps >_{\R}0)\underline{(\exists N^{0})(\forall n^{0}, m^{0}\geq N)}( |a_{n}-a_{m}|\leq\eps), 
\ee
which is equivalent to the well-known `metastable' version:
\be\label{cruxie2}
(\forall \eps >_{\R}0, F)\underline{(\exists M^{0})}(\forall n^{0}, m^{0}\in [M, F(M)])( |a_{n}-a_{m}|\leq\eps), 
\ee
where the formula starting with `$(\forall n, m)$' is equivalent to a quantifier-free one, since the latter quantifier is bounded.  
Now, while \eqref{cruxie} and \eqref{cruxie2} are equivalent, their computational behaviour is quite different:  
On one hand, there is no way to compute an upper bound on $N^{0}$ in the usual definition \eqref{cruxie}, as this would provide a computable solution to the Halting problem.  
On the other hand, an upper bound for $M^{0}$ as in \eqref{cruxie2} is $F^{\lceil \frac{1}{\eps}\rceil+1}(0)$, which is the result of iterating $F$ for $\lceil \frac{1}{\eps}\rceil+1$-many times and then evaluating at $0$.  Moreover, $M^{0}$ as in \eqref{cruxie2} is one of the values in the finite sequence $\langle 0, F(0), F(F(0)), \dots, F^{\lceil \frac{1}{\eps}\rceil+1}(0)\rangle$.  
Thus, we obtain:  
\begin{align}\label{cruxie3}\textstyle
(\exists \theta^{(0\times 1)\di 0^{*}})(\forall k^{0}, F, a_{(\cdot)}\in I([0,1]))(\exists& M \in \theta(k, F))\\
&\textstyle(\forall n, m\in [M, F(M)])( |a_{n}-a_{m}|\leq\frac{1}{k})\notag
\end{align}
where `$a_{(\cdot)}\in I([0,1])$' means that $a_{(\cdot)}$ is a non-decreasing sequence in $[0,1]$.  
\end{exa}
In the previous example, we witness the power of metastability: while the latter only provides computational information about a finite (but arbitrarily large) domain, namely $[M, F(M)]$ for any $F$ as in \eqref{cruxie2}, this information is \emph{highly uniform} in that $\theta$ from \eqref{cruxie3} does not even depend on the \emph{choice of sequence} $a_{(.)}$.  

\medskip

Thirdly, we show that intuitionistic analysis \emph{also} provides highly uniform results \emph{\`a la} \eqref{cruxie3} and metastability.  To this end, we consider the theorem $\DNR$ from Reverse Mathematics which satisfies the (strict) implication $\WKL\di \DNR$ (See \cite{withgusto}).  
Now, $\DNR$ stands for `diagonal non-recursive' and is defined as
\be\tag{$\DNR$}
(\forall A^{1}\leq_{1}1)\underline{(\exists f^{1})(\forall m^{0}, s^{0}, e^{0})}(\varphi_{e, s}^{A}(e)=m \di f(e)\ne m),
\ee
where `$\varphi^{A}_{e, s}(n)=m$' is the usual primitive recursive predicate expressing that the Turing machine with index $e$, input $n$, and oracle $A$, halts with output $m$ after at most $s$ steps.  
As shown in \cite{samzooII}, $f^{1}$ as in $\DNR$ cannot be computed from $A\leq_{1}1$.  
Thus, similar to the equivalence between \eqref{cruxie} and \eqref{cruxie2}, we modify the underlined pair of quantifiers in $\DNR$ as follows: 
\be\label{frolk}
(\forall A^{1}\leq_{1}1, G^{2})\underline{(\exists f^{1})}(\forall m^{0}, s^{0}, e^{0}\leq G(f))(\varphi_{e, s}^{A}(e)=m \di f(e)\ne m)
\ee
Similar to \eqref{cruxie2}, what is left in \eqref{frolk} is the \emph{single} underlined quantifier.   Let us now write down the associated version of \eqref{cruxie3} for $\DNR$ and \eqref{frolk}, as follows:
\be\tag{$\DNR(\zeta)$}
(\forall A\leq_{1}1, G^{2})(\exists f\in \zeta(G))(\forall m, s, e \leq G(f) )(\varphi_{e, s}^{A}(e)=m \di f(e)\ne m).
\ee
Note that $\zeta$ computes a `metastable' diagonal non-recursive $f$ \emph{independent} of $A$, similar to how $\theta$ from \eqref{cruxie3} does not depend on $a_{(\cdot)}$.  
Although there is no a priori reason for $\zeta$ to exist, intuitionistic mathematics proves its existence as follows. 
\begin{thm}\label{pen14}
The system $\RCAo+(\exists \Omega)\MUC(\Omega)+\QFAC$ proves $(\exists \zeta)\DNR(\zeta)$.  
\end{thm}
\begin{proof}
By \cite{simpson2}*{IV.2.3}, $(\exists \Omega)\MUC(\Omega)$ implies $\WKL$ and hence $\DNR$.  Now consider \eqref{frolk} and apply $\QFAC$ to obtain $\Xi^{(1\times 2)\di 1}$ which outputs $f$ as in \eqref{frolk}.  
Now consider $k_{0}:=\Omega(\lambda B. \Xi(G, B))$ and define $\zeta_{0}(G)$ as the finite sequence of the form $\langle \Xi(G, \sigma_{0}), \Xi(G, \sigma_{1}), \dots,\Xi(G,  \sigma_{2^{k_{0}}})\rangle$ where $\sigma_{i}$ is the $i$-th binary sequence of length $k_{0}$.  Clearly, this functional $\zeta_{0}$ satisfies $\DNR(\zeta_{0})$ and we are done.      
\end{proof}
Hence, the intuitionistic fan functional guarantees the existence of a highly uniform (i.e.\ independent of $A$) functional $\zeta$ which computes `metastable' diagonal recursive functions.   
However, $(\exists \zeta)\DNR(\zeta)$ is not a statement of second-order arithmetic, and there is no a priori reason such $\zeta$ would exist in classical mathematics.  

\medskip

Fourth, we shall establish $(\exists \zeta)\DNR(\zeta)$ in classical mathematics using $\STP$ and $\DNR_{\ns}$, where the latter is defined as follows:
\[
(\forall^{\st}G^{2})(\exists^{\st}w^{1^{*}})(\forall A^{1}\leq_{1}1)(\exists f^{1}\in w)(\forall m^{0}, s^{0}, e^{0}\leq G(f))(\varphi_{e, s}^{A}(e)=m \di f(e)\ne m).
\]
In particular, $\zeta$ as in $\DNR(\zeta)$ exists in classical mathematics as follows.
%
\begin{thm}\label{kkringe}
From the proof in $\P$ that $\STP\di \DNR_{\ns}$, a term $t^{3\di3}$ can be extracted such that 
$\textup{\textsf{E-PA}}^{\omega}$ proves $(\forall \Theta^{3})\big[\SCF(\Theta)\di \DNR(t(\Theta))]$.  
\end{thm}
\begin{proof}
See Theorem \ref{finake}
\end{proof}
By the theorem, $\STP$ and the special fan functional provide a way of obtaining the existence of a highly uniform (i.e.\ independent of $A$) functional $\zeta$ which computes `metastable' diagonal recursive functions as in $\DNR(\zeta)$.  Note that we could apply the same method to e.g.\ the monotone convergence theorem (See \cite{dagsamII}), although the latter theorem contradicts intuitionistic mathematics.      

\medskip

In conclusion, metastability is a weak version of convergence which gives rise to highly uniform computational information.  
Thanks to the intuitionistic fan functional, we can assert the existence of a highly uniform functional which computes `metastable' diagonal recursive functions.   
However, the special fan functional, which is just a realiser for $\STP$, also computes such a highly uniform functional. 

\medskip

Of course, theorems \ref{pen14} and \ref{kkringe} are just examples of the following general template, to be explored in \cite{dagsamII}:  The intuitionistic fan functional provides highly uniform functionals based on the general heuristic underlying metastability, which is: 
\begin{quote}
Introducing finite but arbitrarily large domains yields highly uniform computational results.  
\end{quote}
These highly uniform functionals are also part of classical mathematics, as established thanks to Nonstandard Analysis, and \emph{Standardisation} and $\STP$ in particular.    

\medskip

Finally, we find it fitting that the final section of this paper connects intuitionistic analysis and Nonstandard Analysis via a \emph{mathematical} notion, namely \emph{metastability}\footnote{The notion of metastability was known in logic (See \cite{kohlenbach3}*{\S2.29}) under a different name well before Tao first discussed this notion in \cite{taote} and previously on his blog.}, which we imagine 
would have pleased Brouwer.  
\begin{ack}\rm
This research was supported by the following funding entities: FWO Flanders, the John Templeton Foundation, the Alexander von Humboldt Foundation, LMU Munich (via the Excellence Initiative), and the Japan Society for the Promotion of Science.  The work was done partially while the author was visiting the Institute for Mathematical Sciences of the National University of Singapore in 2016. The visit was supported by the Institute.  
The author expresses his gratitude towards these institutions.  The opinions expressed in this paper are not necessarily those of the aforementioned entities.    

\medskip

The author would like to thank Horst Osswald, Ulrich Kohlenbach, Anil Nerode, and Karel Hrbacek for their valuable advice.  In particular, the fundamental idea of this paper, the \emph{local constructivity} of Nonstandard Analysis, is due to Horst Osswald to whom this paper owes the greatest intellectual debt.  Furthermore, the majority of results in this paper build (in one way or another) on previous work by Ulrich Kohlenbach and his school.  Anil Nerode has provided many conceptual ideas and motivation for this long paper, and the author's research in general.   
Finally, Karel Hrbacek proofread this paper in great detail, and has previously motivated and inspired the author.       
\end{ack}

\appendix
\section{The formal systems $\P$ and $\H$}\label{FULL}
We introduce the systems $\P$ and $\H$ from \cite{brie} in detail in Section \ref{PAPA} and \ref{TATA}.  
To this end, we first introduce G\"odel's system $T$ in Section \ref{TITI}.  

\subsection{G\"odel's system ${T}$}\label{TITI}
In this section, we briefly introduce G\"odel's system ${T}$ and the associated systems $\textsf{E-PA}^{\omega}$ and $\textsf{E-PA}^{\omega*}$.  
In his famous \emph{Dialectica} paper (\cite{godel3}), G\"odel defines an interpretation of intuitionistic arithmetic into a quantifier-free calculus of functionals.  This calculus is now known as `G\"odel's system ${T}$', and is essentially just primitive recursive arithmetic (\cite{buss}*{\S1.2.10}) with the schema of recursion expanded to \emph{all finite types}.  
The set of all finite types $\mathbf{T}$ is:
\begin{center}
(i) $0\in \mathbf{T}$   and   (ii)  If $\sigma, \tau\in \mathbf{T}$ then $( \sigma \di \tau) \in \mathbf{T}$,
\end{center}
where $0$ is the type of natural numbers, and $\sigma\di \tau$ is the type of mappings from objects of type $\sigma$ to objects of type $\tau$.  Hence, G\"odel's system ${T}$ includes `recursor' constants $\mathbf{R}^{\rho}$ for every finite type $\rho\in \mathbf{T}$, defining primitive recursion as follows:  
\be\label{PR}\tag{\textsf{\textup{PR}}}
\mathbf{R}^{\rho}(f,g, 0):=f   \textup{ and } \mathbf{R}^{\rho}(f, g, n+1):=g(n, \mathbf{R}^{\rho}(f, g, n)),
\ee
for $f^{\rho}$ and $g^{0\di( \rho\di \rho)}$.  
The system $\textsf{E-PA}^{\omega}$ is a combination of \emph{Peano Arithmetic} and system $T$, and the full axiom of extensionality \eqref{EXT}.  
The detailed definition of $\textsf{E-PA}^{\omega}$ may be found in \cite{kohlenbach3}*{\S3.3};  We do introduce the notion of equality and extensionality in $\textsf{E-PA}^{\omega}$, as these notions are needed below.
\bdefi[Equality]\label{FAK}
The system $\textsf{E-PA}^{\omega}$ includes equality between natural numbers `$=_{0}$' as a primitive.  Equality `$=_{\tau}$' for type $\tau$-objects $x,y$ is then:
\be\label{aparth}
[x=_{\tau}y] \equiv (\forall z_{1}^{\tau_{1}}\dots z_{k}^{\tau_{k}})[xz_{1}\dots z_{k}=_{0}yz_{1}\dots z_{k}]
\ee
if the type $\tau$ is composed as $\tau\equiv(\tau_{1}\di \dots\di \tau_{k}\di 0)$.  
The usual inequality predicate `$\leq_{0}$' between numbers has an obvious definition, and the predicate `$\leq_{\tau}$' is just `$=_{\tau}$' with `$=_{0}$' replaced by `$\leq_{0}$' in \eqref{aparth}.    
The \emph{axiom of extensionality} is the statement that for all $\rho, \tau\in \mathbf{T}$, we have:
\be\label{EXT}\tag{\textsf{E}}  
(\forall  x^{\rho},y^{\rho}, \varphi^{\rho\di \tau}) \big[x=_{\rho} y \di \varphi(x)=_{\tau}\varphi(y)   \big], 
\ee 
\edefi
Next, we introduce $\textsf{E-PA}^{\omega*}$, a definitional extension of $\textsf{E-PA}^{\omega}$ from \cite{brie} with a type for finite sequences.  In particular,  the set $\mathbf{T}^{*}$ is defined as:
\begin{center}
(i) $0\in \mathbf{T}^{*}$,   (ii)  If $\sigma, \tau\in \mathbf{T}^{*}$ then $ (\sigma \di \tau) \in \mathbf{T}^{*}$, and (iii) If $\sigma \in \mathbf{T}^{*}$, then $\sigma^{*}\in \mathbf{T}^{*}$,
\end{center}
where $\sigma^{*}$ is the type of finite sequences of objects of type $\sigma$.  The system $\textsf{E-PA}^{\omega*}$ includes \eqref{PR} for all $\rho\in \mathbf{T}^{*}$, as well as dedicated `list recursors' to handle finite sequences for any $\rho^{*}\in \mathbf{T}^{*}$.        
A detailed definition of $\textsf{E-PA}^{\omega*}$ may be found in \cite{brie}*{\S2.1}.  
We now introduce some notations for $\textsf{E-PA}^{\omega*}$, as also used in \cite{brie}.
\begin{nota}[Finite sequences]\label{skim}\rm
The system $\textsf{E-PA}^{\omega*}$ has a dedicated type for `finite sequences of objects of type $\rho$', namely $\rho^{*}$.  Since the usual coding of pairs of numbers goes through in $\textsf{E-PA}^{\omega*}$, we shall not always distinguish between $0$ and $0^{*}$. 
Similarly, we do not always distinguish between `$s^{\rho}$' and `$\langle s^{\rho}\rangle$', where the former is `the object $s$ of type $\rho$', and the latter is `the sequence of type $\rho^{*}$ with only element $s^{\rho}$'.  The empty sequence for the type $\rho^{*}$ is denoted by `$\langle \rangle_{\rho}$', usually with the typing omitted.  Furthermore, we denote by `$|s|=n$' the length of the finite sequence $s^{\rho^{*}}=\langle s_{0}^{\rho},s_{1}^{\rho},\dots,s_{n-1}^{\rho}\rangle$, where $|\langle\rangle|=0$, i.e.\ the empty sequence has length zero.
For sequences $s^{\rho^{*}}, t^{\rho^{*}}$, we denote by `$s*t$' the concatenation of $s$ and $t$, i.e.\ $(s*t)(i)=s(i)$ for $i<|s|$ and $(s*t)(j)=t(|s|-j)$ for $|s|\leq j< |s|+|t|$. For a sequence $s^{\rho^{*}}$, we define $\overline{s}N:=\langle s(0), s(1), \dots,  s(N)\rangle $ for $N^{0}<|s|$.  
For a sequence $\alpha^{0\di \rho}$, we also write $\overline{\alpha}N=\langle \alpha(0), \alpha(1),\dots, \alpha(N)\rangle$ for \emph{any} $N^{0}$.  By way of shorthand, $q^{\rho}\in Q^{\rho^{*}}$ abbreviates $(\exists i<|Q|)(Q(i)=_{\rho}q)$.  Finally, we shall use $\underline{x}, \underline{y},\underline{t}, \dots$ as short for tuples $x_{0}^{\sigma_{0}}, \dots x_{k}^{\sigma_{k}}$ of possibly different type $\sigma_{i}$.          
\end{nota}
We have used $\textsf{E-PA}^{\omega}$ and $\textsf{E-PA}^{\omega*}$ interchangeably in this paper.  Our motivation is the `star morphism' used in Robinson's approach to Nonstandard Analysis, and the ensuing potential for confusion.  

\subsection{The classical system $\P$}\label{PAPA}
In this section, we introduce the system $\P$, a conservative extension of $\textsf{E-PA}^{\omega}$ with fragments of Nelson's $\IST$.  

\medskip

To this end, we first introduce the base system $\textsf{E-PA}_{\st}^{\omega*}$.  
We use the same definition as \cite{brie}*{Def.~6.1}, where \textsf{E-PA}$^{\omega*}$ is the definitional extension of \textsf{E-PA}$^{\omega}$ with types for finite sequences as in \cite{brie}*{\S2} and the previous section.    
The language of $\textsf{E-PA}_{\st}^{\omega*}$ (and $\P$) is the language of \textsf{E-PA}$^{\omega*}$ extended with a new symbol `$\st_{\rho}$' for any finite type $\rho\in \textbf{T}^{*}$ in the language of \textsf{E-PA}$^{\omega*}$;  The typing of `st' is always omitted.  
\bdefi\label{debs}
The set $\TT^{*}$ is defined as the collection of all the terms in the language of $\textsf{E-PA}^{\omega*}$.  
The system $ \textsf{E-PA}^{\omega*}_{\st} $ is defined as $ \textsf{E-PA}^{\omega{*}} + \TT^{*}_{\st} + \textsf{IA}^{\st}$, where $\TT^{*}_{\st}$
consists of the following axiom schemas.
\begin{enumerate}
\renewcommand{\theenumi}{\roman{enumi}} 
\item The schema\footnote{The language of $\textsf{E-PA}_{\st}^{\omega*}$ contains a symbol $\st_{\sigma}$ for each finite type $\sigma$, but the subscript is essentially always omitted.  Hence $\T^{*}_{\st}$ is an \emph{axiom schema} and not an axiom.\label{omit}} $\st(x)\wedge x=y\di\st(y)$,
\item The schema providing for each closed\footnote{A term is called \emph{closed} in \cite{brie} (and in this paper) if all variables are bound via lambda abstraction.  Thus, if $\underline{x}, \underline{y}$ are the only variables occurring in the term $t$, the term $(\lambda \underline{x})(\lambda\underline{y})t(\underline{x}, \underline{y})$ is closed while $(\lambda \underline{x})t(\underline{x}, \underline{y})$ is not.  The second axiom in Definition \ref{debs} thus expresses that $\st_{\tau}\big((\lambda \underline{x})(\lambda\underline{y})t(\underline{x}, \underline{y})\big)$ if $(\lambda \underline{x})(\lambda\underline{y})t(\underline{x}, \underline{y})$ is of type $\tau$.  We usually omit lambda abstraction for brevity.\label{kootsie}} term $t\in \T^{*}$ the axiom $\st(t)$.
\item The schema $\st(f)\wedge \st(x)\di \st(f(x))$.
\end{enumerate}
The external induction axiom \textsf{IA}$^{\st}$ is as follows.  
\be\tag{\textsf{IA}$^{\st}$}
\Phi(0)\wedge(\forall^{\st}n^{0})(\Phi(n) \di\Phi(n+1))\di(\forall^{\st}n^{0})\Phi(n).     
\ee
\edefi
Secondly, we introduce some essential fragments of $\IST$ studied in \cite{brie}.  
\bdefi[External axioms of $\P$]~
\begin{enumerate}
\item$\HAC_{\INT}$: For any internal formula $\varphi$, we have
\be\label{HACINT}
(\forall^{\st}x^{\rho})(\exists^{\st}y^{\tau})\varphi(x, y)\di \big(\exists^{\st}F^{\rho\di \tau^{*}}\big)(\forall^{\st}x^{\rho})(\exists y^{\tau}\in F(x))\varphi(x,y),
\ee
\item $\textsf{I}$: For any internal formula $\varphi$, we have
\[
(\forall^{\st} x^{\sigma^{*}})(\exists y^{\tau} )(\forall z^{\sigma}\in x)\varphi(z,y)\di (\exists y^{\tau})(\forall^{\st} x^{\sigma})\varphi(x,y), 
\]
\item The system $\P$ is $\textsf{E-PA}_{\st}^{\omega*}+\textsf{I}+\HAC_{\INT}$.
\end{enumerate}
\end{defi}
Note that \textsf{I} and $\HAC_{\INT}$ are fragments of Nelson's axioms \emph{Idealisation} and \emph{Standard part}.  
By definition, $F$ in \eqref{HACINT} only provides a \emph{finite sequence} of witnesses to $(\exists^{\st}y)$, explaining its name \emph{Herbrandized Axiom of Choice}.   

\medskip

The system $\P$ is connected to $\textsf{E-PA}^{\omega}$ by the following theorem.    
Here, the superscript `$S_{\st}$' is the syntactic translation defined in \cite{brie}*{Def.\ 7.1}. 
By the definition of $S_{\st}$, $\varphi$ in the following theorem is internal.
\begin{thm}\label{consresult}
Let $\Phi(\tup a)$ be a formula in the language of \textup{\textsf{E-PA}}$^{\omega*}_{\st}$ and suppose $\Phi(\tup a)^\Sh\equiv\forallst \tup x \, \existsst \tup y \, \varphi(\tup x, \tup y, \tup a)$. If $\Delta_{\intern}$ is a collection of internal formulas and
\be\label{antecedn}
\P + \Delta_{\intern} \vdash \Phi(\tup a), 
\ee
then one can extract from the proof a sequence of closed\footnote{Recall the definition of closed terms from \cite{brie} as sketched in Footnote \ref{kootsie}.\label{kootsie3}} terms $t$ in $\mathcal{T}^{*}$ such that
\be\label{consequalty}
\textup{\textsf{E-PA}}^{\omega*} + \Delta_{\intern} \vdash\  \forall \tup x \, \exists \tup y\in \tup t(\tup x)\ \varphi(\tup x,\tup y, \tup a).
\ee
\end{thm}
\begin{proof}
Immediate by \cite{brie}*{Theorem 7.7}.  
\end{proof}
The proofs of the soundness theorems in \cite{brie}*{\S5-7} provide an algorithm to obtain the term $t$ from the theorem.  In particular, these terms 
can be `read off' from the nonstandard proofs.  The translation $S_{\st}$ can be formalised in any reasonable\footnote{Here, a `reasonable' system is one which can prove the usual properties of finite lists, for which the presence of the exponential function suffices.} system of constructive mathematics.  In fact, the formalisation of the results in \cite{brie} 
in the proof assistant Agda (based on Martin-L\"of's constructive type theory \cite{loefafsteken}) is underway in \cite{EXCESS}.        
The $D_{\st}$ interpretation ($S_{\st}$ is the latter followed by the Kuroda negative translation) from \cite{brie} has been verified and tested on basic examples.     

\medskip

In light of the above results and those in \cite{sambon}, the following corollary (which is not present in \cite{brie}) is essential to our results.  Indeed, the following corollary expresses that we may obtain effective results as in \eqref{effewachten} from any theorem of Nonstandard Analysis which has the same form as in \eqref{bog}.  It was shown in \cite{sambon, samzoo, samzooII} that the scope of this corollary includes the Big Five systems of Reverse Mathematics and the associated `zoo' (\cite{damirzoo}).  
\begin{cor}\label{consresultcor}
If $\Delta_{\intern}$ is a collection of internal formulas and $\psi$ is internal, and
\be\label{bog}
\P + \Delta_{\intern} \vdash (\forall^{\st}\underline{x})(\exists^{\st}\underline{y})\psi(\underline{x},\underline{y}, \underline{a}), 
\ee
then one can extract from the proof a sequence of closed$^{\ref{kootsie3}}$ terms $t$ in $\mathcal{T}^{*}$ such that
\be\label{effewachten}
\textup{\textsf{E-PA}}^{\omega*} + \Delta_{\intern} \vdash (\forall \underline{x})(\exists \underline{y}\in t(\underline{x}))\psi(\underline{x},\underline{y},\underline{a}).
\ee
\end{cor}
\begin{proof}
Clearly, if for internal $\psi$ and $\Phi(\underline{a})\equiv (\forall^{\st}\underline{x})(\exists^{\st}\underline{y})\psi(x, y, a)$, we have $[\Phi(\underline{a})]^{S_{\st}}\equiv \Phi(\underline{a})$, then the corollary follows immediately from the theorem.  
A tedious but straightforward verification using the clauses (i)-(v) in \cite{brie}*{Def.\ 7.1} establishes that indeed $\Phi(\underline{a})^{S_{\st}}\equiv \Phi(\underline{a})$.  
We undertake this verification in Section \ref{techni} below.  
\end{proof}
\subsection{The constructive system $\H$}\label{TATA}
In this section, we define the system $\H$, the constructive counterpart of $\P$. 
The system $\textsf{H}$ was first introduced in \cite{brie}*{\S5.2}, and constitutes a conservative extension of Heyting arithmetic $\textup{\textsf{E-HA}}^{\omega} $ by \cite{brie}*{Cor.\ 5.6}.
We now study the system $\H$ in more detail.  

\medskip

Similar to Definition \ref{debs}, we define $ \textsf{E-HA}^{\omega*}_{\st} $ as $ \textsf{E-HA}^{\omega{*}} + \T^{*}_{\st} + \textsf{IA}^{\st}$, where $\textsf{E-HA}^{\omega*}$ is just $\textsf{E-PA}^{\omega*}$ without the law of excluded middle.  
Furthermore, we define
\[
\H\equiv \textup{\textsf{E-HA}}^{\omega*}_{\st}+\HAC + {\I}+\NCR+\textsf{HIP}_{\forall^{\st}}+\textsf{HGMP}^{\st},
\]
where $\HAC$ is $\HAC_{\INT}$ without any restriction on the formula, and where the remaining axioms are defined in the following definition.
\bdefi[Three axioms of $\H$]\label{flah}~
\begin{enumerate}\rm
\item $\textsf{HIP}_{\forall^{\st}}$
\[
[(\forall^{\st}x)\phi(x)\di (\exists^{\st}y)\Psi(y)]\di (\exists^{\st}y')[(\forall^{\st}x)\phi(x)\di (\exists y\in y')\Psi(y)],
\]
where $\Psi(y)$ is any formula and $\phi(x)$ is an internal formula of \textsf{E-HA}$^{\omega*}$. 
\item $\textsf{HGMP}^{\st}$
\[
[(\forall^{\st}x)\phi(x)\di \psi] \di (\exists^{\st}x')[(\forall x\in x')\phi(x)\di \psi] 
\]
where $\phi(x)$ and $\psi$ are internal formulas in the language of \textsf{E-HA}$^{\omega*}$.
\item \textsf{NCR}
\[
(\forall y^{\tau})(\exists^{\st} x^{\rho} )\Phi(x, y) \di (\exists^{\st} x^{\rho^{*}})(\forall y^{\tau})(\exists x'\in x )\Phi(x', y),
\]
where $\Phi$ is any formula of \textsf{E-HA}$^{\omega*}$
\end{enumerate}
\edefi
Intuitively speaking, the first two axioms of Definition \ref{flah} allow us to perform a number of \emph{non-constructive operations} (namely \emph{Markov's principle} and \emph{independence of premises}) 
on the standard objects of the system $\H$, provided we introduce a `Herbrandisation' as in the consequent of $\HAC$, i.e.\ a finite list of possible witnesses rather than one single witness. 
Furthermore, while $\H$ includes idealisation \textsf{I}, one often uses the latter's \emph{classical contraposition}, explaining why \textsf{NCR} is useful (and even essential) in the context of intuitionistic logic.  

\medskip

Surprisingly, the axioms from Definition \ref{flah} are exactly what is needed to convert nonstandard definitions (of continuity, integrability, convergence, et cetera) into the normal form $(\forall^{\st}x)(\exists^{\st}y)\varphi(x, y)$ for internal $\varphi$, as is clear from Theorem~\ref{nogwelconsenal}.
This normal form plays an equally important role in the constructive case as in the classical case by the following theorem.  
\begin{thm}\label{consresult2}
If $\Delta_{\intern}$ is a collection of internal formulas, $\varphi$ is internal, and
\be\label{antecedn3}
\textup{\textsf{H}} + \Delta_{\intern} \vdash \forallst \tup x \, \existsst \tup y \, \varphi(\tup x, \tup y, \tup a), 
\ee
then one can extract from the proof a sequence of closed\footnote{Recall the definition of closed terms from \cite{brie} as sketched in Footnote \ref{kootsie}.\label{kootsie4}}  terms $t$ in $\mathcal{T}^{*}$ such that
\be\label{consequalty3}
\textup{\textsf{E-HA}}^{\omega*} + \Delta_{\intern} \vdash\  \forall \tup x \, \exists \tup y\in \tup t(\tup x)\ \varphi(\tup x,\tup y, \tup a).
\ee
\end{thm}
\begin{proof}
Immediate by \cite{brie}*{Theorem 5.9}.  
\end{proof}
The proofs of the soundness theorems in \cite{brie}*{\S5-7} provide an algorithm to obtain the term $t$ from the theorem.  
Finally, we point out one very useful principle to which we have access.  
\begin{thm}\label{doppi}
The systems $\P, \H$ prove \emph{overspill}, i.e.\
\be\tag{\textsf{OS}}
(\forall^{\st}x^{\rho})\varphi(x)\di (\exists y^{\rho})\big[\neg\st(y)\wedge \varphi(y)  \big],
\ee
for any internal formula $\varphi$.
\end{thm}
\begin{proof}
See \cite{brie}*{Prop.\ 3.3}.  
\end{proof}
In conclusion, we have introduced the systems $\H$ and $\P$ which are conservative extensions of Peano and Heyting arithmetic with fragments of Nelson's internal set theory.  
We have observed that central to the conservation results (Corollary~\ref{consresultcor} and Theorem~\ref{consresult}) is the normal form $(\forall^{\st}x)(\exists^{\st}y)\varphi(x, y)$ for internal $\varphi$.  
\subsection{Notations and definitions in $\H$ and $\P$}\label{leng}
In this section, we introduce notations and definitions relating to $\H$ and $\P$.  

\medskip

First of all, we mostly use the same notations as in \cite{brie}.  
\begin{rem}[Notations]\label{notawin}\rm
We write $(\forall^{\st}x^{\tau})\Phi(x^{\tau})$ and $(\exists^{\st}x^{\sigma})\Psi(x^{\sigma})$ as short for 
$(\forall x^{\tau})\big[\st(x^{\tau})\di \Phi(x^{\tau})\big]$ and $(\exists x^{\sigma})\big[\st(x^{\sigma})\wedge \Psi(x^{\sigma})\big]$.     
A formula $A$ is `internal' if it does not involve $\st$, and external otherwise.  The formula $A^{\st}$ is defined from $A$ by appending `st' to all quantifiers (except bounded number quantifiers).    
\end{rem}
Secondly, we will use the usual notations for natural, rational and real numbers and functions as introduced in \cite{kohlenbach2}*{p.\ 288-289}. (and \cite{simpson2}*{I.8.1} for the former).  
We only list the definition of real number and related notions in $\P$ and related systems.
\begin{defi}[Real numbers and related notions]\label{keepintireal}\rm~
\begin{enumerate}
\item Natural numbers correspond to type zero objects, and we use `$n^{0}$' and `$n\in \N$' interchangeably.  Rational numbers are defined in the usual way as quotients of natural numbers, and `$q\in \Q$' has its usual meaning.    
\item A (standard) real number $x$ is a (standard) fast-converging Cauchy sequence $q_{(\cdot)}^{1}$, i.e.\ $(\forall n^{0}, i^{0})(|q_{n}-q_{n+i})|<_{0} \frac{1}{2^{n}})$.  
We use Kohlenbach's `hat function' from \cite{kohlenbach2}*{p.\ 289} to guarantee that every sequence $f^{1}$ is a real.  
\item We write `$x\in \R$' to express that $x^{1}=(q^{1}_{(\cdot)})$ is a real as in the previous item and $[x](k):=q_{k}$ for the $k$-th approximation of $x$.    
\item Two reals $x, y$ represented by $q_{(\cdot)}$ and $r_{(\cdot)}$ are \emph{equal}, denoted $x=_{\R}y$, if $(\forall n^{0})(|q_{n}-r_{n}|\leq \frac{1}{2^{n}})$. Inequality $<_{\R}$ is defined similarly.         
\item We  write $x\approx y$ if $(\forall^{\st} n^{0})(|q_{n}-r_{n}|\leq \frac{1}{2^{n}})$ and $x\gg y$ if $x>y\wedge x\not\approx y$.  
\item Functions $F:\R\di \R$ mapping reals to reals are represented by functionals $\Phi^{1\di 1}$ mapping equal reals to equal reals, i.e. 
\be\tag{\textsf{RE}}\label{furg}
(\forall x, y)(x=_{\R}y\di \Phi(x)=_{\R}\Phi(y)).
\ee
\item Sets of objects of type $\rho$ are denoted $X^{\rho\di 0}, Y^{\rho\di 0}, Z^{\rho\di 0}, \dots$ and are given by their characteristic functions $f^{\rho\di 0}_{X}$, i.e.\ $(\forall x^{\rho})[x\in X\asa f_{X}(x)=_{0}1]$, where $f_{X}^{\rho\di 0}$ is assumed to output zero or one.  
\end{enumerate}
\end{defi}
Thirdly, the usual extensional notion of equality has nonstandard variations. 
\begin{rem}[Equality]\label{equ}\rm
The systems $\P$ and $\H$ include equality between natural numbers `$=_{0}$' as a primitive.  Equality `$=_{\tau}$' for type $\tau$-objects $x,y$ is defined as in \eqref{aparth}.
In the spirit of Nonstandard Analysis, we define `approximate equality $\approx_{\tau}$':
\be\label{aparth2}
[x\approx_{\tau}y] \equiv (\forall^{\st} z_{1}^{\tau_{1}}\dots z_{k}^{\tau_{k}})[xz_{1}\dots z_{k}=_{0}yz_{1}\dots z_{k}]
\ee
with the type $\tau$ as in \eqref{aparth}.  
Furthermore, our systems include the \emph{axiom of extensionality} \eqref{EXT}, but as noted in \cite{brie}*{p.\ 1973}, the so-called axiom of \emph{standard} extensionality \eqref{EXT}$^{\st}$ is problematic and cannot be included in $\P$ or $\H$.   
In particular, we cannot have term extraction as in Theorem \ref{TERM2} in the presence of \eqref{EXT}$^{\st}$.  

\medskip

Finally, we note that equality on the reals `$=_{\R}$' is a \emph{defined} notion and not an equality `$=_{\tau}$' of $\P$.  Hence, the first item of Definition \ref{debs} \textbf{does not} apply 
to `$=_{\R}$'.  In fact, $\P+(\forall x, y\in \R)((x=_{\R}y \wedge \st(x))\di \st(y))$ proves\footnote{Fix nonstandard $N^{0}$ and define $x=_{1} 0$ and $y:=(r_{(\cdot)})$ where $r_{k}=\frac{1}{2^{k+N}}$.  Then $x=_{\R}y$ and $\st_{1}(x)$ while $r_{0}$ clearly is nonstandard.} a contradiction.  
\end{rem}

\subsection{Technical results}\label{techni}
In this section, we have gathered the proofs of some of the above theorems too lengthy to fit the body of the text. 

\medskip

First of all, for completeness, we prove our `term extraction' result in detail.  
\begin{thm}\label{consresultcor11}
If $\Delta_{\intern}$ is a collection of internal formulas and $\psi$ is internal, and
\be\label{bog11}
\P + \Delta_{\intern} \vdash (\forall^{\st}\underline{x})(\exists^{\st}\underline{y})\psi(\underline{x},\underline{y}, \underline{a}), 
\ee
then one can extract from the proof a sequence of closed\footnote{Recall the definition of closed terms from \cite{brie} as sketched in Footnote \ref{kootsie}.} terms $t$ in $\mathcal{T}^{*}$ such that
\be\label{effewachten11}
\textup{\textsf{E-PA}}^{\omega*} +\QFAC^{1,0}+ \Delta_{\intern} \vdash (\forall \underline{x})(\exists \underline{y}\in t(\underline{x}))\psi(\underline{x},\underline{y},\underline{a}).
\ee
\end{thm}
\begin{proof}
Clearly, if for internal $\psi$ and $\Phi(\underline{a})\equiv (\forall^{\st}\underline{x})(\exists^{\st}\underline{y})\psi(x, y, a)$, we have $[\Phi(\underline{a})]^{S_{\st}}\equiv \Phi(\underline{a})$, then the corollary follows immediately from Theorem \ref{consresult}.  
A tedious but straightforward verification using the clauses (i)-(v) in \cite{brie}*{Def.\ 7.1} establishes that indeed $\Phi(\underline{a})^{S_{\st}}\equiv \Phi(\underline{a})$.  
For completeness, we now list these five inductive clauses and perform this verification.  

\medskip

Hence, suppose $\Phi(\underline{a})$ and $\Psi(\underline{b})$  in the language of $\P$ have the interpretations
\be\label{dombu}
\Phi(\underline{a})^{S_{\st}}\equiv (\forall^{\st}\underline{x})(\exists^{\st}\underline{y})\varphi(\underline{x},\underline{y},\underline{a}) \textup{ and } \Psi(\underline{b})^{S_{\st}}\equiv (\forall^{\st}\underline{u})(\exists^{\st}\underline{v})\psi(\underline{u},\underline{v},\underline{b}),
\ee
 for internal $\psi, \varphi$.  These formulas then behave as follows by \cite{brie}*{Def.\ 7.1}:
\begin{enumerate}[(i)]
\item $\psi^{S_{\st}}:=\psi$ for atomic internal $\psi$.  
\item$ \big(\st(z)\big)^{S_{\st}}:=(\exists^{\st}x)(z=x)$.
\item $(\neg \Phi)^{S_{\st}}:=(\forall^{\st} \underline{Y})(\exists^{\st}\underline{x})(\forall \underline{y}\in \underline{Y}[\underline{x}])\neg\varphi(\underline{x},\underline{y},\underline{a})$.  
\item$(\Phi\vee \Psi)^{S_{\st}}:=(\forall^{\st}\underline{x},\underline{u})(\exists^{\st}\underline{y}, \underline{v})[\varphi(\underline{x},\underline{y},\underline{a})\vee \psi(\underline{u},\underline{v},\underline{b})]$
\item $\big( (\forall z)\Phi \big)^{S_{\st}}:=(\forall^{\st}\underline{x})(\exists^{\st}\underline{y})(\forall z)(\exists \underline{y}'\in \underline{y})\varphi(\underline{x},\underline{y}',z)$
\end{enumerate}
Hence, fix $\Phi_{0}(\underline{a})\equiv(\forall^{\st}\underline{x})(\exists^{\st}\underline{y})\psi_{0}(\underline{x},\underline{y}, \underline{a})$ with internal $\psi_{0}$, and note that $\phi^{S_{\st}}\equiv\phi$ for any internal formula.  
We have $[\st(\underline{y})]^{S_{\st}}\equiv (\exists^{\st} \underline{w})(\underline{w}=\underline{y})$ and also 
\[
[\neg\st(\underline{y})]^{S_{\st}}\equiv (\forall^{\st} \underline{W} ) (\exists^{\st}\underline{x})(\forall \underline{w}\in \underline{W}[\underline{x}])\neg(\underline{w}=\underline{y})\equiv (\forall^{\st}\underline{w})(\underline{w}\ne \underline{y}).  
\]    
Hence, $[\neg\st(\underline{y})\vee\neg \psi_{0}(\underline{x}, \underline{y}, \underline{a})]^{S_{\st}}$ is just $(\forall^{\st}\underline{w})[(\underline{w}\ne \underline{y}) \vee \neg \psi_{0}(\underline{x}, \underline{y}, \underline{a})]$, and 
\[
\big[(\forall \underline{y})[\neg\st(\underline{y})\vee \neg\psi_{0}(\underline{x}, \underline{y}, \underline{a})]\big]^{S_{\st}}\equiv
 (\forall^{\st}\underline{w})(\exists^{\st}\underline{v})(\forall \underline{y})(\exists \underline{v}'\in \underline{v})[\underline{w}\ne\underline{y}\vee \neg\psi_{0}(\underline{x}, \underline{y}, \underline{a})].
\]
which is just $(\forall^{\st}\underline{w})(\forall \underline{y})[(\underline{w}\ne \underline{y}) \vee \neg\psi_{0}(\underline{x}, \underline{y}, \underline{a})]$.  Furthermore, we have
\begin{align*}
\big[(\exists^{\st}y)\psi_{0}(\underline{x}, \tup y, \tup a)\big]^{S_{\st}}&\equiv\big[\neg(\forall \underline{y})[\neg\st(\underline{y})\vee\neg \psi_{0}(\underline{x}, \underline{y}, \underline{a})]\big]^{S_{\st}}\\
&\equiv(\forall^{\st} \underline{V})(\exists^{\st}\underline{w})(\forall \underline{v}\in \underline{V}[\underline{w}])\neg[(\forall \underline{y})[(\underline{w}\ne \underline{y}) \vee \neg\psi_{0}(\underline{x}, \underline{y}, \underline{a})]].\\
&\equiv (\exists^{\st}\underline{w})(\exists \underline{y})[(\underline{w}= \underline{y}) \wedge \psi_{0}(\underline{x}, \underline{y}, \underline{a})]]\equiv (\exists^{\st}\underline{w})\psi_{0}(\underline{x}, \underline{w}, \underline{a}).
\end{align*}
Hence, we have proved so far that $(\exists^{\st}\underline{y})\psi_{0}(\underline{x}, \underline{y}, \underline{a})$ is invariant under $S_{\st}$.  By the previous, we also obtain:  
\[
\big[\neg \st(\underline{x})\vee (\exists^{\st}y)\psi_{0}(\underline{x}, \tup y, \tup a)\big]^{S_{\st}}\equiv  (\forall^{\st}\underline{w}')(\exists^{\st} \underline{w})[(\underline{w}'\ne \underline{x}) \vee \psi_{0}(\tup x, \tup w, \tup a)].
\]
Our final computation now yields the desired result: 
\begin{align*}
\big[(\forall^{\st} \underline{x})(\exists^{\st}y)\psi_{0}(\underline{x}, \tup y, \tup a)\big]^{S_{\st}}
&\equiv\big[(\forall \underline{x})(\neg \st(\underline{x})\vee (\exists^{\st}y)\psi_{0}(\underline{x}, \tup y, \tup a))\big]^{S_{\st}}\\
&\equiv(\forall^{\st}\underline{w}')(\exists^{\st} \underline{w})(\forall \underline{x})(\exists \underline{w}''\in \underline{w})[(\underline{w}'\ne \underline{x}) \vee \psi_{0}(\tup x, \tup w'', \tup a)].\\
&\equiv(\forall^{\st}\underline{w}')(\exists^{\st} \underline{w})(\exists \underline{w}''\in \underline{w}) \psi_{0}(\tup w', \tup w'', \tup a).
\end{align*}
The last step is obtained by taking $\underline{x}=\underline{w}'$.  Hence, we may conclude that the normal form $(\forall^{\st} \underline{x})(\exists^{\st}y)\psi_{0}(\underline{x}, \tup y, \tup a)$ is invariant under $S_{\st}$, and we are done.    
\end{proof}
Next, we prove two theorems which provide a normal form for $\STP$ and establishes the latter's relationship with the special and intuitionistic fan functional.  
The function $g^{1}$ from \eqref{krog} is called a \emph{standard part} of $f^{1}$.  
\begin{thm}\label{lapdog}
In $\P$, both $\STP$ and $\STP_{\R}$ are equivalent to either of the following:
\begin{align}\label{frukkklk}
(\forall^{\st}g^{2})(\exists^{\st}w)\big[(\forall T^{1}\leq_{1}1)&\big((\forall  \alpha^{1}  \in w(1))( \overline{\alpha}g(\alpha)\not\in T)\\
&\di(\forall \beta\leq_{1}1)(\exists i^{0}\leq w(2))(\overline{\beta}i\not\in T) \big)\big], \notag
\end{align}  
\be\label{krog}
(\forall f^{1})(\exists^{\st} g^{1})\big( (\forall^{\st}n^{0})(\exists^{\st}m^{0})(f(n)=m)\di   f\approx_{1}g\big).
\ee
Furthermore, $\P$ proves $(\exists^{\st}\Theta)\SCF(\Theta)\di \STP$.
\end{thm}
\begin{proof}  
First of all, since any individual real can be given a binary representation (See \cite{polahirst}), the equivalence between $\STP$ and $\STP_{\R}$ is immediate.
Next, $\STP$ is easily seen to be equivalent to 
\begin{align}\label{fanns}
(\forall T^{1}\leq_{1}1)\big[(\forall^{\st}n)(\exists \beta^{0})&(|\beta|=n \wedge \beta\in T )\notag\\
& \di (\exists^{\st}\alpha^{1}\leq_{1}1)(\forall^{\st}n^{0})(\overline{\alpha}n\in T)   \big],
\end{align}
and this equivalence may also be found in \cite{samGH}*{Theorem 3.2}.  For completeness, we first prove the equivalence $\STP\asa \eqref{fanns}$.
Assume $\STP$ and apply overspill to $(\forall^{\st}n)(\exists \beta^{0})(|\beta|=n \wedge \beta\in T )$ to obtain $\beta_{0}^{0}\in T$ with nonstandard length $|\beta_{0}|$.  
Now apply $\STP$ to $\beta^{1}:=\beta_{0}*00\dots$ to obtain a \emph{standard} $\alpha^{1}\leq_{1}1$ such that $\alpha\approx_{1}\beta$ and hence $(\forall^{\st}n)(\overline{\alpha}n\in T)$.  
For the reverse direction, let $f^{1}$ be a binary sequence, and define a binary tree $T_{f}$ which contains all initial segments of $f$.  
Now apply \eqref{fanns} for $T=T_{f}$ to obtain $\STP$.    

\medskip

For the implication \eqref{frukkklk}$\di$\eqref{fanns}, note that \eqref{frukkklk} implies for all standard $g^{2}$
\begin{align}\label{frukkklk2}
(\forall T^{1}\leq_{1}1)(\exists^{\st} ( \alpha^{1}\leq_{1}1,  &~k^{0})\big[(\overline{\alpha}g(\alpha)\not\in T)
\di(\forall \beta\leq_{1}1)(\exists i\leq k)(\overline{\beta}i\not\in T) \big], 
\end{align}  
which in turn yields, by bringing all standard quantifiers inside again, that:
\begin{align}\label{frukkklk3}
(\forall T\leq_{1}1) \big[(\exists^{\st}g^{2})(\forall^{\st}\alpha \leq_{1}1)(\overline{\alpha}g(\alpha)\not\in T)\di(\exists^{\st}k)(\forall \beta\leq_{1}1)(\overline{\beta}k\not\in T) \big], 
\end{align}  
To obtain \eqref{fanns} from \eqref{frukkklk3}, apply $\HAC_{\INT}$ to $(\forall^{\st}\alpha^{1}\leq_{1}1)(\exists^{\st}n)(\overline{\alpha}n\not\in T)$ to obtain standard $\Psi^{1\di 0^{*}}$ with  
$(\forall^{\st}\alpha^{1}\leq_{1}1)(\exists n\in \Psi(\alpha))(\overline{\alpha}n\not\in T)$, and defining $g(\alpha):=\max_{i<|\Psi|}\Psi(\alpha)(i)$ we obtain $g$ as in the antecedent of \eqref{frukkklk3}.  We obtain: 
\be\label{gundark}
(\forall T^{1}\leq_{1}1) \big[(\forall^{\st}\alpha^{1}\leq_{1}1)(\exists^{\st}n)(\overline{\alpha}n\not\in T)\di (\exists^{\st}k)(\forall \beta\leq_{1}1)(\overline{\beta}i\not\in T) \big], 
\ee
which is the contraposition of \eqref{fanns}, using classical logic.  For the implication $\eqref{fanns}  \di \eqref{frukkklk}$, consider the contraposition of \eqref{fanns}, i.e.\ \eqref{gundark}, and note that the latter implies \eqref{frukkklk3}.  Now push all standard quantifiers outside as follows:
\[
(\forall^{\st}g^{2})(\forall T^{1}\leq_{1}1)(\exists^{\st} ( \alpha^{1}\leq_{1}1, ~k^{0})\big[(\overline{\alpha}g(\alpha)\not\in T)
\di(\forall \beta\leq_{1}1)(\exists i\leq k)(\overline{\beta}i\not\in T) \big], 
\]
and applying idealisation \textsf{I} yields \eqref{frukkklk} by taking the maximum of all elements pertaining to $k$.  
The equivalence involving \eqref{frukkklk} also immediately establishes the final part of the theorem.    

\medskip

For the remaining equivalence, the implication $\eqref{krog}\di \ref{STP}$ is trivial, and for the reverse implication, fix $f^{1}$ such that 
$(\forall^{\st}n)(\exists^{\st}m)f(n)=m$ and let $h^{1}$ be such that $(\forall n,m)(f(n)=m\asa h(n,m)=1)$.  
Applying $\HAC_{\INT}$ to the former formula, there is standard $\Phi^{0\di 0^{*}}$ such that $(\forall^{\st}n)(\exists m\in \Phi(n))f(n)=m$, and define 
$\Psi(n):=\max_{i<|\Phi(n)|}\Phi(n)(i)$.  Now define the sequence $\alpha_{0}\leq_{1}1$ as follows:  $\alpha_{0}(0):= h(0,0)$, $\alpha_{0}(1):=h(0, 1)$, \dots, $\alpha_{0}(\Psi(0)):=h(0,\Psi(0))$, $\alpha_{0}(\Psi(0)+1):= h(1, 0)$, $\alpha_{0}(\Psi(0)+2):= h(1, 1)$, \dots, $\alpha_{0}(\Psi(0)+\Psi(1)):= h(1, \Psi(1))$, et cetera.  
Now let $\beta_{0}^{1}\leq_{1}1$ be the standard part of $\alpha_{0}$ provided by $\ref{STP}$ and define $g(n):=(\mu m\leq \Psi(n))\big[\beta_{0}(\sum_{i=0}^{n-1}\Psi(i)+m)=1\big]$.  By definition, $g^{1}$ is standard and $f\approx_{1} g$.     
\end{proof}
\begin{thm}\label{foora}
The axiom \ref{STP} can be proved in $\P$ plus the axiom
\be\label{kunta}\tag{\textsf{\textup{NUC}}}
(\forall^{\st}Y^{2})(\forall f^{1}, g^{1}\leq_{1}1)(f\approx_{1} g\di Y(f)=_{0}Y(g)).  
\ee
\end{thm}
\begin{proof}
Resolving `$\approx_{1}$', $\NUC$ implies that
\be\label{kunt2}
(\forall^{\st}Y^{2})(\forall f^{1}, g^{1}\leq_{1}1)(\exists^{\st}N^{0})(\overline{f}N=_{0} \overline{g}N\di Y(f)=_{0}Y(g)).  
\ee
Applying \emph{Idealisation} \textsf{I} to \eqref{kunt2}, we obtain that 
\be\label{kunt32}
(\forall^{\st}Y^{2})(\exists^{\st}x^{0^{*}})(\forall f^{1}, g^{1}\leq_{1}1)(\exists N^{0}\in x)(\overline{f}N=_{0} \overline{g}N\di Y(f)=_{0}Y(g)).  
\ee
which immediately yields that
\be\label{kunt3}
(\forall^{\st}Y^{2})(\exists^{\st}N_{0}^{0})(\forall f^{1}, g^{1}\leq_{1}1)(\overline{f}N_{0}=_{0} \overline{g}N_{0}\di Y(f)=_{0}Y(g)),   
\ee
by taking $N_{0}$ in \eqref{kunt3} to be $\max_{i<|x|}x(i)$ for $x$ as in \eqref{kunt32}.  
Now that $\overline{f}N_{0}*00\dots$ is standard for \emph{any} $f^{1}$ and standard $N_{0}$ by the basic axioms of $\P$ (See Definition~\ref{debs}).
Hence, $Y(\overline{f}N_{0}*00\dots)$ is also standard and \eqref{kunt3} becomes 
\be\label{altje}
(\forall^{\st}Y^{2})(\forall f^{1}\leq_{1}1)(\exists^{\st}M_{0}^{0})(Y(f)\leq M_{0}), 
\ee
if we take $M_{0}=Y(\overline{f}N_{0}*00\dots)$.  Applying \emph{Idealisation} to \eqref{altje}, we obtain    
\be\label{altje2}
(\forall^{\st}Y^{2})(\exists^{\st}y^{0^{*}})(\forall f^{1}\leq_{1}1)(\exists M_{0}^{0}\in y)(Y(f)\leq M_{0}), 
\ee
and defining $N$ as the maximum of all elements in $y$ in \eqref{altje2}, we obtain
\be\label{altje3}
(\forall^{\st}Y^{2})(\exists^{\st}N^{0})(\forall f^{1}\leq_{1}1)(Y(f)\leq N), 
\ee
Now fix some standard $g^{2}$ in \eqref{frukkklk} and let $N_{1}$ be its standard upper bound on Cantor space from \eqref{altje3}.  
Define the required (standard) $w$ as follows: $w(2)$ is $N_{1}$ and $w(1)$ is the finite sequence consisting of all binary sequences $\alpha_{i}^{1}=\sigma_{i}^{0^{*}}*00\dots$ for $i\leq 2^{N_{1}}$ and $|\sigma_{i}|=N_{1}$.  Then $\STP$ follows from Theorem \ref{lapdog}.          
\end{proof}

\begin{cor}\label{scrufa}
From the proof in $\P$ that $\NUC\di \STP$, a term $t^{3\di3}$ can be extracted such that 
$\textup{\textsf{E-PA}}^{\omega*}+\QFAC^{1,0}$ proves $(\forall \Omega^{3})\big[\MUC(\Omega)\di \SCF(t(\Omega))]$.  
\end{cor}
\begin{proof}
The proof amounts to nothing more than applying Remark \ref{doeisnormaal}.  
Indeed, by the proof of the theorem, $\NUC$ is equivalent to the normal form \eqref{kunt3}, while \eqref{frukkklk} is a normal form for $\STP$ 
and we abbreviate the latter normal form as $(\forall^{\st}g^{2})(\exists^{\st}w^{1^{*}})B(g, w)$.  
Hence, $\ref{kunt}\di \STP$ becomes 
$(\forall^{\st}Y^{2})(\exists^{\st}N^{0})A(Y,N)\di (\forall^{\st}g^{2})(\exists^{\st}w^{1^{*}})B(g, w)$, which yields
\be\label{hoer}
\big[(\exists^{\st}\Omega^{3})(\forall Y^{2})A(Y,\Omega(Y))\di (\forall^{\st}g^{2})(\exists^{\st}w^{1^{*}})B(g, w)\big], 
\ee
by strengthening the antecedent.  Bringing all standard quantifiers up front:
\be\label{calvarie}
(\forall^{\st}\Omega^{3}, g^{2})(\exists^{\st}w^{1^{*}})\big[(\forall Y^{2})A(Y,\Omega(Y))\di B(g, w)\big];
\ee
Applying Corollary \ref{consresultcor} to `$\P\vdash \eqref{calvarie}$', we obtain a term $t^{3\di3}$ such that 
\be\label{calvarie2}
(\forall \Omega^{3}, g^{2})(\exists w\in t(\Omega, g))\big[(\forall Y^{2})A(Y,\Omega(Y))\di B(g, w)\big] 
\ee
is provable in $\textsf{E-PA}^{\omega*}+\QFAC^{1,0}$.  
Bringing all quantifiers inside again, \eqref{calvarie2} yields
\be\label{calvarie3}
(\forall \Omega^{3})\big[(\forall Y^{2})A(Y,\Omega(Y))\di  (\forall g^{2})(\exists w\in t(\Omega, g))B(g, w)\big], 
\ee
Clearly, the antecedent of \eqref{calvarie3} expresses that $\Omega^{3}$ is the fan functional.  To define a functional $\Theta$ as in $\SCF(\Theta)$ from $t(\Omega, g)$, note that the latter is a finite sequence of numbers and binary sequences by \eqref{frukkklk};  Using basic sequence coding, we may assume that $t(\Omega, g)=t_{0}(\Omega, g)*t_{1}(\Omega, g)$, where the first (resp.\ second) part contains the binary sequences (resp.\ numbers).  Now define $\Theta( g)(1)$ as $\max_{i<|t_{1}(\Omega, g)|}t_{1}(\Omega, g)(i)$ and $\Theta( g)(2):=t_{0}(\Omega, g)$, and note that $\Theta$ indeed satisfies $\SCF(\Theta)$.   
\end{proof}
Finally, we prove Theorem \ref{kkringe} from the final section of this paper.  
\begin{thm}\label{finake}
From the proof in $\P$ that $\STP\di \DNR_{\ns}$, a term $t^{3\di3}$ can be extracted such that 
$\textup{\textsf{E-PA}}^{\omega}$ proves $(\forall \Theta^{3})\big[\SCF(\Theta)\di \DNR(t(\Theta))]$.  
\end{thm}
\begin{proof}
By Definition \ref{kefi}, $\STP$ proves $\WKL^{\st}$, which implies $\DNR^{\st}$ as follows: 
\[
(\forall^{\st} A^{1}\leq_{1}1)(\exists^{\st} f^{1})(\forall^{\st} m^{0}, s^{0}, e^{0})(\varphi_{e, s}^{A}(e)=m \di f(e)\ne m),
\]
which immediately yields the following equivalent `metastable' version:
\be\label{derfg}
(\forall^{\st} A^{1}\leq_{1}1, G^{2})(\exists^{\st} f^{1})(\forall m^{0}, s^{0}, e^{0}\leq G(f))(\varphi_{e, s}^{A}(e)=m \di f(e)\ne m),
\ee
Since $A^{1}$ only occurs as `call by standard value' in \eqref{derfg}, and since for every $B\leq_{1}1$ there is standard $A\approx_{1}B$ by $\STP$, \eqref{derfg} becomes
\[
(\forall^{\st} G^{2})(\forall A^{1}\leq_{1}1)(\exists^{\st} f^{1})(\forall m^{0}, s^{0}, e^{0}\leq G(f))(\varphi_{e, s}^{A}(e)=m \di f(e)\ne m),
\]
and applying \emph{Idealisation} to the previous formula, we obtain $\DNR_{\ns}$.
Since both $\DNR_{\ns}$ and $\STP$ have normal forms, 
we may apply term extraction as in Theorem~\ref{TERM2} to the proof `$\P\vdash \STP\di \DNR_{\ns}$', and obtain a term which computes $\zeta$ in terms of the special fan functional $\Theta$.  
\end{proof}

\bye